\documentclass{article}
\usepackage[english]{babel}

\usepackage[letterpaper,top=2cm,bottom=2cm,left=2cm,right=2cm,marginparwidth=1.75cm]{geometry}
\usepackage{calrsfs}
\usepackage{amssymb}
\usepackage{amsmath}
\usepackage{amsthm}
\usepackage{subcaption}
\usepackage{multirow}
\usepackage{graphicx}
\usepackage{amsfonts}
\usepackage{mathtools}
\usepackage{mathrsfs}
\usepackage{placeins}

\usepackage{booktabs}
\newtheorem{remark}{Remark}

\newtheorem{assumption}{Assumption}
\newtheorem{theorem}{Theorem}
\newtheorem{proposition}{Proposition}

\usepackage[mathscr]{euscript}
\usepackage[dvipsnames]{xcolor}
\usepackage[colorlinks=true, allcolors=OliveGreen]{hyperref}
\usepackage{filecontents}
\usepackage{tikz}
\usepackage{pgfplots}
\usetikzlibrary{external}
\usepackage{authblk}
\usepackage{todonotes}
\usepackage{array}
\newcolumntype{C}{>{\centering\arraybackslash}p{2.5cm}}

\usepackage{caption}
\usepackage{subcaption}
\usepackage{tikz}
\usepackage{csvsimple}
\usepackage{siunitx}
\usepackage{tikz} % To generate the plot from csv
\usepackage{pgfplots}
\pgfplotsset{compat=newest} % Allows to place the legend below plot
\usepgfplotslibrary{units} % Allows to enter the units nicely
\newcommand{\averagel}{\{\!\!\{}
\newcommand{\averager}{\}\!\!\}}
\newcommand{\tnorm}{|\!\!\:|\!\!\:|}
\newcommand{\jumpl}{[\![}
\newcommand{\jumpr}{]\!]}

\newcommand{\partition}{\mathcal{T}_h}

\newcommand{\faces}{\mathcal{F}_h}

\DeclareMathAlphabet{\mathcalligra}{T1}{calligra}{m}{n}
\graphicspath{{Images/}}

\tikzset{  font={\fontsize{15pt}{12}\selectfont}}

\title{A high-order discontinuous Galerkin method for the numerical modeling of epileptic seizures\footnote{\textbf{Funding}: PFA has been partially funded by the research grants PRIN2020 n. 20204LN5N5 funded by MUR. PFA has been partially supported by ICSC—Centro Nazionale di Ricerca in High Performance Computing, BigData, and Quantum Computing funded by European Union—NextGeneration EU. The present research has been supported by MUR, grant Dipartimento di Eccellenza 2023-2027. PFA, SP and MC are members of INdAM-GNCS. PFA and MC are partially funded by the European Union (ERC SyG, NEMESIS, project number 101115663). Views and opinions expressed are however those of the authors only and do not necessarily reflect those of the European Union or the European Research Council Executive Agency. Neither the European Union nor the granting authority can be held responsible for them. CBLS's Ph.D. activity has been funded by the National Recovery and Resilience Plan (NRRP), Mission 4, Component 1 – Investment 3.4 and Investment 4.1 funded by the European Union.}}

\author[1]{Caterina B. Leimer Saglio}

\affil[1]{MOX-Dipartimento di Matematica, Politecnico di Milano, Piazza Leonardo da Vinci 32, Milan, 20133, Italy}

\author[1]{Stefano Pagani}
\author[1]{Mattia Corti}
\author[1]{Paola F. Antonietti}
\begin{document}
\maketitle

\begin{abstract}
Epilepsy is a clinical neurological disorder characterized by recurrent and spontaneous seizures consisting of abnormal high-frequency electrical activity in the brain.
In this condition, the transmembrane potential dynamics are characterized by rapid and sharp wavefronts traveling along the heterogeneous and anisotropic conduction pathways of the brain. 
This work employs the monodomain model, coupled with specific neuronal ionic models characterizing ion concentration dynamics, to mathematically describe brain tissue electrophysiology in grey and white matter at the organ scale. This multiscale model is discretized in space with the high-order discontinuous Galerkin method on polygonal and polyhedral grids (PolyDG) and advanced in time with a Crank-Nicolson scheme. This ensures, on the one hand, efficient and accurate simulations of the high-frequency electrical activity that is responsible for epileptic seizure and, on the other hand, keeps reasonably low the computational costs by a suitable combination of high-order approximations and agglomerated polytopal meshes.  
We numerically investigate synthetic test cases on a two-dimensional heterogeneous squared domain discretized with a polygonal grid, and on a two-dimensional brainstem in a sagittal plane with an agglomerated polygonal grid that takes full advantage of the flexibility of the PolyDG approximation of the semidiscrete formulation. Finally, we provide a theoretical analysis of stability and an \textit{a-priori} convergence analysis for a simplified mathematical problem.
\end{abstract} 

\section{Introduction}\label{sec:1}
Epilepsy is a neurological disorder characterized by seizures, an abnormal high-frequency electrical activity in the brain. This pathology can significantly reduce the quality of life to the point of life-threatening \cite{beghi2020epidemiology}. Clinical investigations into the causes, mechanisms, and treatments of epilepsy have increased in recent years \cite{da2003epilepsies,breakspear2017dynamic,rizzone2001deep}, and have taken advantage of recent developments in the mathematical modeling of epileptic activity both at neuronal \cite{kager2007seizure,cressman2009influence,somjen2008computer,stefanescu2012computational} and at organ level \cite{schreiner2022simulating,erhardt2020dynamics,lopez2020function}. Recent results \cite{schreiner2022simulating,erhardt2020dynamics} have shown that the bidomain model, typically employed in cardiac electrophysiology \cite{henriquez1993simulating}, can be used to simulate electrical activity in the brain, provided that it is supplemented with suitable ionic models that describe the variations in neuron ion concentrations across the membrane \cite{cressman2009influence, hodgkin1990quantitative, tsumoto2006bifurcations, fitzhugh1961impulses}. More specifically, ion models are exploited for neuronal modeling to describe the different ion concentrations' dynamics that influence the cellular action potential \cite{schwartz2016analytic,rincon2016inverse}. Although modeling the cardiac and neuronal contexts can be done by exploiting the same organ-level equations, there is a substantial difference between the two: neural conduction velocity is approximately 20-60 times faster than the cardiac one. This makes numerical simulation of the underlying physical process evermore challenging \cite{erhardt2020dynamics,hoermann2018adaptive}. 
\par
In modeling electrical activity within the human brain, suitable specific peculiarities of brain structure have to be accounted for \cite{brodal2004central}: different conductivity tensors must be incorporated in the mathematical model to correctly describe different brain tissues, such as grey and white matter. Both tissues are vital components of the brain and spinal cord, although there are substantial differences at both the structural and functional levels \cite{lopez2020function,szmurlo2006bidomain}. Since white matter is composed of axonal fibers covered by myelin and not directly by neuronal bodies like grey matter, it is highly anisotropic, with conductivity tensors following specific preferential directions depending on the local direction of the axonal connections \cite{brodal2004central}. This paper considers the Barreto-Cressman ionic model \cite{cressman2009influence} coupled with the monodomain model to describe electric organ-level activity \cite{franzone2005simulating}. The Barreto-Cressman model is a conductance-based model describing the evolution of the neuron's transmembrane potential and the dynamics of different ion concentrations (calcium, potassium, and sodium) that, through the opening and closing of ion channels govern the action potential mechanism \cite{ullah2009influence,alonso2018models}. The coupling with the monodomain equation upscales the transmembrane potential in three dimensions, generating spatial propagation patterns. In this way, it is possible to study the multiscale effects on the evolution of the transmembrane potential both from a microscopic (single neuron) \cite{jaeger2021derivation} and macroscopic perspective (brain tissue) \cite{foutz2022brain}. 
\par
Standard discretization techniques, such as the finite element method, involve high computational costs to simulate the discretized system on realistic brain geometries. Indeed, modeling wave propagation with the brain tissue is characterized by very complicated geometrical details, steep wavefronts, and high heterogeneity of the materials, requiring extremely fine spatio-temporal discretizations to ensure sufficient accuracy and a correct description of the wavefront. To overcome the inherent geometrical and functional complexities, we propose, for the space discretization of the model, a high-order discontinuous Galerkin method \cite{arnold2002unified} on polygonal and polyhedral grids (PolyDG) \cite{cangiani2014hp,antonietti2013hp,bassi2012flexibility}. This choice is motivated by the following: (i) high-order approximations have demonstrated remarkable performance in accurately approximating wavefront propagation phenomena, see, e.g., \cite{antonietti2021high};
(ii) in the simulation of brain physiology and pathology, the geometric flexibility of PolyDG can be successfully exploited to tame computational complexity, reduce the computational cost,  control dispersion and dissipation in the solution \cite{corti2023numerical,CortiFK};
(iii) the construction of the meshes in this context is facilitated by agglomeration strategies that allow flexibility and the preservation of high quality in the description of complex boundaries \cite{antonietti2022agglomeration}. 
\\
The aim of this work is to introduce, theoretically analyze, and test in practice a PolyDG method for the numerical approximation of brain electrophysiology, especially in pathological scenarios characterized by high-frequency impulses that generate rapid wavefronts traveling along the white and gray matter.  
Numerical test cases are conducted using a specific parametrization of the Barreto-Cressman ionic model to simulate epileptic seizures both on simplified geometries and in brain sections reconstructed from medical images. In both cases, we encode the distinction between white and grey matter into the conductivity tensor. In these scenarios, we present a quantitative analysis of wavefront conduction velocities on varying mesh refinement and polynomial degree. Finally, a simplified ionic model is exploited to carry out a theoretical stability analysis and to prove \textit{a-priori}  convergence error estimate. 
\\

The paper is organized as follows. In Section \ref{sec:2}, we introduce the mathematical model of the monodomain equation coupled with the Barreto-Cressman ionic model and its weak formulation. Section \ref{sec:3} presents the semi-discrete formulation for the coupled problem. In Section \ref{sec:4}, we introduce the time-discretization exploiting a second-order semi-implicit method. In Section \ref{sec:5}, we present two-dimensional simulations, including white and grey matters and unstable initial conditions. We also numerically analyze the evolution of the transmembrane potential in a realistic setting given by the brainstem section reconstructed from medical images from \cite{lamontagne2019oasis}. Section \ref{sec:6} provides a theoretical analysis for a simplified monodomain problem presenting stability results of the semi-discretized problem and an a priori error estimate proof. Finally, in Section \ref{sec:9}, we present the results of the convergence tests with an analytical solution.

\section{The mathematical model}
\label{sec:2}
\noindent In this section, we present the mathematical formulation of the monodomain model coupled with a general ionic model posed in abstract form. This model describes the evolution of the transmembrane potential within brain tissue under pathophysiological conditions \cite{schreiner2022simulating}. Given an open, bounded, polygonal domain $\Omega \in \mathbb{R}^d$, $(d=2,3)$ and a final time $T>0$, we introduce the transmembrane potential $u = u(\boldsymbol{x},t)$ with $u: \Omega \times [0,T] \rightarrow \mathbb{R}$, and the vector $\boldsymbol{y} = \boldsymbol{y}(\boldsymbol{x},t)$ with $\boldsymbol{y}: \Omega \times [0,T] \rightarrow \mathbb{R}^n, n\ge1,$ containing the ion concentrations and gating variables of the ionic model. 
\noindent The coupled problems reads as follows: For any time $ t \in (0,T]$, find $u=u(\boldsymbol{x},t)$ and $\boldsymbol{y}=\boldsymbol{y}(\boldsymbol{x},t)$ such that:
\begin{equation}
    \label{eq:monodomain}
    \begin{dcases}
        \chi_m C_m  \frac{\partial u}{\partial t} - \nabla \cdot (\mathbf{\Sigma} \nabla u) +  f(u,\boldsymbol{y}) = I^\mathrm{ext} & \mathrm{in} \; \Omega \times (0,T], \\
        \frac{\partial \boldsymbol{y}}{\partial t} + \boldsymbol{m}(u,\boldsymbol{y}) = \boldsymbol{0} &\mathrm{in} \; \Omega \times (0,T], \\
        \mathbf{\Sigma} \nabla u \cdot \boldsymbol{n} = 0  & \mathrm{on}\; \partial \Omega  \times (0,T],\\
        u(0) = u^0, \; \boldsymbol{y}(0) = \boldsymbol{y}^0 &\mathrm{in}\; \Omega.
    \end{dcases}
\end{equation}
In Equation~\eqref{eq:monodomain}, $f(u,\boldsymbol{y})$ and $\boldsymbol{m}(u,\boldsymbol{y})$ represent the ionic forces and the evolution of the $n$ different ion concentrations appearing in the ionic model, respectively. These functions may be characterized by one or more variables representing the concentration of certain ions, such as calcium and potassium, or gating variables relating to the opening and closing of ion channels in the cell membrane \cite{schwartz2016analytic,cressman2009influence}. We assume homogeneous Neumann boundary conditions and $\boldsymbol{n} = \boldsymbol{n}(\boldsymbol{x})$ is the outward normal vector to $\partial \Omega$. Finally, we introduce the initial conditions where $u^0$ and $\boldsymbol{y}^0$. In Equation \eqref{eq:monodomain} $\boldsymbol{\Sigma}$ represents the conductivity tensor, we assume to be constant in time and piecewise in space, characterized by its tangential component ($\sigma_t$) and its normal component ($\sigma_n$):
\begin{equation}
\label{eq:sigma}
      \boldsymbol{\Sigma} = \begin{bmatrix} \sigma_t & 0 \\ 0 & \sigma_n \end{bmatrix}.
\end{equation}
Specifically, we employ a fully isotropic conductivity in grey matter and an anisotropic conductivity in white matter \cite{schreiner2022simulating}. Where $I_\mathrm{ion}$ is the forcing ionic current, $\chi_m$ is the membrane surface area per unit volume, and $C_m$ is the membrane capacitance.
For the weak formulation of the problem, we consider the Sobolev space $V= H^1(\Omega)$, and we employ a standard definition of the scalar product in $L^2(\Omega)$, denoted by $(\cdot,\cdot)_{\Omega}$. The induced norm is denoted by $\|\cdot\|$. We remind that for vector-valued and tensor-valued functions, the definition extends componentwise \cite{salsa2022partial}.
\noindent Starting from Equation \eqref{eq:monodomain}, we introduce the ionic component and the dynamics component of the ionic model as:
\begin{equation}
    \begin{aligned}
       a(u,v) &= (\boldsymbol{\Sigma} \nabla u,\nabla v)_{\Omega}, & 
       r_{\text{ion}}(u,\boldsymbol{y},v) &= (f(u,\boldsymbol{y}),v)_{\Omega} &\forall \:  v \in V, \quad
       r_{\text{m}}(u,\boldsymbol{y},\boldsymbol{w}) &=  (\boldsymbol{m}(u,\boldsymbol{y}),\boldsymbol{w})_{\Omega} &\forall \:  \boldsymbol{w} \in V^n.
    \end{aligned}
    \label{eq:barreto_weak}
\end{equation}
We assume that the forcing terms, physical parameters, and initial conditions in Equation \eqref{eq:monodomain} are sufficiently regular, $i.e.$: $f(u,\boldsymbol{y}) \in L^2(0,T;L^2(\Omega))$, $I^\mathrm{ext}(u,\boldsymbol{y}) \in L^2(0,T;L^2(\Omega))$, $\boldsymbol{m}(u,\boldsymbol{y}) \in [L^2(0,T;L^2(\Omega))]^n$, $\chi_m$ and $C_m \in L_+^\infty(\Omega)$ where $L^{\infty}_+(\Omega):=\{v \in L^\infty(\Omega): v \ge 0 \text{ a.e. in } \Omega\}$, $u^0\in L^2(\Omega)$ and $\boldsymbol{y}^0\in [L^2(\Omega)]^n$.\\

\noindent The weak formulation of the problem \eqref{eq:monodomain} reads:
$ \forall \: t \in (0,T] $ find $u(t) \in V ,\boldsymbol{y}(t) \in V^n$ such that:
\begin{equation}
 \left\{ 
    \begin{aligned}
        &\chi_m C_m  \left(\frac{\partial u(t)}{\partial t},v\right)_{\Omega} + a(u(t),v) + \mathrm{\chi_m} r_{\text{ion}}(u(t),v) = (I^\text{ext},v)_\Omega  &\forall \:  v \in V, \\
        & \left(\frac{\partial \boldsymbol{y}(t)}{\partial t},\boldsymbol{w}\right)_{\Omega} + r_{\text{m}}(u(t),\boldsymbol{y}(t),\boldsymbol{w}) = 0  &\forall \:  \boldsymbol{w} \in V^n, \\
        &u(0) = u^0, \; \boldsymbol{y}(0) = \boldsymbol{y}^0 &\text{in } \Omega .
    \end{aligned}
    \right.
    \label{eq::weak_general}
\end{equation}

\subsection{Barreto-Cressman ionic model}
\label{sec::21}
\noindent The Barreto-Cressman ionic model \cite{cressman2009influence} can be defined as a conductance-based model that represents the membrane potential of a neuron and the dynamic interactions of intra- and extra-cellular ion concentrations~\cite{barreto2011ion}. It can accurately display neuronal burstings, which is the epilepsy phase characterized by the fast-spiking behavior of the transmembrane potential~\cite{ullah2009influence,barreto2011ion}, as well as neuronal quiescence, which is instead distinguishable by a slower oscillatory behavior of the quantities of interest. The complete model derivation can be found in \cite{cressman2009influence}. 
This ionic conductance-based model is characterized by three different ionic concentrations: intracellular sodium $s(t) =[\mathrm{Na}]_i(t)$, extracellular potassium $k(t) = [\mathrm{K}]_o(t)$, and intracellular calcium $c(t) = [\mathrm{Ca}]_i(t)$. 
\par
The Barreto-Cressman model consists of three different ionic variables that characterize the model and the three gating variables $g^s,g^k,g^c$ that control the opening and closing of the corresponding ion channels, $i.e.$
\begin{equation}
\label{eq:yBC}
    \begin{aligned}
      \boldsymbol{y}(\boldsymbol{x},t) = 
      &\left[
    \begin{aligned}
       c(\boldsymbol{x},t), \:
       k(\boldsymbol{x},t),\:
       s(\boldsymbol{x},t), \:
       g^s(\boldsymbol{x},t), \:
       g^k(\boldsymbol{x},t), \:
       g^c(\boldsymbol{x},t) \:
    \end{aligned}
    \right]^\top, \;\text{ with } \; \boldsymbol{y}(\boldsymbol{x},t): \Omega\times [0,T] \rightarrow \mathbb{R}^6.
\end{aligned}
\end{equation}

The ionic current characterizing the dynamics of the transmembrane potential is characterized by different contributions arising from the sodium, potassium, and chlorine current, defined and analyzed more specifically below. Under specific parametrizations and initial conditions, the Barreto-Cressman ionic model shows that a single cell subject to intra and extracellular ion concentration dynamics can exhibit recurrent seizure-like events. At the tissue level, the coupling of the monodomain model with the conductance-based ion model is exploited to differentiate individual epileptic seizures from a condition of abnormal brain activity. We then introduce the ionic forcing term of the model in Equation \eqref{eq:monodomain} as follows.
\begin{equation}
    \label{eq:FBC}
        f(u,\boldsymbol{y}) = I_\mathrm{ion}(u,\boldsymbol{y}) = I_\mathrm{Na}(u,\boldsymbol{y}) + I_\mathrm{K}(u,\boldsymbol{y}) + I_\mathrm{Cl}(u,\boldsymbol{y}). 
\end{equation}
The ionic current $I_\mathrm{ion}$ depends directly on the three ionic variables $s$,$c$, $k$, and the three gating variables. $\mathrm{\chi_m}$ is the membrane capacitance per unit area, and $I_\mathrm{ion}$ is the transmembrane ionic current per unit area. In addition, four other types of currents are introduced: the current related to the ability of cells to remove excess potassium from the extracellular space $I_\mathrm{Glia}$, the current representing potassium diffusion $I_\mathrm{diff}$, the current that denotes the sodium–potassium pump $I_\mathrm{pump}$, and the external forcing term $I_\mathrm{ext}$.
The general expressions for the gating variables vary depending on the specific variable being considered. We set
\begin{equation*}
    \label{eq:tau}
    \begin{aligned}
       \tau_j(u) &= \frac{1}{a_j(u) + b_j(u)}, \qquad g_{\infty}^j(u) = \frac{a_j(u)}{a_j(u) + b_j(u)}, \qquad \mathrm{where}\, j=s,k,c.
    \end{aligned}
\end{equation*}
Table \ref{table:ab} reports to the expressions of $a_j(\cdot)$ and $b_j(\cdot)$, which enters the definition of $\tau_j$ and $g_{\infty}^j$ for the three different gating variables. 
\noindent The ionic currents related to sodium-potassium channels ($I_{\text{pump}},I_{\text{Glia}},I_{\text{diff}}$) are defined as follows:
\begin{equation*}
    \label{eq:equazioni_ionico}
    \begin{aligned}
        I_\mathrm{pump}(k,s) = \frac{\rho}{1+\exp(5.5 - k(t))}\cdot\frac{1}{1+\exp\left(\frac{25-s(t)}{3}\right)},\quad
        I_\mathrm{Glia}(k) = \frac{\mathrm{G_{glia}}}{1+\exp\left(\frac{18 - k(t)}{2.5}\right)},\quad
        I_\mathrm{diff}(k) = \epsilon(k(t) - K_{\mathrm{bath}}).
    \end{aligned}
\end{equation*}
\begin{table}[h!]
\centering\caption{Functions for the Barreto-Cressman ionic model \cite{barreto2011ion,cressman2009influence}}\label{table:ab}%
\begin{tabular}{ll}  
\toprule
 $a_j=a_j(t)$ & $b_j=b_j(t)$ \\
\midrule
      $ \displaystyle a_s(t) = \frac{0.1(u(t) +30)}{1-\exp(-0.1(u(t)+30))}$ &  $\displaystyle b_s = 4 \exp\left(- \frac{u(t)+55}{18}\right)$  \\
     $ \displaystyle a_k(t) =0.07\exp\left(-0.2(u(t)+44)\right)$ &  $ \displaystyle b_k = \frac{1}{1 + \exp(-0.1(u(t)+14))}$ \\
    $ \displaystyle a_c(t) = \frac{0.01(u(t)+34)}{1-\exp(-0.1(u(t)+34))}$ & $ \displaystyle b_c = \frac{1}{8}\exp\left(-\frac{u(t)+44}{80}\right)$ \\
\bottomrule
\end{tabular}
\end{table}

\noindent Note that the different ionic currents depend non-linearly on the gating variables:
\begin{equation}
    \label{eq:equazioni_correnti}
    \begin{aligned}
        I_\mathrm{Na}(g^s,g^k,u) &= \left( \mathrm{G_{NaL}} + \mathrm{G_{Na}} (g^s(t))^3g^k(t) \right) (u(t) - \mathrm{E_{Na}}), \\
        I_\mathrm{K}(g^c,c,u) &= \left( \mathrm{G_{K}}(g^c(t))^4 + \mathrm{G_{AHP}}\frac{c(t)}{1 + c(t)} + \mathrm{G_{KL}} \right) (u(t) - \mathrm{E_{K}}),\\
        I_\mathrm{Cl}(u) &=  \mathrm{G_{CIL}}(u(t) - \mathrm{E_{Cl}}). 
    \end{aligned}
\end{equation}
Here, $I_\mathrm{Na}, I_\mathrm{K}, I_\mathrm{Cl}$ refer to the currents of sodium, potassium, and chloride ions, respectively. Sodium channels cause the nerve cell membrane to depolarize and facilitate the conduction of action potentials across the neuronal cell surface while calcium channels contribute to the overall electrical excitability of neurons and play a key role in regulating the release of neurotransmitters to pre-synaptic nerve terminals. Finally, potassium channels are primarily responsible for restoring the cell membrane's resting potential after triggering an action potential and regulating the balance between input and output in individual neurons. 
The Nerst potentials can be rewritten as:
\begin{equation*}
    \label{eq:Nerst_1}
    \begin{aligned}
      &E_{\text{Ca}} = 120\;\mathrm{mV}, \quad
      &E_{\text{Na}} = 26.64 \log \left( \frac{270 - n}{n} \right), \quad
      &E_{\text{K}} = 26.64 \log \left(\frac{k}{(158 - n)} \right), \quad 
      &E_{\text{Cl}} = 26.64 \log \left(\frac{[\mathrm{Cl}]_i}{[\mathrm{Cl}]_o} \right).\\
    \end{aligned}
\end{equation*}
Finally, the dynamics of the ionic model is given by:
\begin{equation}
\label{eq::GBC}
    \begin{aligned}
      \boldsymbol{m}(u,\boldsymbol{y}) = 
      &\left[
    \begin{aligned}
       & \frac{c(t)}{80} + \mathrm{G_{Ca}}\frac{0.002(u(t)-\mathrm{E_{Ca}})}{1 + \exp\left(-\frac{25 + u(t)}{2.5} \right)}  \\
       &   \frac{1}{\tau} \left( I_\mathrm{diff} - 14 I_\mathrm{pump} - I_\mathrm{glia} + 7 \gamma I_\mathrm{K} \right)  \\
       & \frac{1}{\tau} \left( \gamma I_\mathrm{Na} - 3I_\mathrm{pump} \right)  \\
       & 3 \frac{g^s(t) - g^s_{\infty}}{ \tau_{c}}  \\
       & 3 \frac{ g^k(t) - g^k_{\infty}}{ \tau_{k}} \\
       &  3 \frac{ g^c(t) - g^c_{\infty}}{ \tau_{s}}
    \end{aligned}
    \right].
\end{aligned}
\end{equation}

 \begin{table}[h]
 \centering
 \footnotesize
    \caption{Parameters for Barreto-Cressman ionic model. Values taken from \cite{barreto2011ion}.}\label{table:G2}% 
    \begin{tabular}{llll}  
    \toprule
    \textbf{Parameter} &  \textbf{Description} &  \textbf{Values} & \\
    \midrule
    $C_m$ &  \text{Membrane capacitance}& 10$^{-2}$ & $\mathrm{mF/cm^2}$   \\
    G$_{\text{AHP}}$ &  \text{Conductance of after hyperpolarization current}& 0.01 &$\mathrm{mS/cm^2}$   \\
    G$_{\text{KL}}$ &\text{Conductance of potassium leak current} & $0.05$ &$\mathrm{mS/cm^2}$ \\
    G$_{\text{Na}}$ & Conductance of persistent sodium current & $100$ &$\mathrm{mS/cm^2}$  \\
    G$_{\text{CIL}}$ & \text{Conductance of chloride leak current}& $0.05$ &$\mathrm{mS/cm^2}$ \\
    G$_{\text{NaL}}$ & \text{Conductance of sodium leak current} & $0.0175$ &$\mathrm{mS/cm^2}$  \\
    G$_{\text{Ca}}$ & \text{Calcium conductance} & $0.1$ &$\mathrm{mS/cm^2}$ \\
    G$_{\text{K}}$ & \text{Conductance of potassium current} & $40.0$ &$\mathrm{mS/cm^2}$   \\
    G$_{\mathrm{glia}}$ & \text{Strength of glial uptake} & $66.66$ &$\mathrm{mM/s}$   \\
    K$_{\text{bath}}$ & \text{Conductance of potassium} & $8.0$ &$\mathrm{mM}$\\
    \bottomrule
\end{tabular}  
\end{table}

\begin{table}[h]
\centering
\footnotesize
    \caption{Variables of the  Barreto-Cressman ionic model.}\label{table:G1}%
    \begin{tabular}{lll}  
     \toprule
    \textbf{Variable} &  \textbf{Description} & \textbf{Units}  \\
    \midrule
    $n$ &  \text{Intra/extra-cellular sodium concentration} & $\mathrm{mM}$  \\
    $k$ & \text{Intra/extra-cellular  potassium concentration} & $\mathrm{mM}$ \\
    $c$ & \text{Intra/extra-cellular  calcium concentration}  & $\mathrm{mM}$  \\
    $g^s$ & \text{Activating sodium gate} & \\
    $g^k$ & \text{Activating potassium gate} &\\
    $g^c$ & \text{Inactivating sodium gate} &\\
    $I_\mathrm{{K}}$ & \text{Potassium current} & $\mathrm{\mu A/cm^2}$ \\
    $I_\mathrm{{Na}}$ & \text{Sodium current} &$\mathrm{\mu A/cm^2}$ \\
    $I_\mathrm{{Cl}}$ & \text{Chloride current} &$\mathrm{\mu A/cm^2}$ \\
    $I_\mathrm{{diff}}$ & \text{Potassium diffusion to the nearby reservoir} & $\mathrm{mM/s}$ \\
    $I_\mathrm{{pump}}$ & \text{Pump current} &$\mathrm{mM/s}$ \\
    $I_\mathrm{{glia}}$ & \text{Glial uptake} &$\mathrm{mM/s}$ \\
    $\mathrm{E_{Ca}}$ & \text{Nerst potential of calcium} & $\mathrm{mV}$   \\
    $\mathrm{E_K}$ & \text{Nerst potential of potassium} & $\mathrm{mV}$  \\
    $\mathrm{E_{Na}}$ & \text{Nerst potential of sodium} & $\mathrm{mV}$   \\
    $\mathrm{E_{Cl}}$ & \text{Nerst potential of chloride} & $\mathrm{mV}$ \\
    $\tau_g$ & \text{Forward rate constant for transition between open/closed gate} &$\mathrm{mM/s}$ \\
    $g_{\infty}$ & \text{Backward rate constant for transition between open/closed gate} & $\mathrm{mM/s}$ \\
 \bottomrule
\end{tabular}
\end{table}
\noindent Taking advantage of the abstract semi-discrete formulation of Section \ref{sec:3}, it is possible to analyze the semi-discrete formulation of the Barreto-Cressman model coupled with the monodomain model. We can construct the weak formulation of the coupled multiscale model as follows.
\vspace{1mm}\\
\noindent For any $t \in (0,T]$, find $(u(t),\boldsymbol{y}(t)) \in  V\times V^n$:
\begin{equation}
 \left\{
    \begin{aligned}
        &\chi_m C_m  \left(\frac{\partial u(t)}{\partial t},v\right)_{\Omega} + a(u(t),v) + \mathrm{\chi_m} r_{\text{ion}}(u(t),\boldsymbol{y}(t),v) =  (I^{\mathrm{ext}},v)_\Omega  &\forall \:   v \in V, \\
       &\left(\frac{d\boldsymbol{y}(t)}{dt},\boldsymbol{w}\right) = - \boldsymbol{r}_{\text{m}}(u,\boldsymbol{y},\boldsymbol{w})  &\forall \:  \boldsymbol{w} \in V^n, \\
        &u(0) = u^0,\; \boldsymbol{y}(0)=\boldsymbol{y}^0 &\text{in } \Omega.
    \end{aligned}
    \right.
    \label{eq:barreto_weak_}
\end{equation}
where $r_\mathrm{ion}$, $r_\mathrm{m}$, and $a$ are defined as in \eqref{eq:barreto_weak}. In Table \ref{table:G2} and in Table \ref{table:G1}, we show all the parameters and variables in the model, respectively.

\section{PolyDG semi-discrete formulation}
\label{sec:3}
\noindent We next introduce the PolyDG semi-discrete formulation of the problem in Equation \eqref{eq::weak_general}. Let $\mathcal{T}_h$ be a polytopic mesh partition of the domain $\Omega$ made of disjoint elements $K$, where for each element, we define $h_K$ as its diameter, setting $h = \text{max}_{K\in \mathcal{F}_h} h_K < 1$.
We define the interfaces as the intersection of the ($d$-1)-dimensional facets of neighbouring elements. In the case when $d = 2$, we note that the interfaces of a given element $K$ will always consist of segments. For $d = 3$, we assume that each interface of an element $K$ may be subdivided by a set of planar triangles defined as $\mathcal{F}_h$. We denote by $\mathcal{F}_h^I$ the union of all interior faces that are contained in $\Omega$ and $\mathcal{F}_h^N$ the ones lying on $\partial \Omega$.
\begin{assumption} (Mesh Regularity \cite{pietro2020hybrid})
\label{as:mesh}
The mesh $\{\mathcal{T}_h\}_h$ satisfies the following properties:
\begin{itemize}
    \item \textit{Shape Regularity:} $\forall \: K \in \mathcal{T}_h \textit{ it holds}: \quad c_1h_K^d \lesssim |K| \lesssim c_2h_K^d$
\item \textit{Contact Regularity: $\forall \: F$ in $\mathcal{F}_h$ with F $\subseteq \overline{K}$ for some $K \in \mathcal{T}_h$, it holds $h_K^{d-1}\lesssim |F|$, where $|F|$ is the Hausdorff measure of the space F. }
\item \textit{Submesh Condition: There exists a shape-regular, conforming, matching simplicial submesh $\hat{\mathcal{T}}_h$ such that :} 
\begin{itemize}
    \item $\forall \: \hat{K} \in \hat{\mathcal{T}}_h \exists K \in \mathcal{T}_h: \quad \hat{K} \subseteq K$.
    \vspace{1mm}
    \item  \textit{The family $\{\hat{\mathcal{T}}_h\}$ is shape and contact regular.}
    \vspace{1mm}
    \item  \textit{$\forall \: \hat{K} \in \hat{\mathcal{T}}_h, K \in \mathcal{T}_h$ with $\hat{K} \subseteq K$, it holds $h_K \le h_{\hat{K}}$}.
\end{itemize}
\end{itemize}
\end{assumption}
\noindent Let us define $\mathbb{P}^p(K)$ the space of polynomial of total degree $p\ge1$ over the element $K$. We can introduce the discontinuous finite element space: $ V_h^{\text{DG}}= \{v_h \in L^2(\Omega) : v_h|_{K} \in \mathbb{P}^p (K)\quad \forall \: K \in \mathcal{T}_h \}$. Given $F \in \mathcal{F}_h^I$ be the face shared by the elements $K^{\pm}$, let $\boldsymbol{n}^{\pm}$ be the normal vector on face $F$. Given a regular scalar-valued function $v$ and vector-valued function $\boldsymbol{q}$ we define the trace operators \cite{arnold2002unified}:
\begin{equation*}
    \begin{aligned}
      &\averagel v \averager = \frac{1}{2} (v^+ + v^-) , \quad & \jumpl v  \jumpr& = v^+ \boldsymbol{n}^+ + v^- \boldsymbol{n}^-, \quad &\text{on $F \in \mathcal{F}_h^I$},&\\
      &\averagel \boldsymbol{q} \averager =  \frac{1}{2} (\boldsymbol{q}^+ + \boldsymbol{q}^-) , \quad & \jumpl \boldsymbol{q} \jumpr & = \boldsymbol{q}^+\cdot \boldsymbol{n}^+ + \boldsymbol{q}^-\cdot \boldsymbol{n}^-,  \quad &\text{on $F \in \mathcal{F}_h^{I}$}.&\\
    \end{aligned}
\end{equation*}
We remind that we use the superscripts $\pm$ on the functions to denote the traces of the functions on $F$ taken within the interior to $K^{\pm}$, respectively. In order to introduce the discretization we define the following penalization function $\eta: \mathcal{F}_h \rightarrow \mathbb{R}_+$:
\begin{equation}
\label{eq:eta}
\eta = \eta_0 \{\Sigma_K\}_A\frac{p^2}{\{h\}_H} \quad\text{on} \; F \in \mathcal{F}_h^I,
\end{equation}
where $\{\cdot\}_A$ is the arithmetic average operator and 
$\{\cdot\}_H$ is the armonic average operator $i.e.$ $\{h\}_H = \frac{2h_+h_-}{h_+ + h_-}$ and $\eta_0>0$ is a parameter to be chosen. Moreover, we define $\Sigma_K=\|\sqrt{\boldsymbol{\Sigma}}|_K\|^2_2$. Implicit in this assumption is that the discontinuities of $K$ are aligned with $\tau_n$. We introduce the bilinear form $\mathcal{A}(\cdot,\cdot): V_h^{\text{DG}}\times V_h^{\text{DG}} \rightarrow \mathbb{R}$ as:
\begin{equation*}
    \begin{aligned}
    \small
       \mathcal{A}(u,v) = \int_{\Omega} \mathbf{\Sigma}\nabla_h u \cdot  \nabla_h v \;dx&+ \sum_{F \in \mathcal{F}_h^I} \int_F (\eta  \jumpl u \jumpr \cdot  \jumpl v  \jumpr - \averagel \boldsymbol{\Sigma} \nabla u \averager \cdot  \jumpl v\jumpr - \jumpl u \jumpr \cdot \averagel \boldsymbol{\Sigma} \nabla v\averager) d\sigma \quad \forall \: u,v \in V^{\text{DG}}_h ,
    \end{aligned}
\end{equation*}
where $\nabla_h$ is the elementwise gradient. The semi-discrete formulation of problem \eqref{eq:monodomain} reads:
\vspace{1mm}\\
\noindent For any $t \in (0,T]$, find $(u_h(t),\boldsymbol{y}_h(t)) \in V^{\text{DG}}_h \times \left[V^{\text{DG}}_h\right]^n \text{ such that:}$
\begin{equation}
\label{eq:weak_barreto_2}
 \left\{ 
    \begin{aligned}
        &\chi_m C_m  \left(\frac{\partial u_h(t)}{\partial t},v_h\right)_{\Omega} + \mathcal{A}(u_h(t),v_h) + \chi_m\left( f(u_h(t),\boldsymbol{y}_h(t)),v_h\right)_{\Omega} =  (I_h^{\mathrm{ext}},v_h)_{\Omega}  &\forall \:  v_h \in V_h^{\text{DG}}, \\
        &  \left(\frac{d\boldsymbol{y}_h(t)}{d t},\boldsymbol{w}_h\right)_{\Omega} + \left( \boldsymbol{m}(u_h(t),\boldsymbol{y}_h(t)),\boldsymbol{w}_h\right)_{\Omega} =  0  &\forall \:  \boldsymbol{w}_h \in [V_h^{\text{DG}}]^n, \\
        &u_h(0) = u_h^0,\; \boldsymbol{y}_h(0) = \boldsymbol{y}_h^0 &\text{in } \Omega.
    \end{aligned}
    \right.
\end{equation}
%\subsection{Barreto-Cressman coupled semi-discrete formulation}
%\label{sec:31}
\noindent Under the same assumptions in Section \ref{sec:2}, it is possible to introduce the PolyDG semi-discrete formulation also for the coupled monodomain problem with the Barreto-Cressman ionic model, which reads as follows: 
\vspace{1mm}
\\
\noindent For any $t \in (0,T] $ find $(u_h(t), \boldsymbol{y}(t)) \in V_h^{\text{DG}}\times[V_h^{\text{DG}}]^n$ such that:
\begin{equation}
 \left\{
    \begin{aligned}
        &\chi_m C_m  \left(\frac{\partial u_h(t)}{\partial t},v_h\right)_{\Omega} + \mathcal{A}(u_h(t),v_h) + \mathrm{\chi_m} r_{\text{ion}}(u_h(t),\boldsymbol{y}_h(t),v_h) =  (I_h^{\mathrm{ext}},v_h)_\Omega  &\forall \:   v_h \in V_h^\mathrm{DG}, \\
       &\left(\frac{d\boldsymbol{y}_h(t)}{dt},\boldsymbol{w}_h\right)_{\Omega} = - \boldsymbol{r}_{\text{y}}(u_h(t),\boldsymbol{y}_h(t),\boldsymbol{w}_h)  &\forall \:  \boldsymbol{w}_h \in  [V_h^\mathrm{DG}]^n, \\
         &u_h(0) = u_h^0,\; \boldsymbol{y}_h(0)=\boldsymbol{y}_h^0 &\text{in } \Omega.
    \end{aligned}
    \right.
    \label{eq:barreto_weak_parameters}
\end{equation}

\section{Fully-discrete formulation}
\label{sec:4}
\noindent Let $N_h$ be the dimension of $V_h^\mathrm{DG}$ and let $(\varphi_j)^{N_h}_{j=0}$ be a suitable basis for $V_{h}^\mathrm{DG}$, then $u_h(t) = \sum_{j=0}^{N_h} U_j(t)\varphi_j$ and $y_l(t) = \sum_{j=0}^{N_h} Y^l_j(t)\varphi_j$ for all $l=1,...,n$. We denote $\mathbf{U} \in \mathbb{R}^{N_h}$, $\mathbf{Y}_l \in \mathbb{R}^{N_h}$ for all $l=1,...,n$ and $\mathbf{Y} = [\mathbf{Y}_1,...,\mathbf{Y}_n]^\top$. We define the matrices:
\begin{equation}
    \begin{aligned}
      [\mathbf{M}]_{ij} &= (\varphi_i,\varphi_j)_{\Omega}, \; &\text{(Mass matrix),}& \quad  i,j = 1,...,N_h \\
      [\mathbf{F}]_{j} &= (I^\text{ext},\varphi_j)_{\Omega},  \; &\text{(Forcing term),}& \quad j = 1,...,N_h\\
      [\mathbf{I}(u,\boldsymbol{y})]_{j} &= (f(u,\boldsymbol{y}),\varphi_j)_{\Omega},  \; &\text{(Non-linear ionic forcing term),}& \quad  j = 1,...,N_h\\
      [\mathbf{G}_l(u,\boldsymbol{y})]_{j} &= (\boldsymbol{m}_l(u,\boldsymbol{y}),\boldsymbol{\varphi}_j)_{\Omega},  \; &\text{(Dynamics of the ionic model),}& \quad  j = 1,...,N_h, \; l=1,...,n  \\ 
      [\mathbf{A}]_{ij} &= \mathcal{A}(\varphi_i,\varphi_j) &\text{(Stiffness matrix),}& \quad  i,j = 1,...,N_h.
    \end{aligned}
\end{equation}
\noindent The terms related to the ionic forcing and the ionic coupling dynamics are identified by $\mathbf{I}(t),\mathbf{G}(t)$, respectively. The algebraic form of problem \eqref{eq:weak_barreto_2} reads
\begin{equation}
\left\{
    \label{eq:Full_discrete_generical}
    \begin{aligned}
      & \chi_m C_m \mathbf{M} \dot{\mathbf{U}}(t) + \mathbf{A}  \mathbf{U}(t) + \mathbf{I}(t) = \mathbf{F}(t)  , \\
      &\dot{\mathbf{Y}}(t) + \mathbf{G}(t) = \mathbf{0},\\
      &\mathbf{U}(0) = \mathbf{U}_0, \; \mathbf{Y}(0) = \mathbf{Y}_0.
    \end{aligned}
\right.
\end{equation}
We divide the interval $[0, T]$ into N subintervals $(t_n,t_{n+1}]$ of length $\Delta t$ so that $t_n =  n\Delta t$, for $n=0,...,N-1$. Here, we employ a semi-implicit treatment of the diffusion term while the ion term is treated explicitly. 
Furthermore, for the temporal discretization, we consider a Crank-Nicholson second-order scheme \cite{sundnes2007computing}. Given $\mathbf{U}_0$ and $\mathbf{Y}_0$, the following discrete scheme reads:
\\
Find $\mathbf{U}^{n+1} \simeq \mathbf{U}(t_{n+1})$ and  $\mathbf{Y}^{n+1} \simeq \mathbf{Y}(t_{n+1})$ for $n=0,...,N-1$, such that:
\begin{equation}
\left\{
    \label{eq:Full_discrete_complete}
    \begin{aligned}
      & \left(\chi_m C_m M + \frac{\Delta t}{2}A \right) \mathbf{U}^{n+1}  =  \left(\chi_m C_m M - \frac{\Delta t}{2}A \right) \mathbf{U}^{n} + \Delta t \mathbf{F}^{n+1} - \mathrm{\chi_m} \Delta t \mathbf{I}_\mathrm{stim}^{n+1}  , \\
      &\mathbf{Y}^{n+1} = \mathbf{Y}^{n} - \Delta t \mathbf{G}^{n},\\
      &\mathbf{U}^0 = \mathbf{U}_0, \; \mathbf{Y}^0 = \mathbf{Y}_0,
    \end{aligned}
\right.
\end{equation}
\noindent where the ionic current at step $n+1$ is computed by second-order interpolation exploiting the values of the ionic terms $I^n$ and $I^{n+1}$ evaluations at the previous steps, as follows:
\begin{equation}
\begin{aligned}
  &\mathbf{I}_\mathrm{stim}^{n+1} = \frac{3}{2}\mathbf{I}^{n} - \frac{1}{2} \mathbf{I}^{n-1}.
\end{aligned}
\label{eq:Istim}
\end{equation}
A computational advantage of such a semi-implicit scheme is that the stiffness and mass matrices can be assembled only at the beginning of the computation \cite{pezzuto2016space}. 

By exploiting the definitions introduced in Section \ref{sec:4} and the interpolation of the ionic current (ICI) framework \cite{pathmanathan2011significant}, the ionic vector is defined as follows:
\begin{equation*}
      [\mathbf{I}]^{n}_{j} = (I_\mathrm{{ion}}^{n},\varphi_j)_{\Omega} =  ( I^{n}_\mathrm{Na} + I^{n}_\mathrm{K} + I^{n}_\mathrm{Ca},\varphi_j)_{\Omega}, \quad  j = 1,...,N_h.
\end{equation*}
\noindent The time discretization of the problem considering the Barreto-Cressman ionic model is done by treating in an implicit way the diffusive component and in an explicit way the ionic forcing component. For this coupled problem, we also exploit a semi-implicit second-order method as described above. The fully discrete formulation reads as in Equation~\eqref{eq:Full_discrete_complete}.

\section{Numerical results}
\label{sec:5}
In this section, we present a set of numerical tests aimed at assessing the numerical performance of the method in approximating pathophysiological scenarios of brain electrophysiology.
The initial data for the potential are defined such that the unstable region of grey matter has an initial potential $u^0 = -50\,\mathrm{mV}$ while the remaining part of the domain has $u^0 = -67\,\mathrm{mV}$, as done in \cite{barreto2011ion}. All the other variables of the ionic model are initialized as in Table~\ref{table:BC_IC}, while the parameters are set as in Table~\ref{table:G2}. In Table~\ref{table:conductivity}, we report the conductivity values taken exploiting the bulk conductivity tensor \cite{potse2006comparison} for intracellular end extracellular conductivities taken from \cite{schreiner2022simulating}. The numerical
simulations have been obtained on the open-source Lymph library \cite{antonietti2024textttlymph}, implementing
the PolyDG method for multiphysics. 
\begin{table}[!htb]
\begin{minipage}[t]{0.47\linewidth}
\centering
    \caption{Initial conditions for the coupled monodomain Barreto-Cressman ionic model. Values taken from \cite{barreto2011ion}.}\label{table:BC_IC}
    \centering 
    \begin{tabular}{lll}
    \toprule
    %\rowcolor{bluepoli!40} % comment this line to remove the color
   \textbf{Variable} &\textbf{Inital Value} & \textbf{Units} \\
    \midrule
    $u_n^0$ & $-50$ & $\mathrm{mV}$ \\
    $u_s^0$ & $-67$ & $\mathrm{mV}$ \\
    $n^0$ &  $15.5$ & $\mathrm{mM}$ \\
    $c^0$ & $0$ & $\mathrm{mM}$     \\
    $k^0$ & $7.8$ & $\mathrm{mM}$\\
    $g^{n,0}$ & $0.0936$ & -\\
    $g^{c,0}$  &  $0.08553$ & -\\
    $g^{k,0}$  &  $0.96859$ & -\\
    \bottomrule
    \end{tabular}
\end{minipage}\hfill
\begin{minipage}[t]{0.5\linewidth}
\centering 
\caption{Values for the parameters of the Bulk conductivity tensor from \cite{schreiner2022simulating}.}
\label{table:conductivity}%
\vspace{10mm}
\begin{tabular}{llll}  
\toprule
 \textbf{Tissue type} & \textbf{$\sigma_n$ } $[\mathrm{Sm^{-1}}]$  & \textbf{$\sigma_t$} $[\mathrm{Sm^{-1}}]$ \\
\midrule
    Grey matter     & 0.0735 & 0.0735 \\
    White matter    &  0.0557 & 0.0139 \\
\bottomrule
\end{tabular}
\end{minipage}
\end{table}
\subsection{Test case 1: Simplified computational domain, physiological/pathological coefficients}
The behavior of white and grey matter in response to internal or external stimuli, such as an ion imbalance or external brain stimulation, differs because of the preponderance of cell bodies in grey matter and axons in white matter. This makes white matter strongly anisotropic, with conductivity tensors following specific preferential directions depending on the part of the tissue analyzed \cite{mardal2022mathematical}. By defining $\Omega_0$ as the subsection of unstable grey matter, $\Omega_\mathrm{WMV}$ as the subsection of white matter with vertical anisotropy, and $\Omega_\mathrm{WMH}$ as the subsection of white matter with horizontal anisotropy; defined as in Table \ref{table:domain1}. \begin{figure}[h!]
\begin{subfigure}[b]{0.33\textwidth}
\centering
    \begin{tabular}{ll}  
    \toprule
   \textbf{Variable} &\textbf{Subregion}  \\
\midrule
    {$\Omega$} & (0,1)x(0,1) \\
    {$\Omega_{0}$} & (0,1)x(0.4,0.6)  \\
    {$\Omega_{\text{WMV}}$} & (0,0.5)x(0,0.4)  \\
    {$\Omega_{\text{WMH}}$} & (0.5,1)x(0,0.4)  \\
\bottomrule
    \end{tabular}
    \vspace{4mm}
    \caption{Distinction of different tissues.}
    \label{table:domain1}
\end{subfigure}
\begin{subfigure}[b]{0.33\textwidth}
\centering
\includegraphics[scale=0.17]{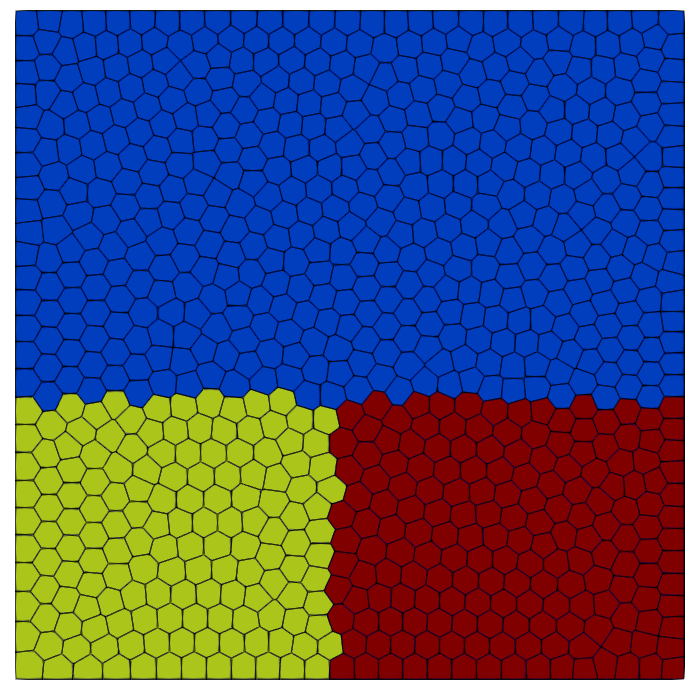} \caption{Polytopal grid} \label{fig:comput_domain_square}
\end{subfigure}\hfill
\begin{subfigure}[b]{0.33\textwidth}
\centering
\includegraphics[scale=0.135]{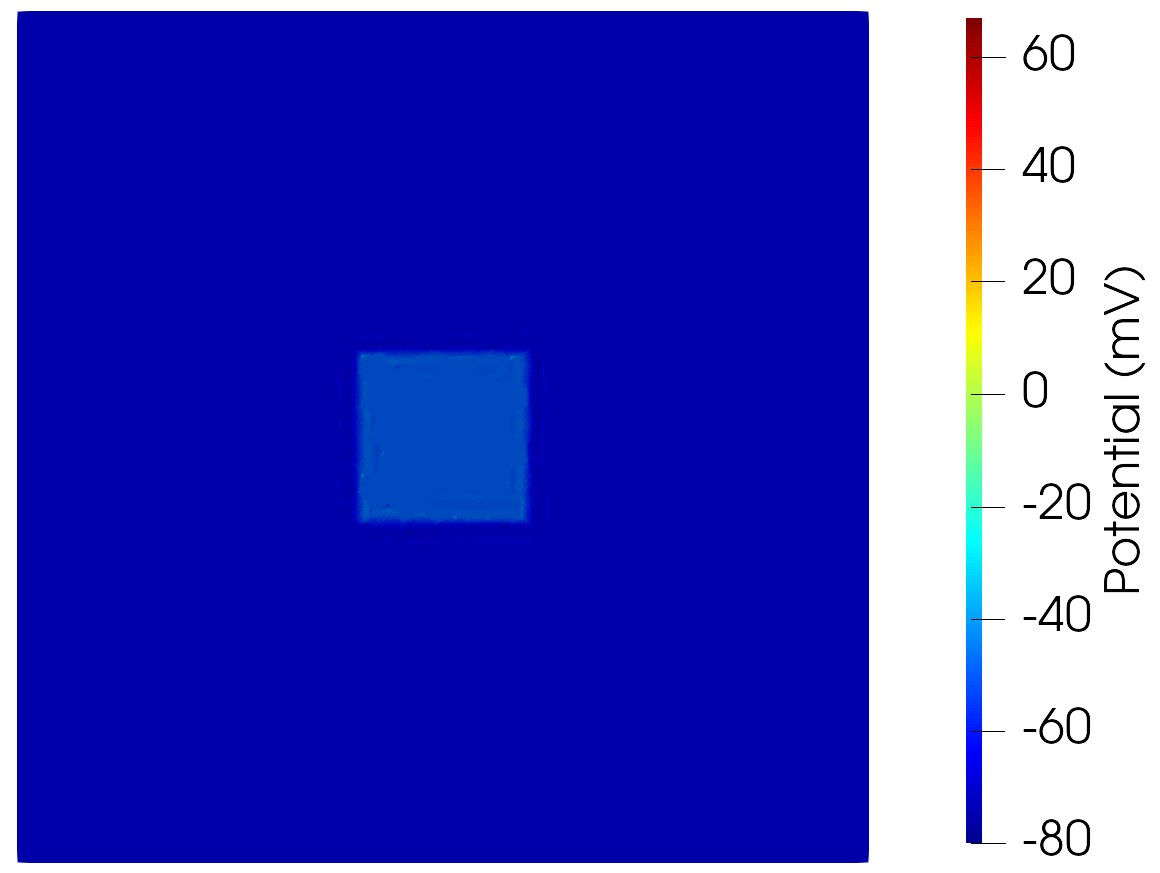} \caption{Initial condition for $u$} \label{fig:comput_domain_IC}
\end{subfigure}
    \caption{Test case 1: (\ref{fig:comput_domain_square}) computational domain and corresponding grid, grey matter region (blue), horizontal anisotropic white matter region (lime), vertical anisotropic white matter (red), (\ref{fig:comput_domain_IC}) initial condition of the transmembrane potential.}
    \label{fig:mesh_brain_bc}
\end{figure}
To investigate the effect of heterogeneity on the numerical approximation, we consider a domain of size $(0,1)\times(0,1)\;\mathrm{cm}$ that is subdivided into subregions characterized by different values of the conductivity tensors. Specifically, the conductivities values reported in Table~\ref{table:conductivity} encode the difference between the white and grey matters. Moreover, within the white matter, we add a further differentiation in vertical and horizontal anisotropy, with a conductivity tensor that has symmetrically opposite directions. 
\begin{figure}[h!]
    \begin{subfigure}[b]{\textwidth}
    \centering
    \includegraphics[scale=0.28]{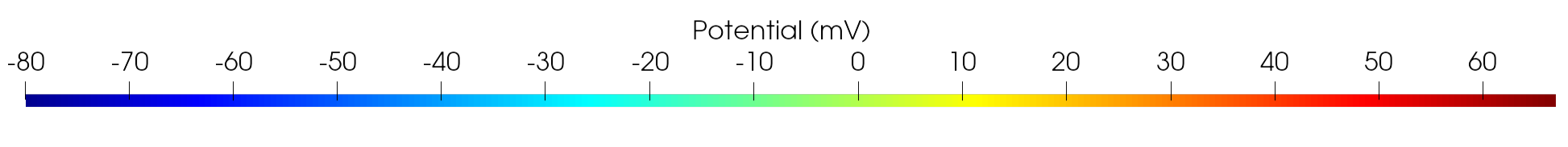}
    \end{subfigure}
    \begin{subfigure}[b]{0.33\textwidth}
      \centering
    \includegraphics[scale=0.22]{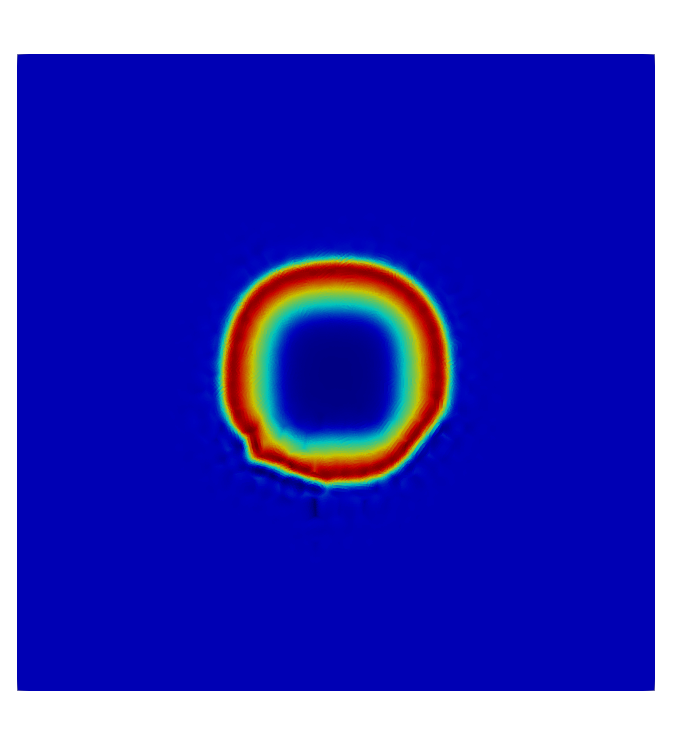} 
    \caption{$t=3 \; \mathrm{ms}$}\label{fig:2D_1}
    \end{subfigure}\hfill
    \begin{subfigure}[b]{0.33\textwidth}
      \centering
    \includegraphics[scale=0.22]{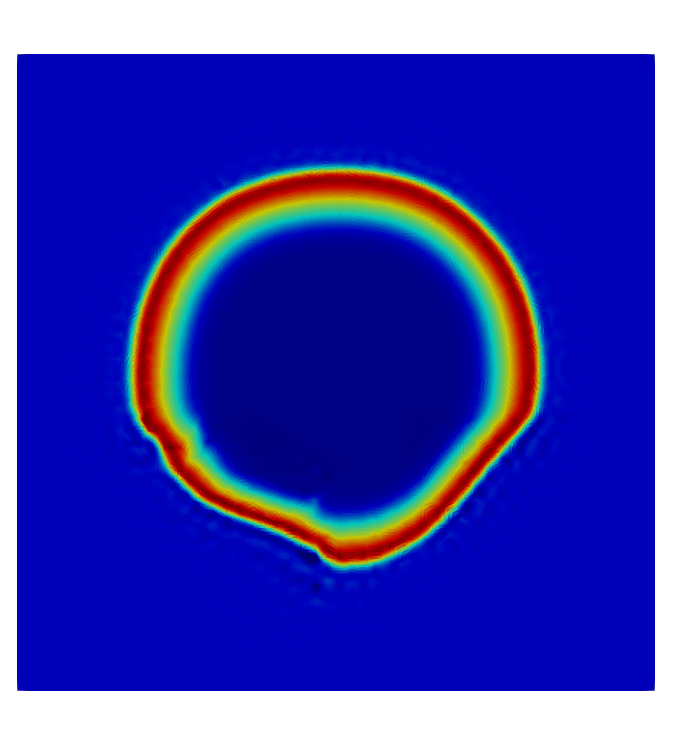} 
    \caption{$t=4.5 \;\mathrm{ms}$}\label{fig:2D_2}
    \end{subfigure}\hfill
    \begin{subfigure}[b]{0.33\textwidth}
      \centering
    \includegraphics[scale=0.22]{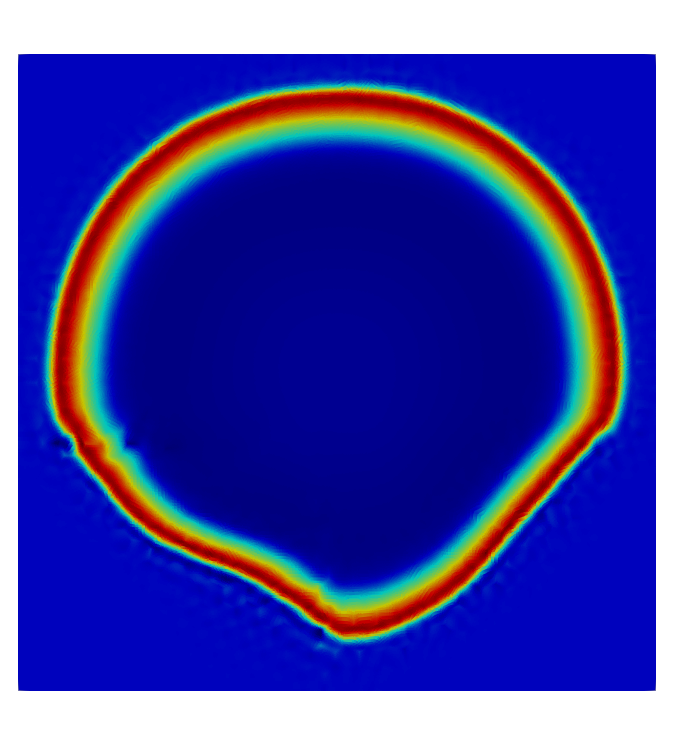} \caption{$t=6\;\mathrm{ms}$}\label{fig:2D_3} 
\end{subfigure}
\caption{Test case 1: Snapshots of the transmembrane potential in gray-white matter tissue at three different time instants.}
\label{fig:bc_evolution_2D}
\end{figure}  

The domain is thus divided into three regions plus an unstable area, introduced into the grey matter, where a potential imbalance is assumed as an initial condition. Then, the change in the potential generated by the unstable area evolves freely within the two different tissue types. The simulation is performed considering a mesh with 800 elements and with $\Delta t = 10^{-3}$ and $T=20ms$.
Figure~\ref{fig:mesh_brain_bc} shows the computational domain and grid used for the simulation. 
Subdomain distinction of unstable grey matter ($\Omega_{0}$), stable grey matter, white matter with vertical anisotropy ($\Omega_{\text{WMV}}$) and white matter with horizontal anisotropy ($\Omega_{\text{WMH}}$) is reported in Figure~\ref{fig:comput_domain_square}.
\noindent The evolution of the transmembrane potential over the domain is shown in Figure~\ref{fig:bc_evolution_2D}. The potential evolves isotropically in the white matter and anisotropically in the grey matter, actively propagating the wavefront originated from the unstable grey matter throughout the tissue.   
In Figure~\ref{fig:bc_evolution}, we present the transmembrane potential profile along the domain's diagonal. The speed of the transmembrane potential in the white matter is observed to be lower and generates, as expected, small oscillations due to the low conductivity values.

\begin{figure}[h!]
\begin{subfigure}[b]{0.33\textwidth}
 \resizebox{0.95\textwidth}{!}{
\begin{tikzpicture}{
      \begin{axis}[
        width=3.875in,
        height=3.56in,
        at={(2.6in,1.099in)},
        scale only axis,
        %xmode=log,
        xmin=-0.05,
        xmax=1.5,
        xminorticks=true,
        xlabel = { $l$ [cm]},
        ylabel = { $u(T)$},
        %log x ticks with fixed point,
        %ymode=log,
        ymin=-90,
        ymax=70,
        yminorticks=true,
        axis background/.style={fill=white},
        title style={font=\bfseries},
        xmajorgrids,
        xminorgrids,
        ymajorgrids,
        yminorgrids,
        legend style={legend cell align=left, align=left, draw=white!15!black}
        ]
        \addplot        table[x=arc_length,y= Field 2, col sep=comma,mark=none] {csv/wm_gm_9_2.csv};
        \draw[dashed,red,line width=2.0pt] (0.7071,-90) -- (0.7071,70);
      \end{axis}}
    \end{tikzpicture}}
     \caption{$t=3 \;\mathrm{ms}$} \label{fig:0D_1}
\end{subfigure}\hfill
\begin{subfigure}[b]{0.33\textwidth}
 \resizebox{0.95\textwidth}{!}{
     \begin{tikzpicture}{
      \begin{axis}[
        width=3.875in,
        height=3.56in,
        at={(2.6in,1.099in)},
        scale only axis,
        %xmode=log,
        xmin=-0.05,
        xmax=1.5,
        xminorticks=true,
        xlabel = { $l$ [cm]},
        ylabel = { $u(T)$},
        %log x ticks with fixed point,
        %ymode=log,
        ymin=-90,
        ymax=70,
        yminorticks=true,
        axis background/.style={fill=white},
        title style={font=\selectfont},
        xmajorgrids,
        xminorgrids,
        ymajorgrids,
        yminorgrids,
        legend style={legend cell align=left, align=left, draw=white!15!black}
        ]
        \addplot
        table[x=arc_length,y= Field 2, col sep=comma,mark=none] {csv/wm_gm_16_2.csv}; 
        \draw[dashed,red,line width=2.0pt] (0.7071,-90) -- (0.7071,70);
        %\legend{Plot}
      \end{axis}}
    \end{tikzpicture}}
     \caption{$t=4.5 \;\mathrm{ms}$} \label{fig:0D_2}
    \end{subfigure}\hfill
    \begin{subfigure}[b]{0.33\textwidth}
  \resizebox{0.95\textwidth}{!}{
    \begin{tikzpicture}{
      \begin{axis}[
        width=3.875in,
        height=3.56in,
        at={(2.6in,1.099in)},
        scale only axis,
        %xmode=log,
        xmin=-0.05,
        xmax=1.5,
        xminorticks=true,
        xlabel = { $l$ [cm]},
        ylabel = { $u(T)$},
        %log x ticks with fixed point,
        %ymode=log,
        ymin=-90,
        ymax=70,
        yminorticks=true,
        axis background/.style={fill=white},
        title style={font=\bfseries},
        xmajorgrids,
        xminorgrids,
        ymajorgrids,
        yminorgrids,
        legend style={legend cell align=left, align=left, draw=white!15!black}
        ]
        \addplot
        table[x=arc_length,y= Field 2, col sep=comma,mark=none] {csv/wm_gm_28_2.csv};
        \draw[dashed,red,line width=2.0pt] (0.7071,-90) -- (0.7071,70);
        %\legend{Plot}
      \end{axis}}
    \end{tikzpicture}}
     \caption{$t=6 \;\mathrm{ms}$} \label{fig:0D_3}
    \end{subfigure}
\caption{Test case 1: Profile of the transmembrane potential along the domain's diagonal.}
\label{fig:bc_evolution}
    \end{figure}
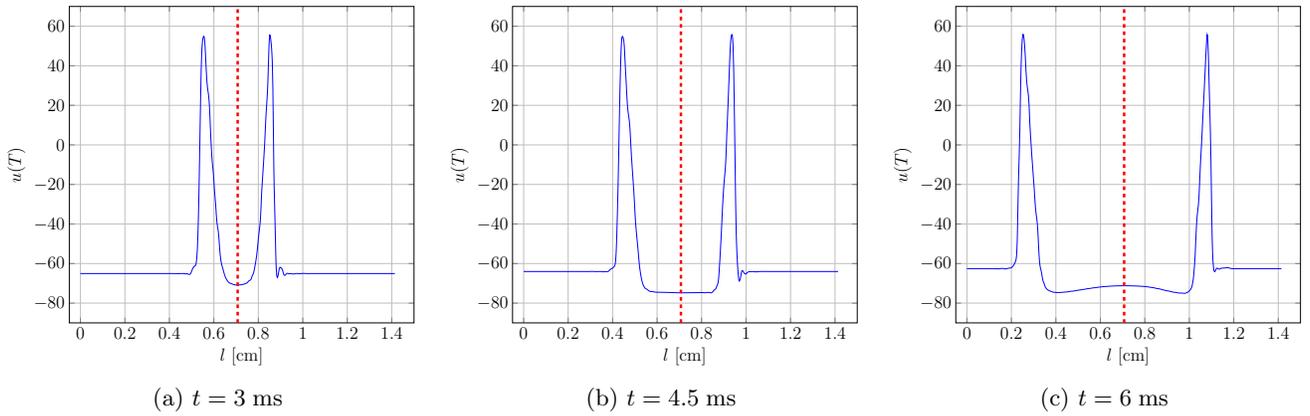
    
\subsubsection*{Approximation of the conduction speed}
\noindent For a numerical verification, we study the evolution of the conduction velocities of the monodomain problem coupled with the Barreto-Cressman ion model as a function of the spatial discretization and the polynomial degree used. Conduction velocity values are calculated by measuring the length of the space traveled by the wavefront in a fixed amount of time along one of the two diagonals. 
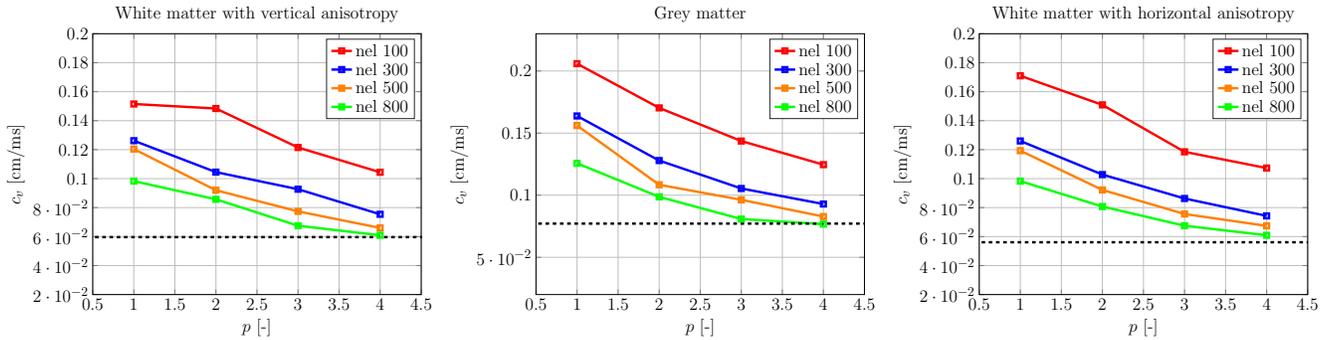
\begin{figure}[h!]
\centering
\begin{subfigure}[b]{0.33\textwidth}
    \resizebox{\textwidth}{!}{
\begin{tikzpicture}{
\begin{axis}[%
width=3.775in,
height=3in,
at={(1.733in,0.687in)},
scale only axis,
xmin=0.5,
xmax=4.5,
xminorticks=true,
xlabel = { $p$ [-]},
ylabel = { $c_v$ [cm/ms]},
%log x ticks with fixed point,
ymin=0.02,
ymax=0.2,
yminorticks=true,
axis background/.style={fill=white},
title={\color{black} White matter with vertical anisotropy },
xmajorgrids,
xminorgrids,
ymajorgrids,
yminorgrids,
legend style={legend cell align=left, align=left, draw=white!15!black}
]
\addplot [color=red, mark=square,line width=2.0pt]
  table[row sep=crcr]{%
    1 0.1515\\
    2 0.1484\\
    3 0.1215\\
    4 0.1044\\
};
\addlegendentry{nel $100$}
\addplot [color=blue, mark=square,line width=2.0pt]
  table[row sep=crcr]{%
    1 0.1262\\
    2 0.1045\\
    3 0.0927\\
    4 0.0754\\
};
\addlegendentry{nel $300$}
\addplot [color=orange,mark=square, line width=2.0pt]
  table[row sep=crcr]{%
    1 0.1204\\
    2 0.092\\
    3 0.0774\\
    4 0.066\\
};
\addlegendentry{nel $500$}
\addplot [color=green, mark=square, line width=2.0pt]
  table[row sep=crcr]{%
    1 0.098285\\
    2 0.08574\\
    3 0.067502\\
    4 0.060917\\
};
\addlegendentry{nel $800$} 
\draw[dashed,black, line width=2.0pt] (0,0.05964) -- (5,0.05964);      
\end{axis}}
\end{tikzpicture}}
\end{subfigure}\hfill
\begin{subfigure}[b]{0.33\textwidth}
    \resizebox{\textwidth}{!}{
\begin{tikzpicture}{
\begin{axis}[%
width=3.775in,
height=3in,
at={(1.733in,0.687in)},
scale only axis,
xmin=0.5,
xmax=4.5,
xminorticks=true,
xlabel = { $p$ [-]},
ylabel = { $c_v$ [cm/ms]},
%log x ticks with fixed point,
ymin=0.02,
ymax=0.23,
yminorticks=true,
axis background/.style={fill=white},
title={\color{black} Grey matter},
xmajorgrids,
xminorgrids,
ymajorgrids,
yminorgrids,
legend style={legend cell align=left, align=left, draw=white!15!black}
]
\addplot [color=red,mark=square, line width=2.0pt]
  table[row sep=crcr]{%
    1  0.2059 \\
    2  0.1703  \\
    3  0.1436  \\
    4  0.1246  \\
};
\addlegendentry{nel $100$}
\addplot [color=blue, mark=square,line width=2.0pt]
  table[row sep=crcr]{%
    1 0.1638\\
    2 0.1279\\
    3 0.1054\\
    4 0.0929\\
};
\addlegendentry{nel $300$}
\addplot [color=orange,mark=square, line width=2.0pt]
  table[row sep=crcr]{%
    1 0.1561\\
    2 0.1084\\
    3 0.0963\\
    4 0.0828\\
};
\addlegendentry{nel $500$}
\addplot [color=green, mark=square, line width=2.0pt]
  table[row sep=crcr]{%
    1 0.1257\\
    2 0.098676\\
    3 0.080789\\
    4 0.076719\\
};
\addlegendentry{nel $800$} 
\draw[dashed,black, line width=2.0pt] (0,0.077129) -- (5,0.077129);   
\end{axis}}
\end{tikzpicture}}
\end{subfigure}\hfill
\begin{subfigure}[b]{0.33\textwidth}
    \resizebox{\textwidth}{!}{
\begin{tikzpicture}{
\begin{axis}[%
width=3.775in,
height=3in,
at={(1.733in,0.687in)},
scale only axis,
xmin=0.5,
xmax=4.5,
xminorticks=true,
xlabel = { $p$ [-]},
ylabel = { $c_v$ [cm/ms]},
%log x ticks with fixed point,
ymin=0.02,
ymax=0.2,
yminorticks=true,
axis background/.style={fill=white},
title={\color{black}  White matter with horizontal anisotropy},
xmajorgrids,
xminorgrids,
ymajorgrids,
yminorgrids,
legend style={legend cell align=left, align=left, draw=white!15!black}
]
\addplot [color=red, mark=square, line width=2.0pt]
  table[row sep=crcr]{%
    1 0.171\\
    2 0.1509\\
    3 0.1185\\
    4 0.1073\\
};
\addlegendentry{nel $100$}  
\addplot [color=blue, mark=square, line width=2.0pt]
  table[row sep=crcr]{%
    1 0.126\\
    2 0.1028\\
    3 0.0863\\
    4 0.0743\\
};
\addlegendentry{nel $300$}   
\addplot [color=orange, mark=square, line width=2.0pt]
  table[row sep=crcr]{%
    1 0.1193\\
    2 0.0922\\
    3 0.0756\\
    4 0.0674\\
};
\addlegendentry{nel $500$} 
\addplot [color=green, mark=square, line width=2.0pt]
  table[row sep=crcr]{%
    1 0.098285\\
    2 0.08074\\
    3 0.067502\\
    4 0.060917\\
};
\addlegendentry{nel $800$} 
\draw[dashed,black, line width=2.0pt] (0,0.0561) -- (5,0.0561);    
\end{axis}}
\end{tikzpicture}} 
\end{subfigure}
\caption{Test case 1: Approximation of the velocity of numerical solution with refinement in $p$ for different tissues: White matter with vertical anisotropy (left) grey matter (center) and white matter with horizontal anisotropy (right).}
\label{fig:cv_BC}
\end{figure}
Since there is no exact solution of the coupled model, we construct a reference simulation on a fine grid composed by 800 elements and exploit a high polynomial degree ($p=6$) to obtain a high-resolution approximation of the solution. Here the penalty parameter is chosen equal to $\eta_0 = 10$.
Figure~\ref{fig:cv_BC} shows the evolution of the wavefront conduction velocity as a function of the polynomial degree $p$ for four different spatial discretizations (characterized by computational meshes formed by $100, 300, 500, 800$ elements, corresponding to a mesh size of $0.18, 0.106, 0.083, 0.0581$ $cm$, respectively) in the three different tissue types described above. Figure~\ref{fig:cv_BC_h} shows the evolution of the conduction velocity as a function of the degrees of freedom of the monodomain discretization. As expected, both increasing the polynomial degree and reducing the mesh size significantly improves the approximation of the conduction velocity.

Moreover, Figure~\ref{fig:cv_BC_h} proves that increasing the polynomial degree is computationally more convenient than refining the mesh, as evidenced by the trend of the discretization based on $300$ elements (blue line) with respect to ones based on $500$ elements (orange line) and $800$ elements (green line), provided that the mesh is not excessively coarse (as in the case of the $100$ elements mesh represented by the red line).

\begin{figure}[h!]
\centering
\begin{subfigure}[b]{0.33\textwidth}
    \resizebox{\textwidth}{!}{
\begin{tikzpicture}{
\begin{axis}[%
width=3.775in,
height=3in,
at={(1.733in,0.687in)},
scale only axis,
xmin=0,
xmax=12800,
xminorticks=true,
xlabel = { ndof [-]},
ylabel = { $c_v$ [cm/ms]},
%log x ticks with fixed point,
ymin=0.02,
ymax=0.2,
yminorticks=true,
axis background/.style={fill=white},
title={\color{black} White matter with vertical anisotropy},
xmajorgrids,
xminorgrids,
ymajorgrids,
yminorgrids,
legend style={legend cell align=left, align=left, draw=white!15!black}
]
\addplot [color=red, mark=square,line width=2.0pt]
  table[row sep=crcr]{%
    300 0.1515\\
    900 0.1484\\
    1800 0.1215\\
    3000 0.1044\\
};
\addlegendentry{nel $100$}
\addplot [color=blue, mark=square,line width=2.0pt]
  table[row sep=crcr]{%
    900 0.1262\\
    1800 0.1045\\
    3000 0.0927\\
    4500 0.0754\\
};
\addlegendentry{nel $300$}
\addplot [color=orange,mark=square, line width=2.0pt]
  table[row sep=crcr]{%
    1500 0.1204\\
    3000 0.092\\
    5000 0.0774\\
    7500 0.066\\
};
\addlegendentry{nel $500$}
\addplot [color=green,mark=square, line width=2.0pt]
  table[row sep=crcr]{%
    2400 0.098285\\
    4800 0.08574\\
    8000 0.067502\\
    12000 0.060917\\
};    
\addlegendentry{nel $800$}
\draw[dashed,black, line width=2.0pt] (0,0.05964) -- (12800,0.05964);      
\end{axis}}
\end{tikzpicture}}    
\end{subfigure}\hfill
\begin{subfigure}[b]{0.33\textwidth}
    \resizebox{\textwidth}{!}{
\begin{tikzpicture}{
\begin{axis}[%
width=3.775in,
height=3in,
at={(1.733in,0.687in)},
scale only axis,
xmin=0,
xmax=12800,
xminorticks=true,
xlabel = {ndof [-]},
ylabel = {$c_v$ [cm/ms]},
%log x ticks with fixed point,
ymin=0.05,
ymax=0.25,
yminorticks=true,
axis background/.style={fill=white},
title={\color{black} Grey matter},
xmajorgrids,
xminorgrids,
ymajorgrids,
yminorgrids,
legend style={legend cell align=left, align=left, draw=white!15!black}
]
\addplot [color=red,mark=square, line width=2.0pt]
  table[row sep=crcr]{%
    300  0.2059 \\
    900  0.1703  \\
    1800  0.1436  \\
    3000  0.1246  \\
};
\addlegendentry{nel $100$}
\addplot [color=blue, mark=square,line width=2.0pt]
  table[row sep=crcr]{%
    900 0.1638\\
    1800 0.1279\\
    3000 0.1054\\
    4500 0.0929\\
};
\addlegendentry{nel $300$}
\addplot [color=orange,mark=square, line width=2.0pt]
  table[row sep=crcr]{%
    1500 0.1561\\
    3000 0.1084\\
    5000 0.0963\\
    7500 0.0828\\
};
\addlegendentry{nel $500$}
\addplot [color=green,mark=square, line width=2.0pt]
  table[row sep=crcr]{%
    2400 0.1257\\
    4800 0.098676\\
    8000 0.080789\\
    12000 0.076719\\
}; 
\addlegendentry{nel $800$}
\draw[dashed,black, line width=2.0pt] (0,0.077129) -- (12800,0.077129);   
\end{axis}}
\end{tikzpicture}}
\end{subfigure}\hfill
\begin{subfigure}[b]{0.33\textwidth}
    \resizebox{\textwidth}{!}{
\begin{tikzpicture}{
\begin{axis}[%
width=3.775in,
height=3in,
at={(1.733in,0.687in)},
scale only axis,
xmin=0,
xmax=12800,
xminorticks=true,
xlabel = { ndof [-]},
ylabel = { $c_v$ [cm/ms]},
%log x ticks with fixed point,
ymin=0.02,
ymax=0.2,
yminorticks=true,
axis background/.style={fill=white},
title={\color{black} White matter with horizontal anisotropy},
xmajorgrids,
xminorgrids,
ymajorgrids,
yminorgrids,
legend style={legend cell align=left, align=left, draw=white!15!black}
]
\addplot [color=red, mark=square, line width=2.0pt]
  table[row sep=crcr]{%
    300 0.171\\
    900 0.1509\\
    1800 0.1185\\
    3000 0.1073\\
};
\addlegendentry{nel $100$}  
\addplot [color=blue, mark=square, line width=2.0pt]
  table[row sep=crcr]{%
    900 0.126\\
    1800 0.1028\\
    3000 0.0863\\
    4500 0.0743\\
};
\addlegendentry{nel $300$}   
\addplot [color=orange, mark=square, line width=2.0pt]
  table[row sep=crcr]{%
    1500 0.1193\\
    3000 0.0922\\
    5000 0.0756\\
    7500 0.0674\\
};
\addlegendentry{nel $500$} 
\addplot [color=green,mark=square, line width=2.0pt]
  table[row sep=crcr]{%
    2400 0.098285\\
    4800 0.08074\\
    8000 0.067502\\
    12000 0.060917\\
};
\addlegendentry{nel $800$}
\draw[dashed,black, line width=2.0pt] (0,0.0561) -- (12800,0.0561);    
\end{axis}}
\end{tikzpicture}}    
\end{subfigure}
\caption{Test case 1: Approximation of the velocity of numerical solution as a function of the total number of degrees of freedom (ndof) for different tissues: White matter with vertical anisotropy (left) grey matter (center) and white matter with horizontal anisotropy (right).}
\label{fig:cv_BC_h}
\end{figure}
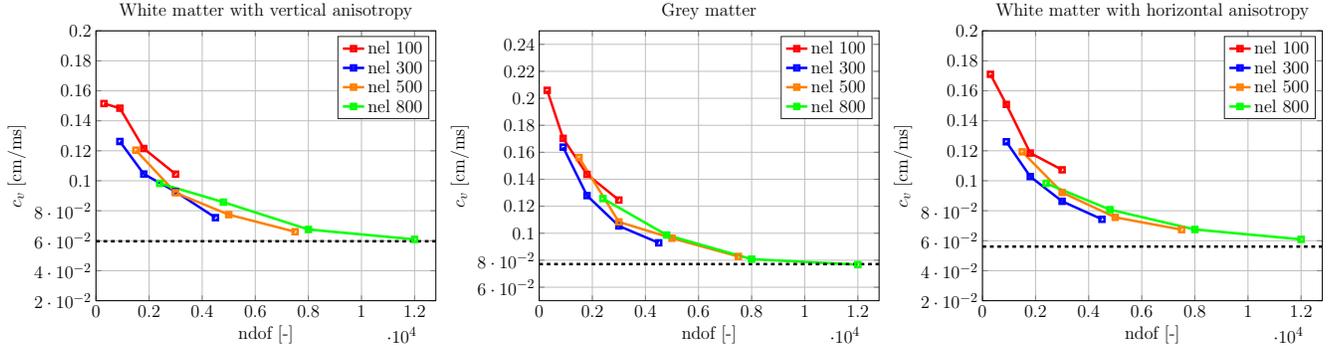

\subsection{Test case 1.1: Influence of potassium dynamics on seizure evolution}
An analysis is then performed on the evolution of transmembrane potential in relation to ionic imbalances generated in the cell membrane, particularly potassium imbalances. The influence of extracellular potassium on neuronal excitability is widely recognized, and dysregulation of this ion is implicated in several forms of epilepsy. The variable $K_\mathrm{bath}$ is modified to simulate this behavior, thus constructing three differentiated test cases and analyzing the evolution of the resulting potential values for $K_\mathrm{bath} = 4mM$ and $K_\mathrm{bath}=8mM$. We consider three configurations: the first two are characterized by a region of gray matter tissue in which the value of the potassium stabilization variable is constant over the entire domain, namely $K_\mathrm{bath}=4mM$ and $K_\mathrm{bath}=8mM$, respectively. The third one is constructed such that the unstable gray matter region (encoded into the initial condition) coincides with the region of high potassium concentration, as shown in Figure \ref{table:domain} (a).
 By defining $\Omega_0$ as the subsection of unstable grey matter, the domain is defined as in Table \ref{table:domain}.

\begin{figure}[ht]
\begin{subfigure}[b]{0.33\textwidth}
\centering
\begin{tabular}{ll}
    \toprule
   \textbf{Variable} &\textbf{Subregion}  \\
\midrule
    {$\Omega$} & (0,1)x(0,1) \\
    {$\Omega_{0}$} & (0,0.4)x(0.4,0.4)  \\
\bottomrule
\end{tabular}
\vspace{9mm}
\caption{Distinction of different tissues.}
\label{table:domain}
\end{subfigure}
\begin{subfigure}[b]{0.33\textwidth}
\centering
\includegraphics[scale=0.115]{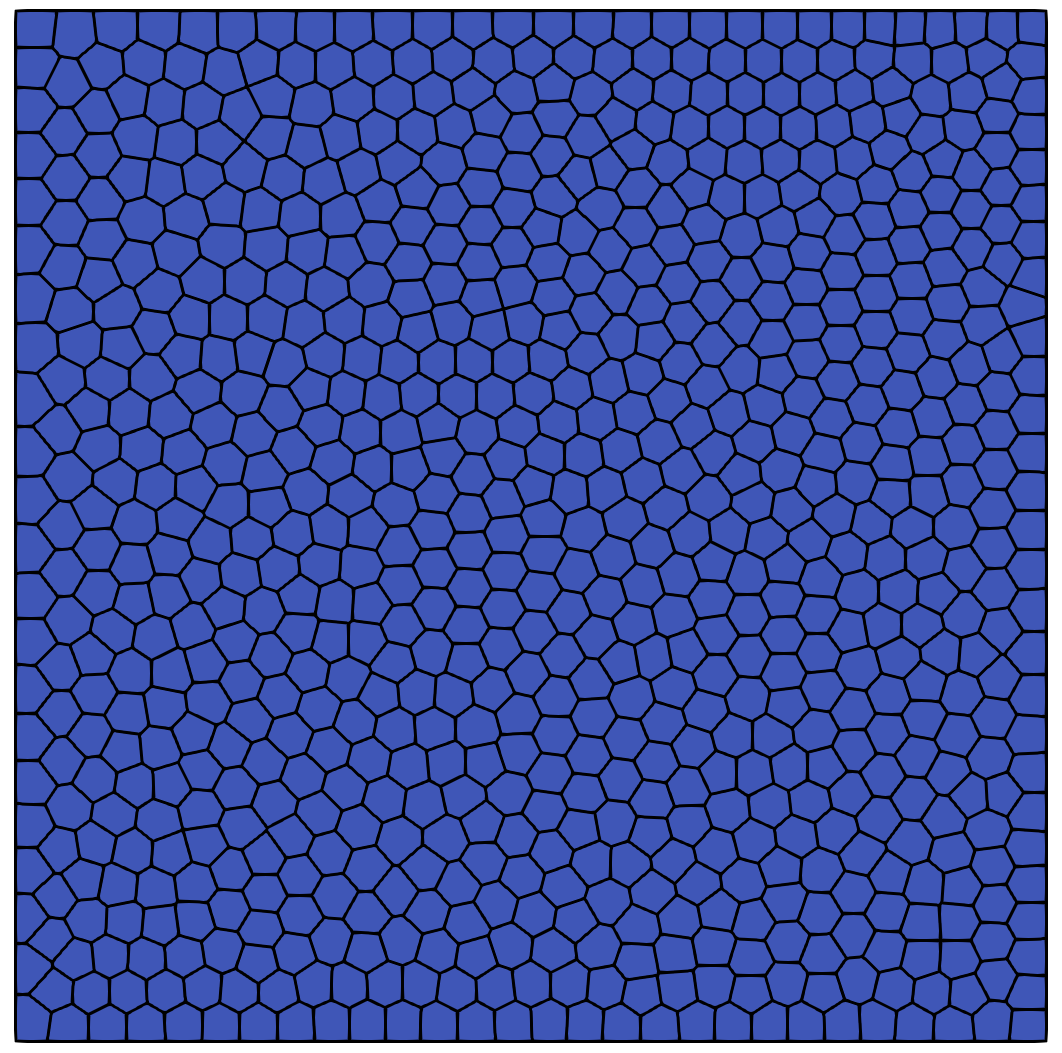} \caption{Polytopal grid} \label{fig:comput_domain_square2}
\end{subfigure}\hfill
\begin{subfigure}[b]{0.33\textwidth}
\centering
\includegraphics[scale=0.115]{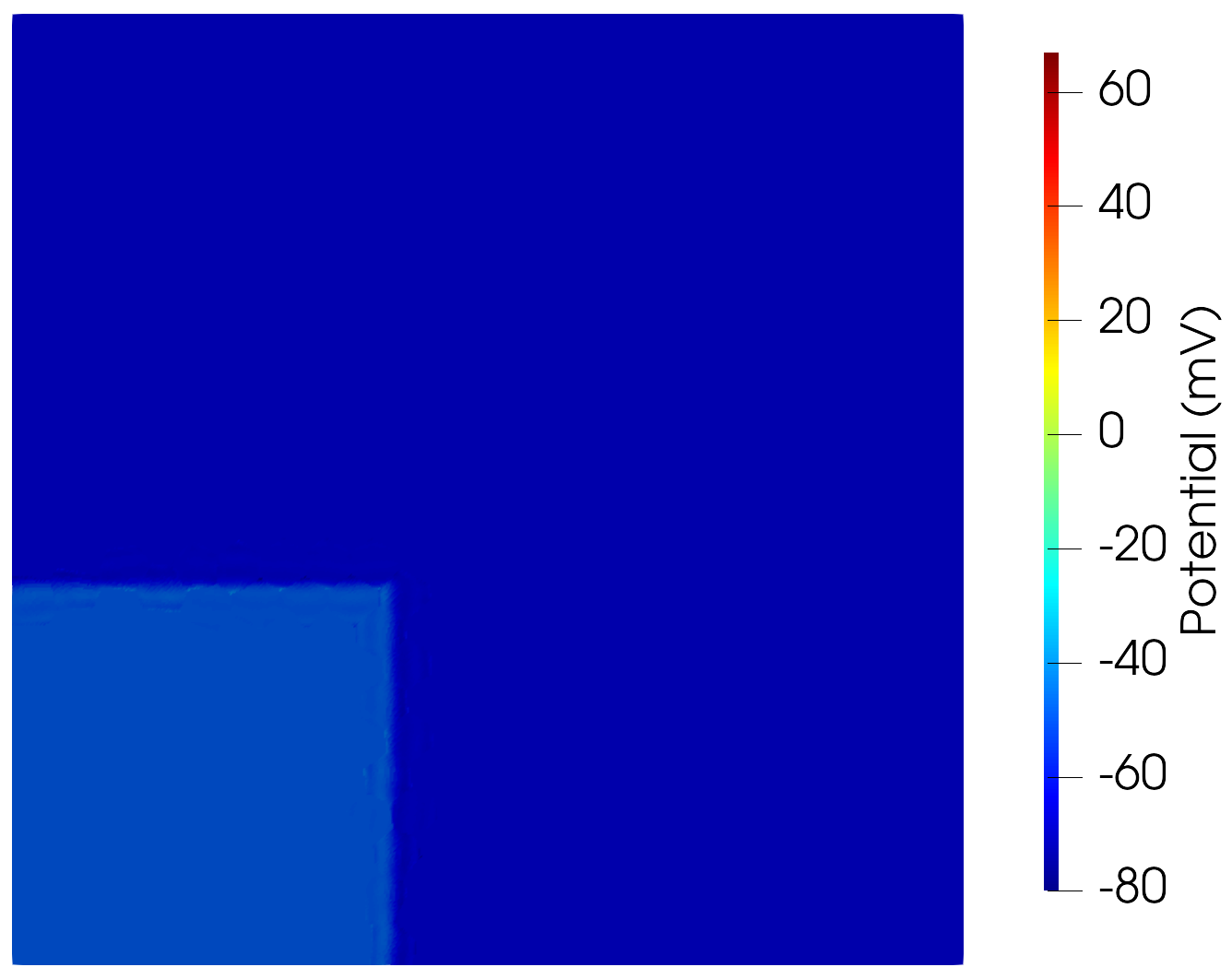} \caption{Initial condition for $u$} \label{fig:comput_domain_IC2}
\end{subfigure}
    \caption{Test case 1.1: (\ref{fig:comput_domain_square2}) computational domain and corresponding grid, grey matter region (blue), (\ref{fig:comput_domain_IC2}) initial condition of the transmembrane potential.}
    \label{fig:mesh_brain_bc2}
\end{figure}

\begin{figure}[ht]
\begin{subfigure}[b]{\textwidth}
    \centering
    \includegraphics[scale=0.28]{photos/scale.png}
    \end{subfigure}\hfill
    \begin{subfigure}[b]{0.33\textwidth}
    \centering
\includegraphics[scale=0.125]{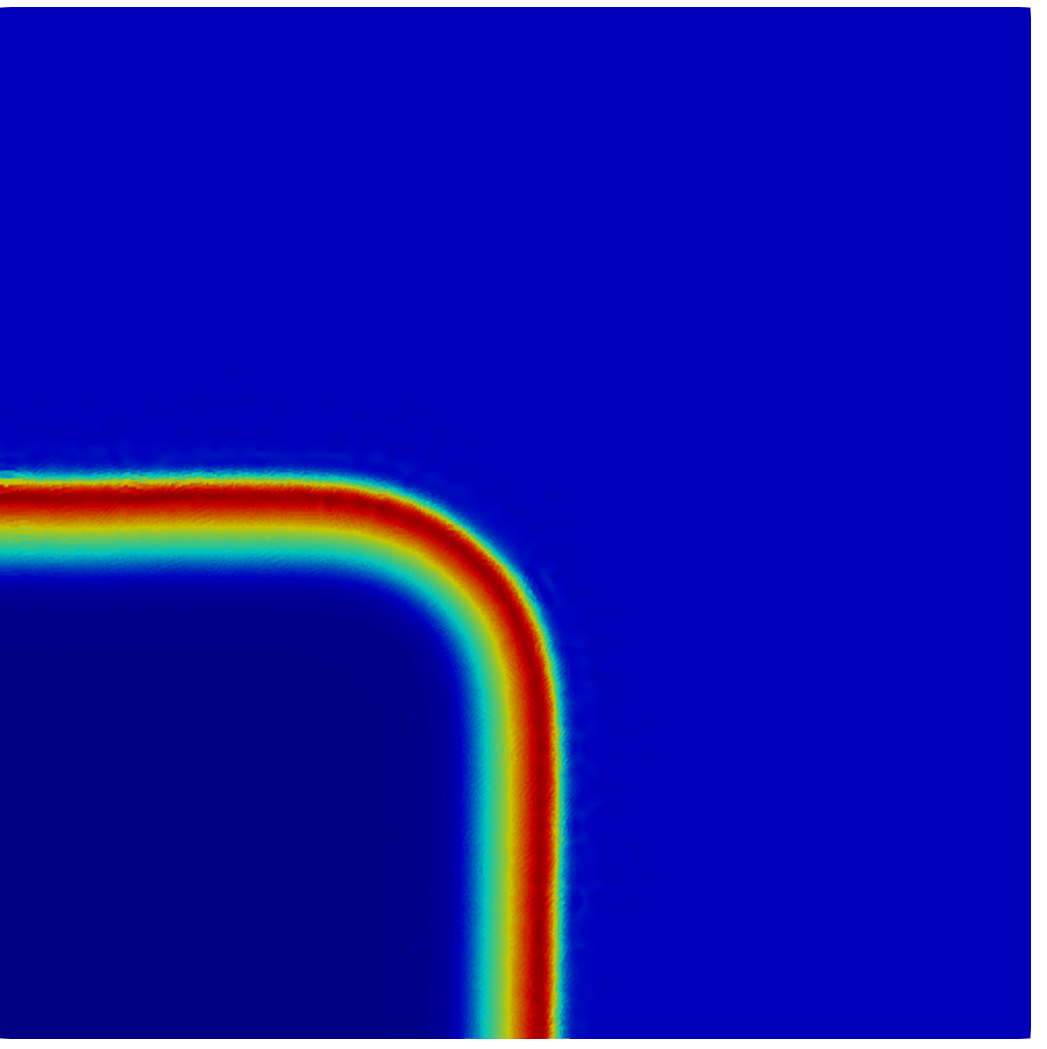} \caption{$t = 3 \mathrm{ms}$}  \label{fig:comput_domain_IC}
  \end{subfigure}\hfill
      \begin{subfigure}[b]{0.33\textwidth}
      \centering
\includegraphics[scale=0.125]{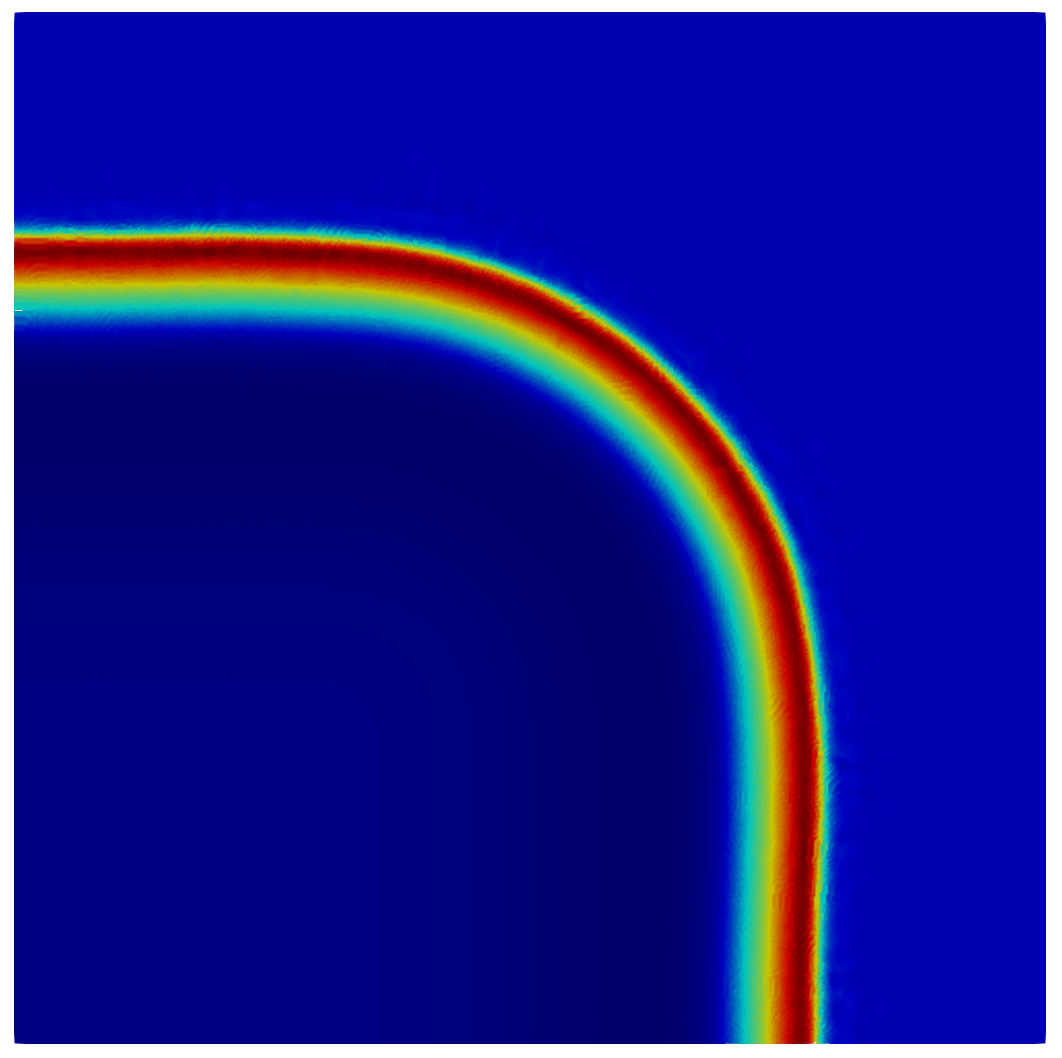} \caption{$t = 6 \mathrm{ms}$}  \label{fig:comput_domain_IC}
  \end{subfigure}\hfill
      \begin{subfigure}[b]{0.33\textwidth}
      \centering
\includegraphics[scale=0.125]{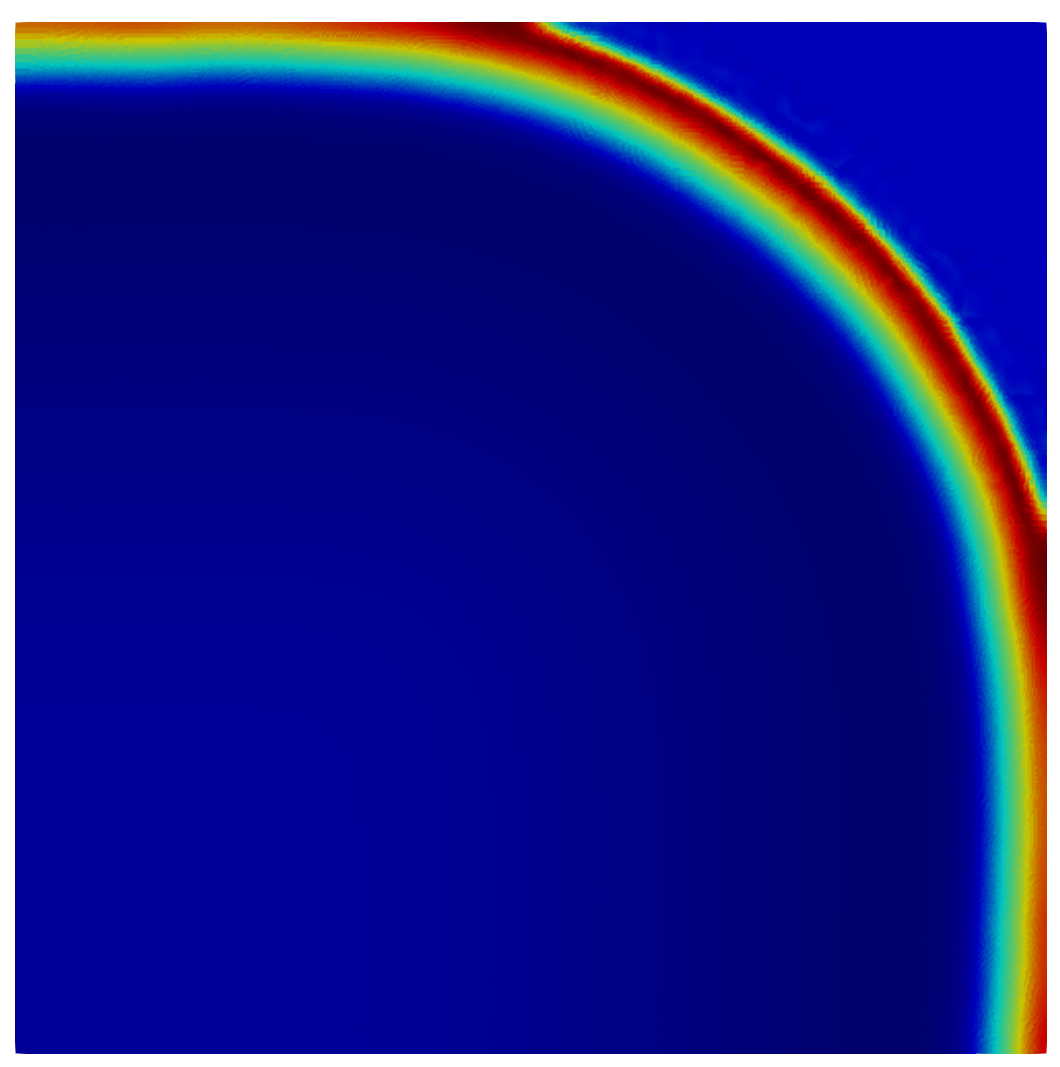} \caption{$t = 9 \mathrm{ms}$}  \label{fig:comput_domain_IC}
  \end{subfigure}
  \caption{Test case 1.1: Snapshots of the transmembrane potential in gray matter tissue at three different time instants with $K_\mathrm{bath}=4 \;\mathrm{mM}$.}
\end{figure}
To appreciate the influence of potassium on a seizure episode, we show in Figure \ref{fig:K} the evolution of the transmembrane potential profile in $(\hat{x},\hat{y})= (0.7,0.7)$. Simulations are performed considering a domain $\Omega = (0,1)^2$ while the polygonal mesh is characterized by 500 elements. For all test cases, fifth-degree polynomials ($p=5$) are exploited for discretization while $\Delta t = 2.5e^{-3}$.
In Figure \ref{fig:k4}, a simulation is presented in which $K_\mathrm{bath}=4mM$ is imposed throughout the domain.
In this setting, a high-frequency drastic change in the potential is observed during the first milliseconds, followed by a regularization process leading to a steady state. The potential finally reaches an asymptotic value, indicating the end of the self-induced episode. 
In Figure \ref{fig:k8}, instead, we present a simulation where the value of $K_\mathrm{bath}$ has been increased up to $8mM$ throughout the whole domain. In this setting, we observe a different behavior of the transmembrane potential evolution from the previous test case, with a larger number of self-induced activations characterized by an increasing frequency. Finally, in Figure \ref{fig:k48}, we present a simulation where the value of $K_\mathrm{bath}$ is defined as a function of the unstable region $\Omega_{0}$. Numerical results show that the unstable region with $K_\mathrm{bath}=8mM$ is dominant for the characterization of the potential dynamics also in regions not characterized by this condition ($\Omega \setminus \Omega_{0}$ with $K_\mathrm{bath}=4mM$). Specifically, this simulation shows that the behavior of the transmembrane potential is close to the one of Figure \ref{fig:k8}. 
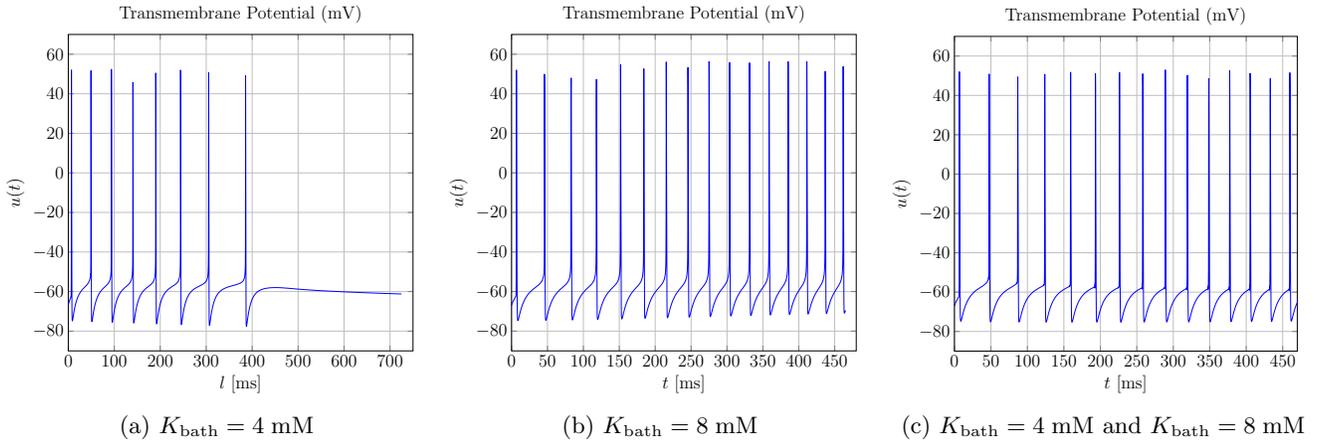
\begin{figure}[h!]
\begin{subfigure}[b]{0.33\textwidth}
 \resizebox{0.95\textwidth}{!}{
\begin{tikzpicture}{
      \begin{axis}[
        width=3.875in,
        height=3.56in,
        at={(2.6in,1.099in)},
        scale only axis,
        %xmode=log,
        xmin=-0.05,
        xmax=750,
        xminorticks=true,
        xlabel = { $l$ [ms]},
        ylabel = { $u(t)$},
        %log x ticks with fixed point,
        %ymode=log,
        ymin=-90,
        ymax=70,
        yminorticks=true,
        axis background/.style={fill=white},
        title={\color{black}  Transmembrane Potential (mV)},
        xmajorgrids,
        xminorgrids,
        ymajorgrids,
        yminorgrids,
        legend style={legend cell align=left, align=left, draw=white!15!black}
        ]
        \addplot        table[x=t,y=v, col sep=comma,mark=none] {csv/sol_k4.csv};
      \end{axis}}
    \end{tikzpicture}}
     \caption{$K_\mathrm{bath} = 4 \;\mathrm{mM} $} \label{fig:k4}
\end{subfigure}\hfill
\begin{subfigure}[b]{0.33\textwidth}
 \resizebox{0.95\textwidth}{!}{
     \begin{tikzpicture}{
      \begin{axis}[
        width=3.875in,
        height=3.56in,
        at={(2.6in,1.099in)},
        scale only axis,
        %xmode=log,
        xmin=-0.05,
        xmax=480,
        xminorticks=true,
        xlabel = { $t$ [ms]},
        ylabel = { $u(t)$},
        %log x ticks with fixed point,
        %ymode=log,
        ymin=-90,
        ymax=70,
        yminorticks=true,
        axis background/.style={fill=white},
        title={\color{black} Transmembrane Potential (mV)},
        xmajorgrids,
        xminorgrids,
        ymajorgrids,
        yminorgrids,
        legend style={legend cell align=left, align=left, draw=white!15!black}
        ]
        \addplot
        table[x=t,y=v, col sep=comma,mark=none] {csv/k8.csv}; 
        %\legend{Plot}
      \end{axis}}
    \end{tikzpicture}}
     \caption{$K_\mathrm{bath} = 8 \;\mathrm{mM} $} \label{fig:k8}
    \end{subfigure}\hfill
    \begin{subfigure}[b]{0.33\textwidth}
  \resizebox{0.95\textwidth}{!}{
    \begin{tikzpicture}{
      \begin{axis}[
        width=3.875in,
        height=3.56in,
        at={(2.6in,1.099in)},
        scale only axis,
        %xmode=log,
        xmin=-0.05,
        xmax=470,
        xminorticks=true,
        xlabel = { $t$ [ms]},
        ylabel = { $u(t)$},
        %log x ticks with fixed point,
        %ymode=log,
        ymin=-90,
        ymax=70,
        yminorticks=true,
        axis background/.style={fill=white},
        title={\color{black}  Transmembrane Potential (mV)},
        xmajorgrids,
        xminorgrids,
        ymajorgrids,
        yminorgrids,
        legend style={legend cell align=left, align=left, draw=white!15!black}
        ]
        \addplot
        table[x=t,y= v, col sep=comma,mark=none] {csv/k48.csv};
        %\legend{Plot}
      \end{axis}}
    \end{tikzpicture}}
     \caption{$K_\mathrm{bath} = 4 \;\mathrm{mM} $ and $K_\mathrm{bath} = 8 \;\mathrm{mM} $ } \label{fig:k48}
    \end{subfigure}
\caption{Test case 1.1: Evolution of the transmembrane potential for different values of $K_\mathrm{bath}$ in a specific point of the domain $(\hat{x},\hat{y}) = (0.7,0.7)$.}
\label{fig:K}
    \end{figure}
    
\subsection{Test case 2: brain section and pathological coefficients}
For the second test case, we consider a mesh of the sagittal section of the brain constructed starting from structural Magnetic Resonance Images (MRI) from the OASIS-3 database \cite{lamontagne2019oasis}. The method exploited to construct the mesh \cite{pietro2020hybrid} exploits the agglomeration of triangular meshes to construct polytopal meshes, which can optimally handle complex geometries. 
Taking into account only the brainstem of the section, we simulate the evolution of the transmembrane potential by imposing unstable initial conditions in a grey matter region. The exploited mesh is reported in Figure~\ref{fig:mesh_brain} (left), where we show the subdivision of grey and white matter. In Figure~\ref{fig:mesh_brain}, we also report the unstable grey matter portion of the section from which the seizure starts. As an initial condition for the transmembrane potential, by defining $\Omega_0$ as the subsection of unstable grey matter, we adopt the following analytical function:
\begin{align*}
u(0) =  -67 + 17e^{-2(x - x_0)^2 - 2(y-y_0)^2} \chi_{\Omega_0}(x,y),
\end{align*}
where $\chi$ is the indicator function. The considered brain section has a maximum longitudinal length of $7.5\,\mathrm{cm}$ and a vertical length of $7\,\mathrm{cm}$. The mesh is characterized by $8476$ elements, of which $3927$ identify white matter and the remaining grey matter. Figure~\ref{fig:BC_central} shows the evolution of the transmembrane potential originated in the unstable grey matter region. 
Due to the strong anisotropy of the white matter \cite{fields2010neuroscience}, there is a scattering of the anomalous signal consistent with what was previously analyzed on the square mesh. The white matter regions within the section slow down the signal, favoring the horizontal to vertical direction. Also, in this test case, fifth-degree polynomials are exploited for the discretization, and $\Delta t = 2.5e^{-3}$ for the time discretization.

\begin{figure}[h!]
\begin{subfigure}[b]{0.3\textwidth}
\centering    \includegraphics[scale=0.17]{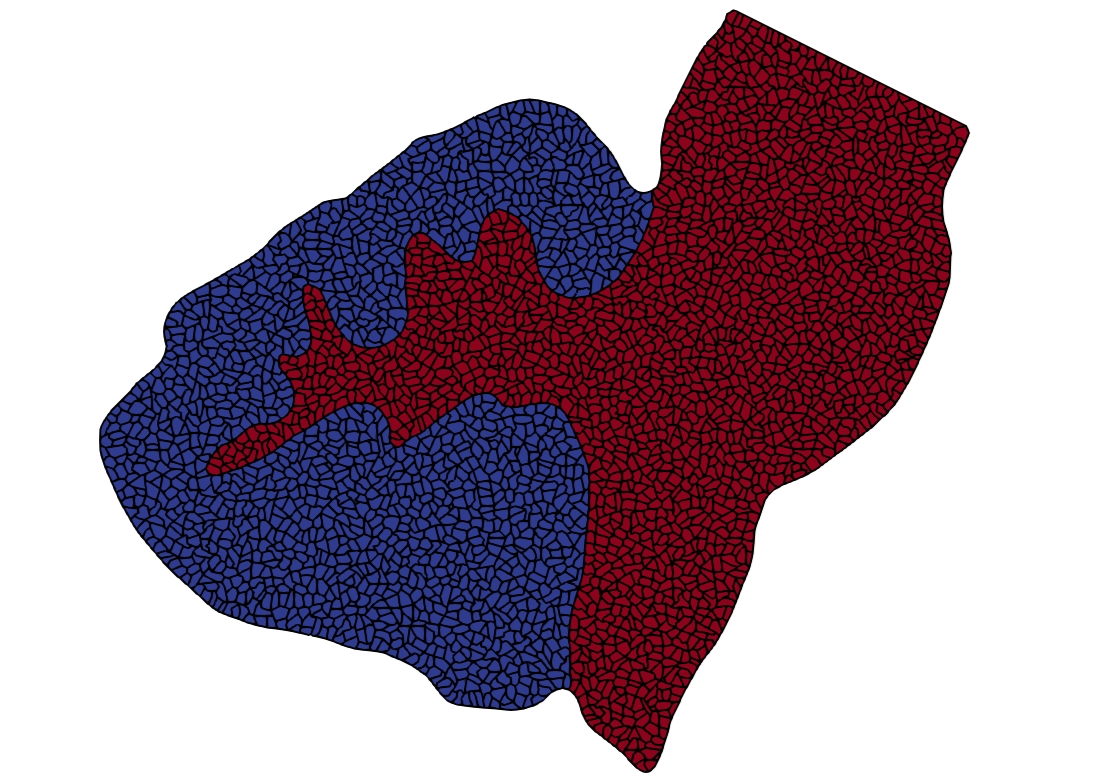}  \caption{Polytopal grid} \label{fig:comput_domain_brain}
\end{subfigure}\hfill
\begin{subfigure}[b]{0.3\textwidth}
\centering
  \includegraphics[scale=0.17]{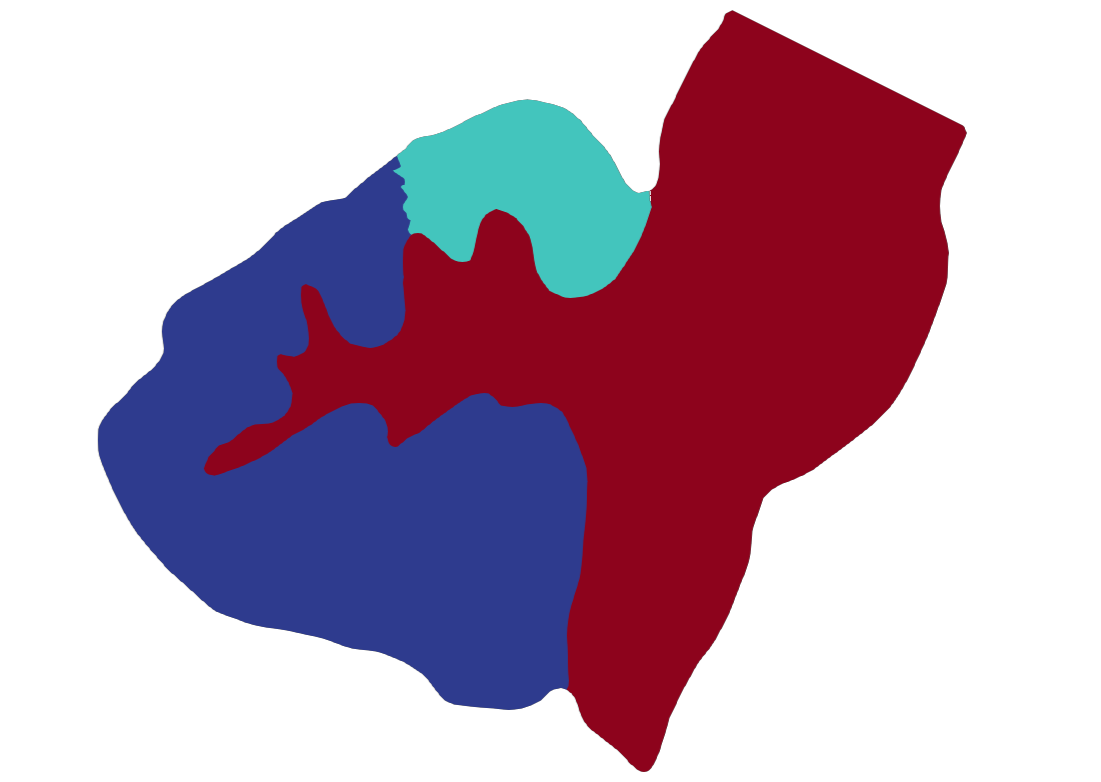} \caption{Tissue differentiation} \label{fig:comput_domain_IC_poly_brain}
\end{subfigure}\hfill
\begin{subfigure}[b]{0.3\textwidth}
\centering
    \includegraphics[scale=0.11]{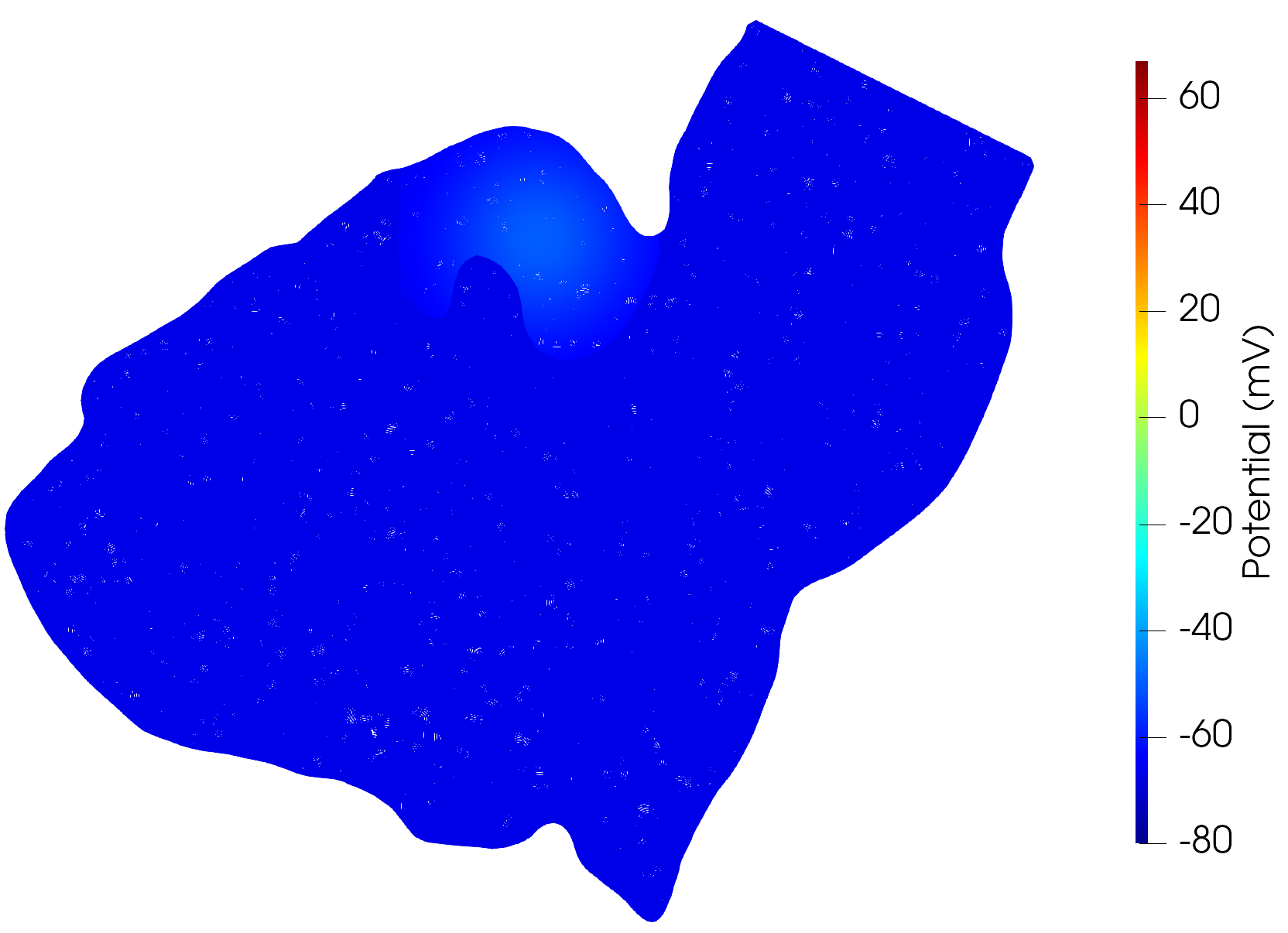} \caption{Initial condition for $u$} \label{fig:comput_domain_IC_brain}
\end{subfigure}
    \caption{Test case 2: (\ref{fig:comput_domain_brain}) computational domain and corresponding grid, (\ref{fig:comput_domain_IC_poly_brain})  initial unstable grey matter region (light blue), stable grey matter region (blue), vertical anisotropic white matter region (red), (\ref{fig:comput_domain_IC_brain}) initial condition of the transmembrane potential.}
    \label{fig:mesh_brain}
\end{figure}
\vspace{-.7cm}
\begin{figure}[h!]
    \centering
    \begin{subfigure}[b]{\textwidth}
    \centering
    \includegraphics[scale=0.28]{photos/scale.png}
    \end{subfigure}\hfill
     \vspace{-.2cm}
    \begin{subfigure}[b]{0.33\textwidth}
    \centering
        \includegraphics[scale=0.14]{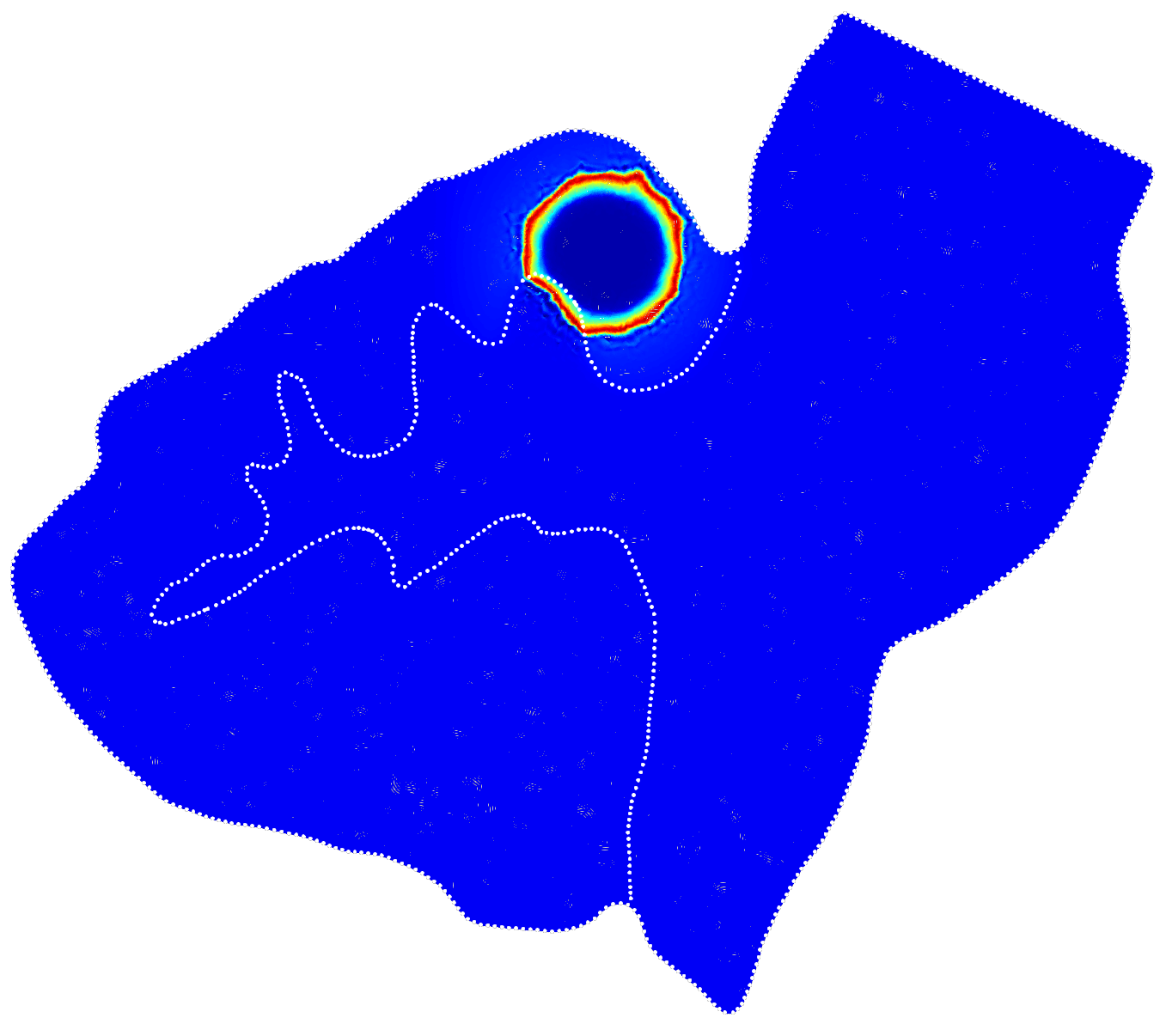} \label{fig:1}
        \caption{$t=4 \;\mathrm{ms}$}
    \end{subfigure}\hfill
         \begin{subfigure}[b]{0.33\textwidth}
      \centering
         \includegraphics[scale=0.14]{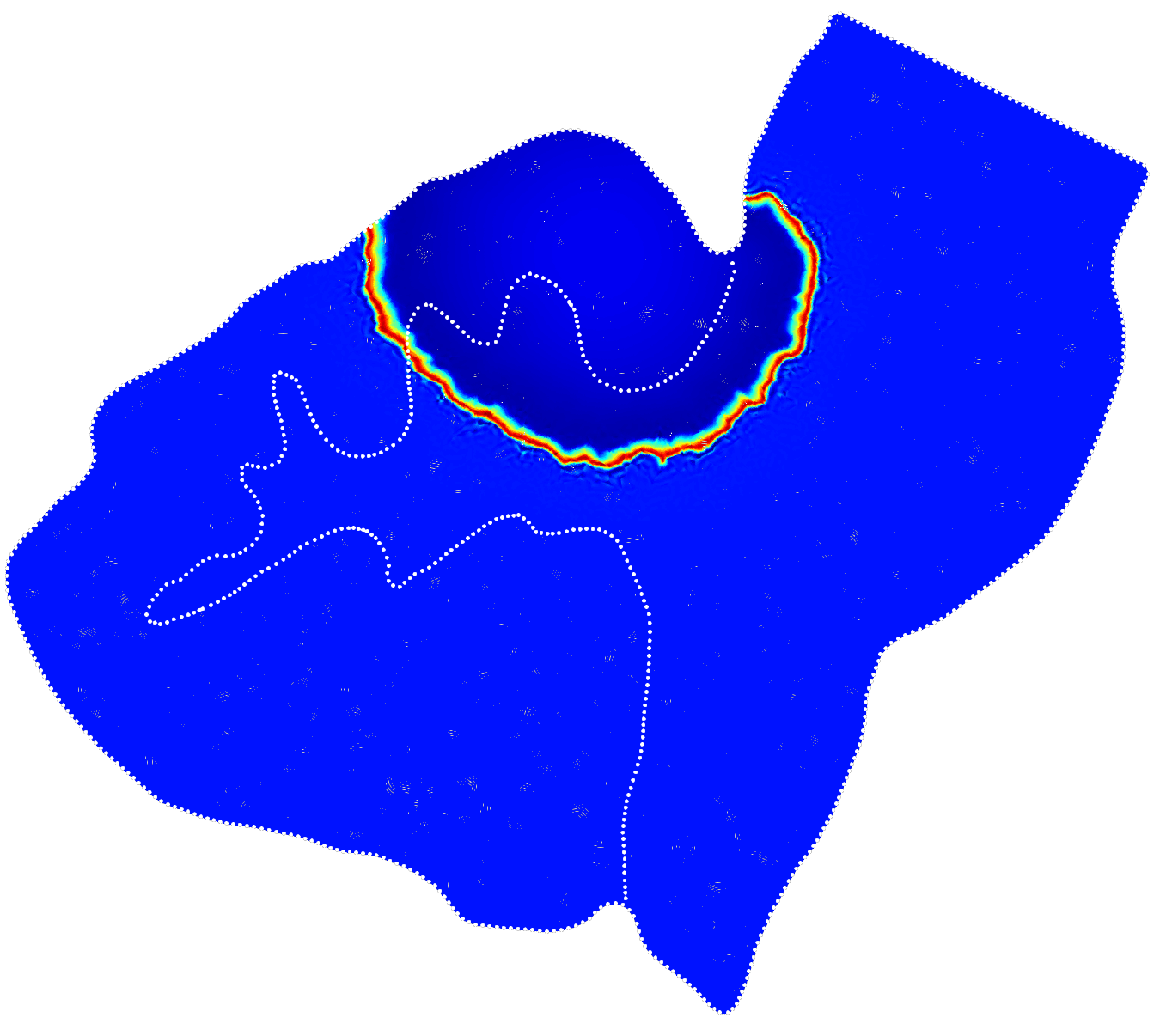}
         \label{fig:2}
        \caption{$t=10\;\mathrm{ms}$}
    \end{subfigure}\hfill
             \begin{subfigure}[b]{0.33\textwidth}
      \centering
         \includegraphics[scale=0.14]{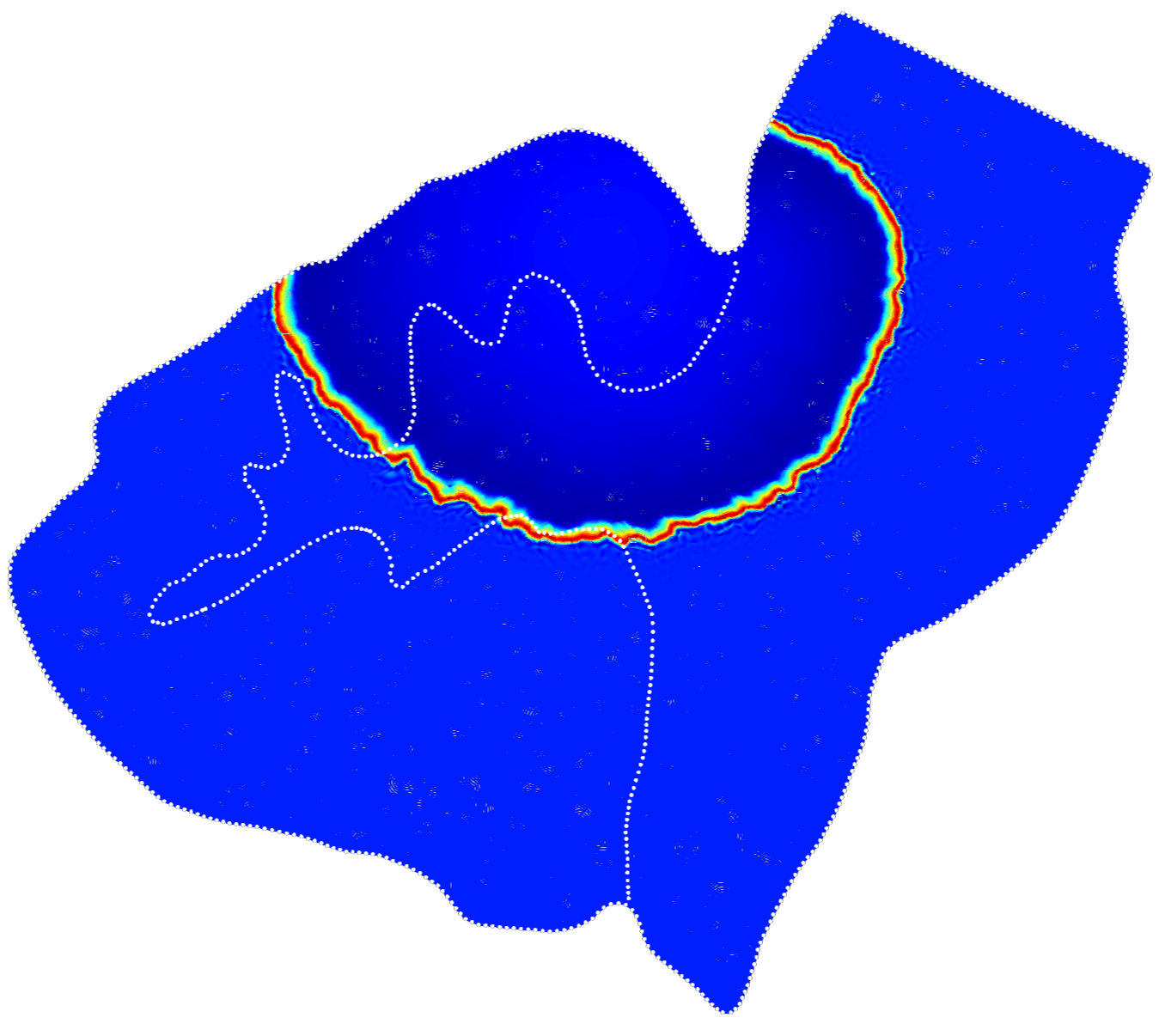}\label{fig:3}
        \caption{$t=16\;\mathrm{ms}$}
    \end{subfigure}\hfill
     \begin{subfigure}[b]{0.33\textwidth}
      \centering
         \includegraphics[scale=0.14]{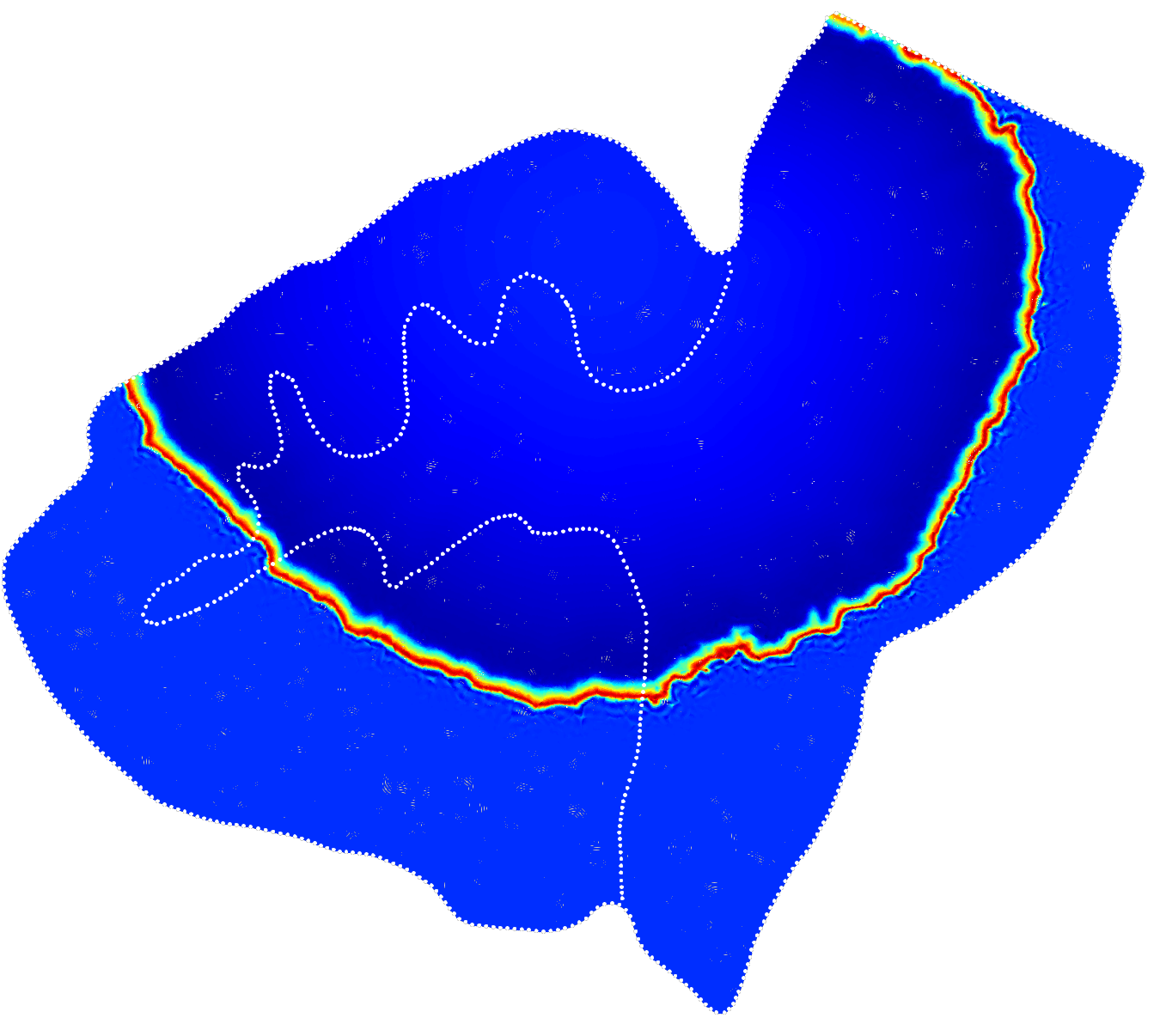}\label{fig:4}
        \caption{$t=24\;\mathrm{ms}$}
    \end{subfigure}\hfill
     \begin{subfigure}[b]{0.33\textwidth}
      \centering
         \includegraphics[scale=0.14]{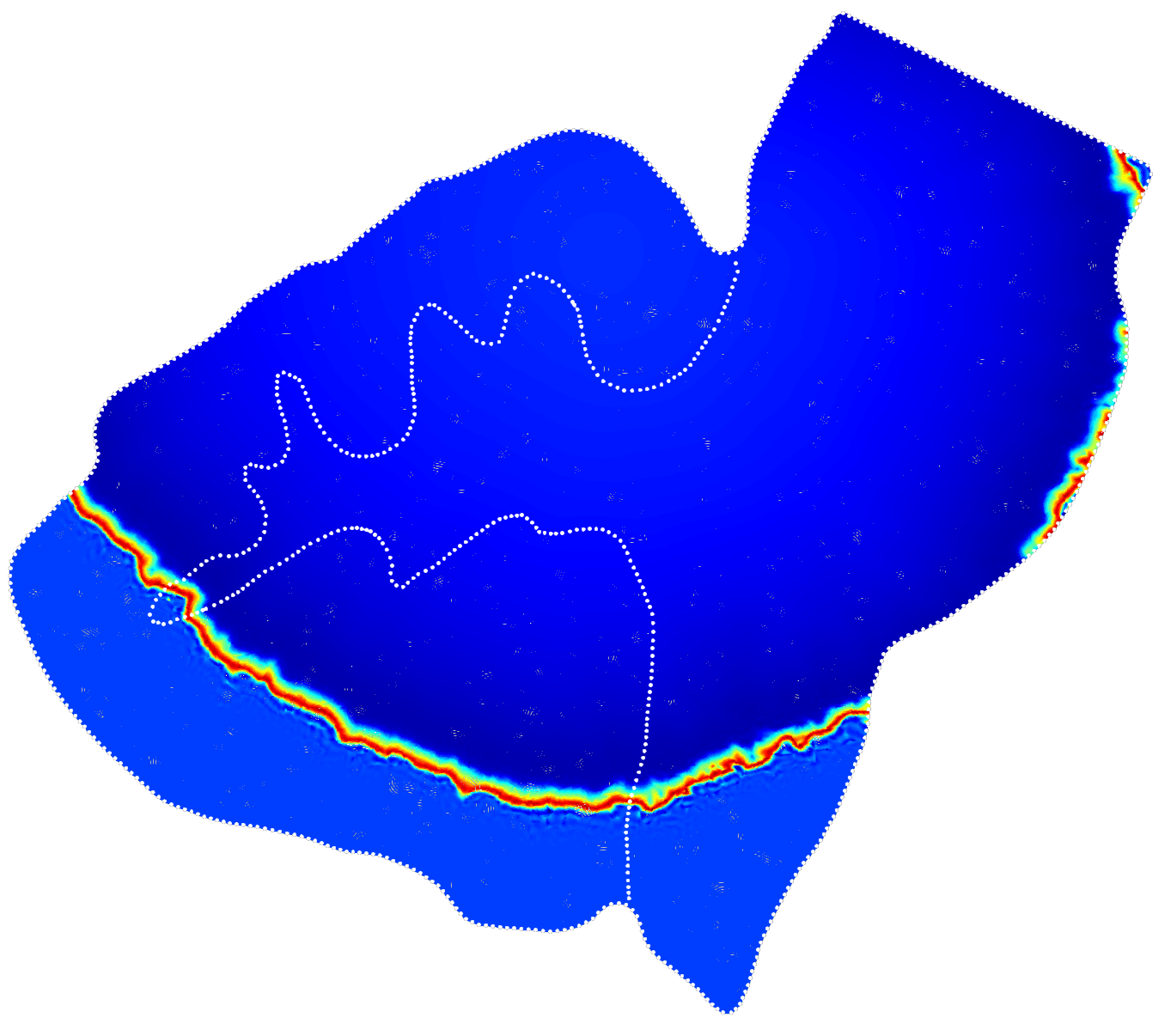}\label{fig:5}
        \caption{$t=34\;\mathrm{ms}$}
    \end{subfigure}\hfill
    \begin{subfigure}[b]{0.33\textwidth}
      \centering
         \includegraphics[scale=0.14]{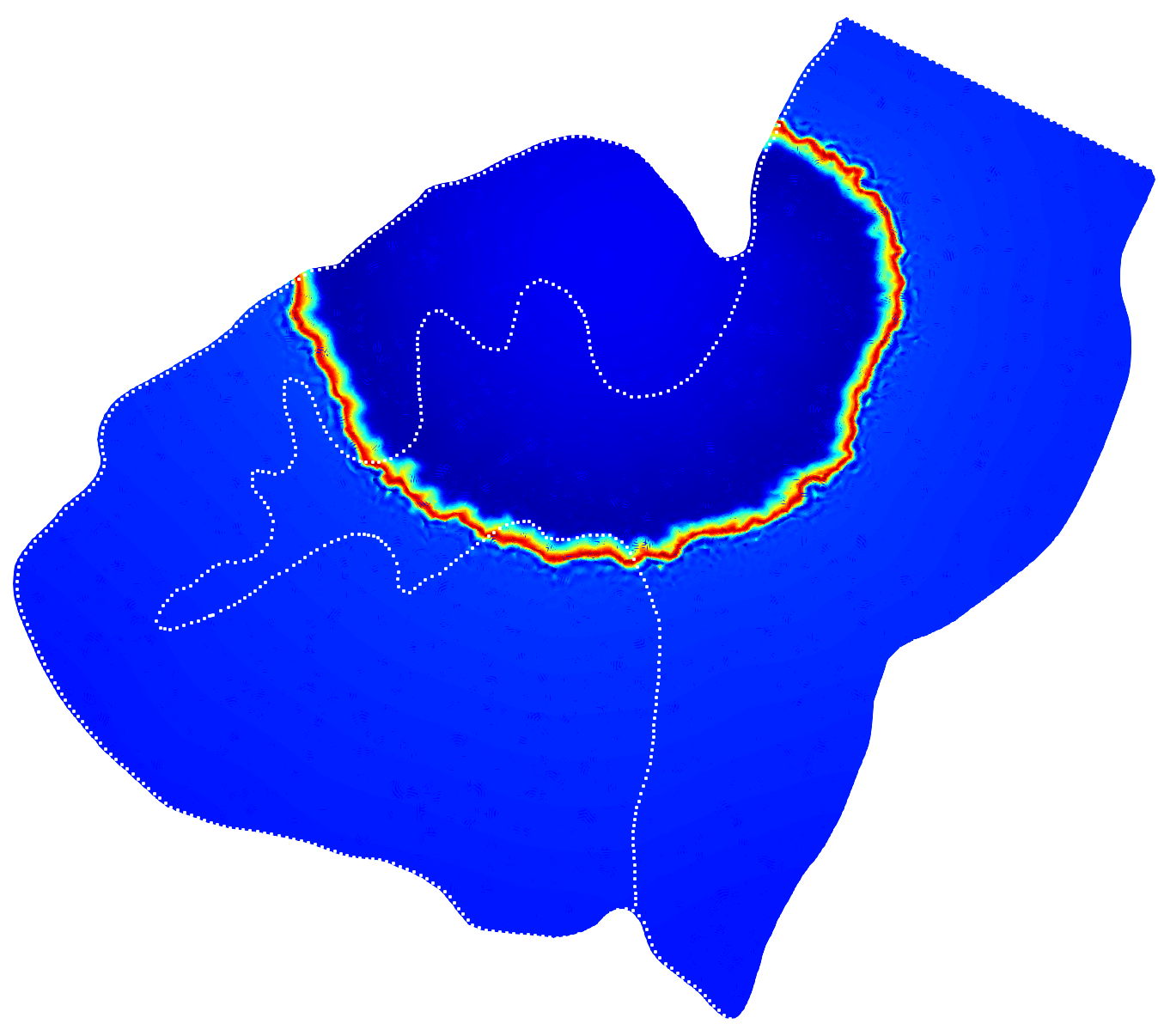}\label{fig:6}
        \caption{$t=70\;\mathrm{ms}$}
    \end{subfigure}
    \caption{Test case 2: snapshots of the transmembrane potential evolution in the sagittal section of the brain.}
    \label{fig:BC_central}
\end{figure}

\section{Theoretical analysis}
\label{sec:6}
\noindent In this section, we present the theoretical analysis of the stability and an \textit{a-priori} convergence
analysis for a benchmark test case. Specifically, we consider the monodomain equation with homogeneous parameters, no forcing currents, and a ionic term that depends on the transmembrane potential in the following non-linear way:
\begin{equation}
    \begin{aligned}
      f(u) = a(u - V_{\text{rest}})(u - V_{\text{thres}})(u - V_{\text{depol}}),
\end{aligned}
\label{eq::F_analitical}
\end{equation}
where $V_\text{rest} \le  V_\text{thres} \le V_\text{depol}$ and $a>0$. Here, $V_\text{rest}$ is the resting potential, $V_\text{thres}$ represents the threshold potential, and $V_\text{depol}$ represents the depolarization potential. Starting from \cite{pezzuto2016space}, where the convergence for the one-dimensional problem is analyzed, the solution is extended to the two-dimensional case by introducing $u_\mathrm{ex}$ as the exact solution of the problem described in the Equation \eqref{eq:one-dimensional} with $f$ defined as \eqref{eq::F_analitical}. The analytic solution is therefore defined as follows.
\begin{equation}
    \begin{aligned}
      u_\mathrm{ex}(\boldsymbol{x},t) = \frac{V_{\text{dep}}-V_{\text{rest}}}{2}\left[1- \tanh\left(\frac{\boldsymbol{x}-\boldsymbol{c}t}{\epsilon}\right)\right] + V_{\text{rest}},
\end{aligned}
\label{eq:one-dimensional}
\end{equation}
where $\epsilon$ characterizes the thickness of the wavefront, while $\boldsymbol{c}$ is the propagation speed of the wave.
\begin{assumption} (Regularity of the coefficients): Concerning the physical parameters of the specific model, we assume the following regularity requirements: $V_{\mathrm{rest}}, V_{\mathrm{depol}}, V_{\mathrm{thres}} \in L^{\infty}(\Omega)$, $a \in L^{\infty}_+(\Omega)$.
\label{as:reg_analytical}
\end{assumption}
\noindent By setting 
\begin{equation} 
    \begin{aligned}
         r_{\text{ion}}(u(t),v) = ( a u(t)^3,v)_{\Omega} &- (a ( V_{\text{thres}} +  V_{\text{depol}} + V_{\text{rest}})   )u(t)^2,v)_{\Omega} + (aV_{\text{thres}}V_{\text{depol}}V_{\text{rest}},v)_{\Omega} - \\& - ( a ( V_{\text{thres}}V_{\text{depol}} +  V_{\text{depol}}V_{\text{rest}} + V_{\text{rest}}V_{\text{thres}})) u(t),v)_{\Omega}  , \qquad \forall v \in V.
    \end{aligned}
    \label{eq:non_linear}
\end{equation}
\noindent Using the definitions introduced in Section \ref{sec:3} and exploiting Assumption \ref{as:mesh}, it is possible to rewrite the semi-discrete formulation for the problem \eqref{eq:monodomain} exploiting \eqref{eq::F_analitical} as follows:
\vspace{2mm}\\
\noindent $\forall t \in (0,T]$ find $u_h(t) \in V_h^{\text{DG}} \text{ such that }:$ 
\begin{equation}
 \left\{ 
    \begin{aligned}
        &\chi_m C_m  \left(\frac{\partial u_h(t)}{\partial t},v_h\right)_{\Omega} + \mathcal{A}(u_h(t),v_h) +  \chi_m r_{\text{ion}}(u_h(t),v_h) = 0  &\forall & v_h \in V_h^{\text{DG}}, \\
        &u_h(0) = u^0, &\text{in }& \Omega.
    \end{aligned}
    \right.
    \label{eq::weak_analitical}
\end{equation}
where $r_{\mathrm{ion}}$ is defined as in \eqref{eq:non_linear}. By exploiting all the definitions in Section \ref{sec:4}, we can introduce the following formulation for the ionic term in the fully discrete formulation, where $[\mathbf{I}]^n_j = (\mathrm{I_{ion}}^n,\varphi_j)_{\Omega}, \quad  j = 1,...,N_h$:
\begin{equation*}
    \begin{aligned}
    \mathrm{I}^n_\mathrm{ion} &= (a (u^3)^n,v)_{\Omega} - (a ( V_{\text{thres}} +  V_{\text{depol}} + V_{\text{rest}})   )(u^2)^n,v)_{\Omega} +(aV_{\text{thres}}V_{\text{depol}}V_{\text{rest}},v)_{\Omega} - \\& - ( a ( V_{\text{thres}}V_{\text{depol}} +  V_{\text{depol}}V_{\text{rest}} + V_{\text{rest}}V_{\text{thres}}))u^n,v)_{\Omega}.
    \end{aligned}
\end{equation*}

\subsection{Stability analysis of the semi-discrete formulation}
\label{sec:7}
\noindent In this section, we present the stability analysis on the semi-discrete formulation of the problem introduced in Equation \eqref{eq::weak_analitical}. We set: 
\begin{equation}
    \begin{aligned}
       &\| {v} \|_\mathrm{DG}^2 =  \|\nabla_h {v}\|^2 +\|\gamma^{\frac{1}{2}}\jumpl{v}\jumpr\|^2_{\mathcal{F}_h^\mathrm{I}} \qquad \forall v\in H^1(K), \; \forall K \in \partition,
    \end{aligned}
    \label{eq::DG_norm}
\end{equation}
The $L^2$-norm on a set of faces $\faces$ will be indicated as $\|\cdot \|_{\faces} = (\sum_{F\in\faces}\| \cdot \|_{L^2(F)}^2)^{1/2}$.
We also introduce the following definition:
\begin{equation}
\label{eq:normDGfull}
\begin{aligned}
        \tnorm v \tnorm_\text{DG} = \|v\|_\text{DG} + \|\eta^{-\frac{1}{2}}\averagel \boldsymbol{\Sigma} \nabla_h v \averager\|_{L^2(\mathcal{F}_h^I)} \quad \forall v \in H^2(\mathcal{T}_h),
\end{aligned}
\end{equation}
where $\eta$ is defined in Equation \eqref{eq:eta} and $H^s((\mathcal{T}_h)$ is the space of piecewise $H^s$ functions, $s\ge0$. We remark that it exists $\hat{C}>0$ such that $\hat{C}\tnorm v \tnorm^2_\text{DG} \le \|v\|^2_\text{DG}$ for all $ v \in  V_h^{\text{DG}}$. 
\begin{proposition}
The bilinear form $\mathcal{A}(\cdot,\cdot)$ is continuous and coercive: 
\begin{align}
         \label{eq::coercivity_1} \exists M \ge 0:& \quad |\mathcal{A}(v_h,w_h)| \le M \tnorm v_h\tnorm _{\mathrm{DG}} \|w_h\|_{\mathrm{DG}} \quad &\forall v_h \in H^2(\partition),\; \forall w_h \in W_{h}^{\mathrm{DG}}\\
         \label{eq::coercivity_2}   \exists \mu \ge 0:& \quad \mathcal{A}(v_h,v_h) \ge \mu \| v_h \|^2_{\mathrm{DG}} \quad &\forall v_h \in W_{h}^{\mathrm{DG}}
   \end{align}
where $M$ and $\mu$ are independent of $h$. Coercivity holds provided that the penalty coefficient appearing in \ref{eq:eta} is chosen large enough. 
\end{proposition}
\noindent Finally, we introduce the definition of the energy norm operator $\|\cdot\|_{\epsilon}: H^1(\mathcal{T}_h) \rightarrow \mathbb{R}$ as follows:
\begin{equation}
\begin{aligned}
  \|v\|^2_{\epsilon} = \|v\|^2 + \int_{0}^t \frac{2\mu}{C_m \chi_m}\|v\|^2_{\mathrm{DG}}\mathrm{d}s + \int_0^t \frac{a}{C_m}\|v\|^4_{L^4}\mathrm{d}s \quad \forall v\in H^1(\partition),
\end{aligned}
\label{eq::energy_norm}
\end{equation}
where $\chi_m,C_m,a >0$.
\begin{theorem} \label{theorem:theorem1}
Let Assumption \ref{as:mesh} and \ref{as:reg_analytical} be satisfied and let $u_h(t)$ be the solution of the Equation \eqref{eq::weak_analitical} for any $t \in (0, T]$. Let the stability parameter $\eta$ be large enough. Then the following stability estimate holds: 
\begin{equation}
\begin{aligned}
  \|u_h(t)\|^2_{\epsilon} \lesssim \|u_h(0)\|^2 + a\dfrac{\|\omega\|_{L^{\infty}(\Omega)}^2+\|\theta\|_{L^{\infty}(\Omega)}^2+\|\phi\|_{L^{\infty}(\Omega)}^2}{ C_m}|\Omega| t,
\end{aligned}
\label{eq::theorem1}
\end{equation}
where $\omega, \; \theta, \; \phi$ are defined as $\phi =  ( V_{\text{thres}} +  V_{\text{depol}} + V_{\text{rest}}), \; \theta =  ( V_{\text{thres}}V_{\text{depol}} +  V_{\text{depol}}V_{\text{rest}} + V_{\text{rest}}V_{\text{thres}}), \; \omega = V_{\text{thres}}V_{\text{depol}}V_{\text{rest}}$. \;$C_m$ is the membrane capacitance defined as in Table \ref{table:G2} and $a$ is the parameter introduced in \eqref{eq::F_analitical}.
The hidden constant depends on $t$, on $\Omega$, on the coefficients, but it is independent of the discretization parameters.
\end{theorem}

\begin{proof}
Starting from Equation \eqref{eq::weak_analitical} and taking the test function $v_h = u_h(t)$, we obtain: 
\begin{equation}
\label{eq:testedvm}
\begin{aligned}
         \chi_m C_m (\dot{u}_h(t),u_h(t))_{\Omega} + \mathcal{A}(u_h(t),u_h(t)) +   \chi_m r_{ion}(u_h(t),u_h(t))  = F(u_h(t)),
   \end{aligned}
   \end{equation}
where we exploit the notation $\dot{u}_h = \frac{\partial u_h}{\partial t}$. We remind that the nonlinear term is defined as in expression \eqref{eq:non_linear}. To simplify the notation, we introduce the following constants:
\begin{equation}
\begin{aligned}
         \phi &=  ( V_{\text{thres}} +  V_{\text{depol}} + V_{\text{rest}}), \\ 
         \theta &=  ( V_{\text{thres}}V_{\text{depol}} +  V_{\text{depol}}V_{\text{rest}} + V_{\text{rest}}V_{\text{thres}}),\\
         \omega &= V_{\text{thres}}V_{\text{depol}}V_{\text{rest}},
         \label{eq::ato}
\end{aligned}
\end{equation}
where $V_\mathrm{depol},\;V_\mathrm{rest},$ and $V_\mathrm{thres}$ are introduce in \eqref{eq::F_analitical}. The reaction term then reads as follows:
\begin{equation*}
    r_{ion}(u_h(t),u_h(t)) = a\left((u_h(t))^3 - \phi (u_h(t))^2 + \theta u_h(t) - \omega\right).
\end{equation*}
Exploiting the coercivity estimate in Equation \eqref{eq::coercivity_2}, integrating in time the Equation \eqref{eq:testedvm}, H\"older inequality and Assumption \ref{as:reg_analytical}, we obtain:
\begin{equation*}
\begin{aligned}
\label{eq::stabilita}
         \dfrac{C_m}{2 a} & \|u_h(t)\|^2   +  \int_{0}^t \frac{\mu}{\chi_m  a} \|u_h(s)\|^2_{\mathrm{DG}}\mathrm{d}s   + \int_{0}^t \|u_h(s)\|^4_{L^4}\mathrm{d}s +  \le   \dfrac{C_m}{2 a}  \|u_h(0)\|^2 \\ & + \int_{0}^t \underbrace{\|\omega\|_{L^{\infty}(\Omega)}|\Omega|^{3/4} \|u_h(s)\|_{L^4(\Omega)}  \mathrm{d}s}_{\text{(I)}} + \int_{0}^t \underbrace{\|\theta\|_{L^{\infty}(\Omega)}\|u_h(s)\|^2  \mathrm{d}s}_{\text{(II)}} +\int_{0}^t  \underbrace{\|\phi\|_{L^{\infty}(\Omega)} \|u_h(s)\|^3_{L^3(\Omega)}\mathrm{d}s}_{\text{(III)}}
\end{aligned}
\end{equation*}
We next bound each term separately:
\begin{itemize}
\item[(I)] We exploit Young's inequality with $p=4$ and $p^* = 4/3$: 
\begin{equation*}
\begin{aligned}
         |\Omega|^{3/4} \|u_h(s)\|  \le \left(\frac{k_1}{4}\|u_h\|^4_{L^4} + \frac{3}{4k_1}|\Omega|\right),
\end{aligned}
\end{equation*}
where $k_1>0$ can be arbitrarily chosen. 
\item[(II)] We exploit H\"older's inequality with $p=2$ and $p^* = 2$ and Young's inequality with $p=2$ and $p^* = 2$, to obtain: 
\begin{equation*}
\begin{aligned}
         \|u_h(t)\|^2 \le \|u\|^2_{L^4}|\Omega|^\frac{1}{2} \le \left(\frac{k_2}{2}\|u_h\|^4_{L^4} + \frac{1}{2k_2}|\Omega|\right),
\end{aligned}
\end{equation*}
where $k_2>0$ can be arbitrarily chosen. 
\item[(III)] We exploit H\"older's inequality with $p=4/3$ and $p^* = 4$ and Young's inequality with $p=4/3$ and $p^* = 4$ in order to control the cubic term, as:
\begin{equation*}
\begin{aligned}
          \|u_h(t)\|^3_{L^3} \le \|u_h\|^3_{L^4}|\Omega|^\frac{1}{4} \le  \frac{3k_3}{4}\|u_h\|^4_{L^4} + \frac{1}{4k_3}|\Omega|,
\end{aligned}
\end{equation*}
where $k_3>0$ can be arbitrarily chosen. 
\end{itemize}
\noindent Collecting the previous bounds we obtain:
\begin{equation*}
\begin{aligned}
         \dfrac{C_m}{2 a} \|u_h(t)\|^2 + & \int_{0}^t \frac{\mu}{\chi_m  a} \|u_h(s)\|^2_{\mathrm{DG}}\mathrm{d}s   + \int_{0}^t \|u_h(s)\|^4_{L^4}\mathrm{d}s +  \le   \dfrac{C_m}{2 a}  \|u_h(0)\|^2 \\ + & \int_{0}^t \left( \dfrac{\|\omega\|_{L^{\infty}(\Omega)} k_1 + 2\|\theta\|_{L^{\infty}(\Omega)} k_2 + 3 \|\phi\|_{L^{\infty}(\Omega)} k_3}{4}\right)\|u_h(s)\|_{L^4(\Omega)}^4  \mathrm{d}s \\ + & \int_{0}^t \left(\dfrac{3\|\omega\|_{L^{\infty}(\Omega)}}{4 k_1} + \dfrac{\|\theta\|_{L^{\infty}(\Omega)}}{2 k_2} + \dfrac{\|\phi\|_{L^{\infty}(\Omega)}}{4 k_3}\right)|\Omega| \mathrm{d}s.
\end{aligned}
\end{equation*}
We choose:
\begin{equation*}
\begin{aligned}
 k_1 = \frac{2}{3\|\omega\|_{L^{\infty}(\Omega)}}, \qquad k_2 = \frac{1}{3\|\theta\|_{L^{\infty}(\Omega)}},  \qquad k_3=\frac{2}{9\|\phi\|_{L^{\infty}(\Omega)}},
\end{aligned}
\end{equation*}
and obtain the following estimate:
\begin{equation*}
\small
\begin{aligned}
         \dfrac{C_m}{a} \|u_h(t)\|^2 +  \int_{0}^t \frac{2 \mu}{\chi_m  a} \|u_h(s)\|^2_{\mathrm{DG}}\mathrm{d}s   + \int_{0}^t \|u_h(s)\|^4_{L^4}\mathrm{d}s +  \le   \dfrac{C_m}{a}  \|u_h(0)\|^2 +  \dfrac{9\|\omega\|_{L^{\infty}(\Omega)}^2+12\|\theta\|_{L^{\infty}(\Omega)}^2+9\|\phi\|_{L^{\infty}(\Omega)}^2}{4}|\Omega| t.
\end{aligned}
\end{equation*}
Taking the energy norm defined in \eqref{eq::energy_norm} the thesis follows.
\begin{equation*}
\begin{aligned}
  \|u_h(t)\|^2_{\epsilon} \lesssim \|u_h(0)\|^2_{\epsilon} 
  + a\dfrac{\|\omega\|_{L^{\infty}(\Omega)}^2+\|\theta\|_{L^{\infty}(\Omega)}^2+\|\phi\|_{L^{\infty}(\Omega)}^2}{C_m}|\Omega| t.
\end{aligned}
\end{equation*}
\end{proof}   

\subsection{Error analysis of the semi-discrete formulation}
\label{sec:8}
\noindent In this section, we derive a priori error estimate for the solution of the PolyDG semi-discrete problem in  \eqref{eq::weak_analitical}. In order to prove convergence, we assume $I^{\mathrm{ext}}=0$. Interpolating this type of solution, we can have a function $u_I$, which is $L^{\infty}$ by construction, then \cite{babuska1994p}: 
\begin{equation*}
\begin{aligned}
        \exists M_I >0 \quad \|u_I(t)\|^2_{L^{\infty}} \le M_I \quad \forall t \in (0,T). 
\end{aligned}
\end{equation*}
\begin{proposition}
\label{as:estimate}
Let Assumption \ref{as:mesh} be fulfilled. If $d\ge2$, then the following estimates hold:
\begin{equation}
\begin{aligned}
        \forall v \in H^n(\mathcal{T}_h) \quad \exists u_I \in V_h^\text{DG} : \tnorm v - v_I \tnorm^2_\text{DG} \le \sum_{K \in \mathcal{T}_h} h^{2\mathrm{min}\{p+1,n\}-2}\|v\|_{H^{n}(K)}
\end{aligned}
\end{equation}
\end{proposition}
For detailed proof of the proposition, see \cite{cangiani2014hp}.

\begin{theorem} \label{th:theorem2} Let $u_h$ be the solution of \eqref{eq::weak_analitical} for any $ t \in (0,T]$. Let Assumption \ref{as:reg_analytical} be satisfied and let $u$ be the solution of Equation \eqref{eq:monodomain} where $f(u,\boldsymbol{y})$ is defined as in \eqref{eq::F_analitical} for any $t \in (0,T]$ and let assume it satisfies the following additional regularity requirements: $u \in C^1((0,T];H^n(\Omega)\cap L^\infty(\Omega)),
$ for $n \ge 2$. Let $u_h(t) \in V_h$ be the solution of \ref{eq:barreto_weak} for a sufficiently large penalty parameter $\eta$. Then, the following estimate holds:
\begin{equation}
\label{eq:theorem_regularity}
\begin{aligned}
        \tnorm e_h(t)\tnorm_{\epsilon}^2 \lesssim & \sum_{K\in\partition} h_K^{2\min\{p+1,n\}-2}\int_0^t \left(\|\dot{u}(s)\|^2_{H^n(K)} + \|u(s)\|^2_{H^n(K)}\right)\mathrm{d}s \\ + & \sum_{K\in\partition} h_K^{4\min\{p+1,n\}-4}\int_0^t \|u(s)\|^4_{H^n(K)}\mathrm{d}s \\ + & \sum_{K\in\partition} \left(h_K^{2\min\{p+1,n\}-2} \|u(t)\|^2_{H^n(K)} + h_K^{4\min\{p+1,n\}-4} \|u(t)\|^4_{H^n(K)}\right) \quad t \in (0,T],
\end{aligned}
\end{equation}
under the following additional hypothesis of the constants: $\mu-a\,\chi_m\,(M_I\,C_{E_2}+C_S\,C_{E_4})\phi_\infty-a\,\chi_m\,C_{E_2}\theta_\infty > 0$, where $C_{E_q}$ is the discrete Sobolev embedding constant for the $L^q(\Omega)$ space and $C_S$ is defined in Theorem \ref{theorem:theorem1}.
\end{theorem}

\begin{remark}
The requirement of Theorem  \ref{th:theorem2} about the coefficients 
$\mu-a\,\chi_m\,(M_I\,C_{E_2}+C_S\,C_{E_4})\phi_\infty-a\,\chi_m\,C_{E_2}\theta_\infty > 0$ is reasonable considering the values of the constants in our application. 
\end{remark}

\begin{proof}
First of all, we subtract the Equation \eqref{eq::F_analitical} from Equation \eqref{eq::weak_analitical} to obtain:
\begin{equation*}
\begin{aligned}
         \chi_m C_m(\dot{u} - \dot{u}_h,v_h)_{\Omega} + \mathcal{A}(u - u_h,v_h) + \chi_m r_{\text{ion}}(u,v_h) - \chi_m r_{\text{ion}}(u_h,v_h) = 0 \quad \forall v_h \in V_h^\mathrm{DG}. 
\end{aligned}
\end{equation*}
We define the errors $e_h = u_I -u_h$ and $e_I = u - u_I$, where $u_I$ is a suitable interpolant such that $e_h(0) = 0$. 
By testing against $e_h$ we have:
\begin{equation*}
\begin{aligned}
         \frac{\chi_m C_m}{2} \|\dot{e_h}\|^2 + \mathcal{A}(e_h,e_h) +\chi_m r_{\text{ion}}(u(t),e_h) - \chi_m r_{\text{ion}}(u_h(t),e_h) = \chi_m C_m(\dot{e}_I,e_h)_{\Omega} + \mathcal{A}(e_I,e_h), 
\end{aligned}
\end{equation*}
\noindent where $r_{\mathrm{ion}}(u(t),e_h) - r_{\mathrm{ion}}(u_h(t),e_h)$ is defined as:
\begin{equation*}
\begin{aligned}
         r_{\mathrm{ion}}(u,e_h) - r_{\mathrm{ion}}(u_h,e_h) =a(u^3- u_h^3,e_h)_{\Omega} - \phi (u^2- u_h^2,e_h)_{\Omega} + \theta (u- u_h,e_h)_{\Omega}.
\end{aligned}
\end{equation*}
Exploiting the coercivity of the bilinear form in \eqref{eq::coercivity_2}, integrating between $0$ and $t$. We remark that $e_h(0)=0$. Then, for large values for the penalty coefficient $\eta$, we obtain:
\begin{equation*}
\begin{aligned}
     \frac{C_m}{2}&\|e_h(t)\|^2 + \int_{0}^t\frac{\mu}{\chi_m}\|e_h(s)\|^2_{\mathrm{DG}} \mathrm{d}s +  \int_{0}^t  \underbrace{a(u(s)^3- u_h(s)^3,e_h)_{\Omega}}_{\mathrm{(I)}}\mathrm{d}s  \le \int_0^t C_m\|\dot{e}_I\|\|e_h\| \mathrm{d}s  \\ + & \int_{0}^t\frac{1}{\chi_m}|\mathcal{A}(e_I(s),e_h(s))|\mathrm{d}s +\int_{0}^t \underbrace{|a \phi (u(s)^2- u_h(s)^2,e_h)_{\Omega}|}_{\mathrm{(II)}} \mathrm{d}s+\int_{0}^t  \underbrace{|a\theta(u(s)- u_h(s),e_h)_{\Omega}|}_{\mathrm{(III)}}\mathrm{d}s.
\end{aligned}
\end{equation*}
We can treat the nonlinear term (I) by rewriting the difference as follows:
\begin{itemize}
    \item[(I)] We first write: 
\begin{equation*}
\begin{aligned}
         u^3 - u_h^3 &= u^3 - u_I^3 + u_I^3 - u_h^3  \\ 
         &= (u_I -u_h)^3 + (u - (u_I))^3 + 3u_Iu_h(u_I -u_h) + 3u_Iu(u-u_I)  \\ & = (e_h)^3 + (e_I)^3 +3u_Iu_h(e_h) + 3u_Iu(e_I).
\end{aligned}
\end{equation*}
This decomposition gives rise to the following term:
\begin{equation*}
\begin{split}
    (u^3 - u_h^3,e_h)_\Omega = & \left((e_h)^3 + (e_I)^3 +3u_Iu_h(e_h) + 3u_Iu(e_I),e_h\right)_\Omega  \\
    = & \|e_h\|_{L^4(\Omega)}^4 + \underbrace{\left((e_I)^3,e_h\right)_\Omega}_{\mathrm{(a)}} + 3 \underbrace{\left(u_I\,u_h\, e_h, e_h\right)_\Omega}_{\mathrm{(b)}} + 3 \underbrace{\left(u_I\,u\,e_I ,e_h\right)_\Omega}_{\mathrm{(c)}}.
\end{split}
\end{equation*}
We can now treat the terms separately as follows:
\begin{itemize}
    \item[(a)] can be bounded using H\"older's inequality and Young's inequality both with $p = 4/3 $ and $p^* = 4$:
    \begin{equation*}
\begin{aligned}
         |(e_I^3,e_h)_{\Omega}| \le \|e_I\|^{3}_{L^4(\Omega)}\|e_h\|_{L^4(\Omega)} \le  \frac{3}{4k_1}\|e_I\|^{4}_{L^4(\Omega)} + \frac{k_1}{4}\|e_h\|^4_{L^4(\Omega)}.
\end{aligned}
\end{equation*}
    \item[(b)] can be bounded using $L^{\infty}$-bound of the interpolant, the stability estimate of the DG solution of Theorem \ref{theorem:theorem1}, H\"older's inequality, the discrete Sobolev-Poincaré-Wirtinger inequality \cite{pietro2020hybrid}:
\begin{equation*}
\begin{aligned}
         |(u_I\,u_h\,e_h,e_h)_{\Omega}| \le \|u_I\|_{L^{\infty}(\Omega)}\|u_h\|\|e_h\|^2_{L^4(\Omega)} \le M_1 C_S C_{E_4}\|e_h\|_{\mathrm{DG}}^2.
\end{aligned}
\end{equation*}
    \item[(c)] can be bounded using $L^{\infty}$-bound of the interpolant and the continuous solution and the Cauchy-Schwarz inequality:
    \begin{equation*}
\begin{aligned}
         |(u_Iue_I,e_h)_{\Omega}| \le \|u\|_{L^{\infty}}\|u_I\|_{L^{\infty}}\|e_I\|\|e_h\| \lesssim \|e_I\|\|e_h\|.
\end{aligned}
\end{equation*}
\end{itemize}
    \item[(II)] We write:
        \begin{equation*}
\begin{aligned}
         u^2 - u_h^2 = & u^2 - u_I^2 + u_I^2 - u_h^2  \\ 
         = & u^2 - u(u_I) + u(u_I) - u_I^2 + u_I^2 - u_Iu_h + u_Iu_h -  u_h^2  \\ 
         = & u(u - u_I) + u_I (u -u_I) + u_I(u_I - u_h) +u_h(u_I - u_h)  \\ 
         = & \underbrace{ue_I}_{\mathrm{(a)}} + \underbrace{u_I e_I}_{\mathrm{(b)}} + \underbrace{u_Ie_h}_{\mathrm{(c)}} + \underbrace{u_he_h}_{\mathrm{(d)}}.
\end{aligned}
\end{equation*}
\begin{itemize}
    \item (a) can be bounded using $L^{\infty}$-bound of the continuous solution, and H\"older's inequality:
        \begin{equation*}
\begin{aligned}
         |(u\,e_I,e_h)_{\Omega}| = \|u\|_{L^{\infty}(\Omega)}|(e_I,e_h)_{\Omega}| \le M_c \|e_h\|\|e_I\|
\end{aligned}
\end{equation*}
    \item (b) can be bounded using $L^{\infty}$-bound of the interpolant, and H\"older's inequality:
        \begin{equation*}
\begin{aligned}
         |(u_Ie_I,e_h)_{\Omega}| = \|u_I\|_{L^{\infty}(\Omega)}|(e_I,e_h)_{\Omega}| \le M_I\|e_h\|\|e_I\|.
\end{aligned}
\end{equation*}
    \item (c) can be bounded using the $L^\infty$-bound of the interpolant, H\"older's inequality, and the\\ Sobolev–Poincaré–Wirtinger discrete inequality \cite{pietro2020hybrid}:
        \begin{equation*} 
\begin{aligned}
         |(u_I(u_I - u_h),e_h)_{\Omega}| = \|u_I\|_{L^{\infty}(\Omega)}(e_h,e_h)_{\Omega} \le \|u_I\|_{L^{\infty}(\Omega)}\|e_h\|^2 \le M_IC_{E_2}\|e_h\|^2_{\mathrm{DG}}.
\end{aligned}
\end{equation*}
\item (d) can be bounded using H\"older's inequality, and Equation \eqref{eq::theorem1}.
\begin{equation*}
\begin{aligned}
         |(u_h e_h,e_h)_{\Omega}| \leq  \|u_h\|\|e_h\|^2_{L^4(\Omega)}\le C_SC_{E_4}\|e_h\|_{\mathrm{DG}}^2.
\end{aligned}
\end{equation*}
where $C_S$ is the constant defined in Equation \eqref{eq::theorem1}.
\end{itemize}
\item[(III)] can be bounded by means of H\"older's inequality after the error decomposition into $e_I+e_h$.
\end{itemize}
We fix the constant derived from the application of Young's inequality as $k_1 = 2/a$. Then, from the above bounds and by also using the property of DG-norms, we can write:
\begin{equation*}
\begin{aligned}
     \frac{C_m}{2}\|e_h(t)\|^2 + & \int_{0}^t\left(\frac{\mu}{\chi_m}-a\,(M_I\,C_{E_2}+C_S\,C_{E_4})\phi_\infty-a\,C_{E_2}\theta_\infty\right)\|e_h(s)\|^2_{\mathrm{DG}} \mathrm{d}s +  \int_{0}^t  \frac{1}{2} \|e_h(s)\|^4_{L^4(\Omega)}\mathrm{d}s  \le \\ + & \int_0^t  \frac{3a^2}{8}\|e_I(s)\|^{4}_{L^4}\mathrm{d}s + \int_0^t a\,M_I\,C_S\|e_h(s)\|\|e_h(s)\|_{\mathrm{DG}}\mathrm{d}s + \int_{0}^t\frac{a M}{\chi_m}\tnorm e_I(s)\tnorm_{\mathrm{DG}}\|e_h(s)\|_{\mathrm{DG}}\mathrm{d}s \\ + & \int_{0}^t  a(\phi_\infty M_c+\phi_\infty M_I+3M_I\,M_c+\theta_\infty)\|e_h(s)\|\|e_I(s)\| \mathrm{d}s + \int_0^t C_m\|\dot{e}_I\|\|e_h\| \mathrm{d}s
\end{aligned}
\end{equation*}
By assumption, we need $\mu-a\,\chi_m\,(M_I\,C_{E_2}+C_S\,C_{E_4})\phi_\infty-a\,\chi_m\,C_{E_2}\theta_\infty > 0$ then exploiting the energy norm defined in \eqref{eq::energy_norm} we get the following estimate, where we can neglect the constants dependencies by means of $\lesssim$:
\begin{equation*}
\begin{aligned}
     \tnorm e_h(t)\tnorm^2_\epsilon \lesssim & \int_0^t  \left(\|\dot{e}_I(s)\|^2 + \|e_I(s)\|^2 + \tnorm e_I(s)\tnorm_{\mathrm{DG}}^2 + \|e_I(s)\|^{4}_{L^4}\right)\mathrm{d}s + \int_0^t \|e_h(s)\|^2\mathrm{d}s.
\end{aligned}
\end{equation*}
By application of H\"older's inequality and of Gr\"onwall's lemma, we obtain: 
\begin{equation*}
\begin{aligned}
         &\tnorm e_h(t)\tnorm_{\epsilon}^2 \lesssim \int_0^t \left(\|\dot{e}_I(s)\|^2 + \|e_I(s)\|^2 + \tnorm e_I(s)\tnorm_{\mathrm{DG}}^2 + \|e_I(s)\|^{4}_{L^4}\right)\mathrm{d}s.
\end{aligned}
\end{equation*}
Using the interpolation bounds of Proposition \ref{as:estimate}, we find:
\begin{equation*}
\begin{aligned}
         &\tnorm e_h(t)\tnorm_{\epsilon}^2 \lesssim \sum_{K\in\partition} \left(h_K^{2\min\{p+1,n\}-2}\int_0^t \left(\|\dot{u}(s)\|^2_{H^n(K)} + \|u(s)\|^2_{H^n(K)}\right)\mathrm{d}s + h_K^{4\min\{p+1,n\}-4}\int_0^t \|u(s)\|^4_{H^n(K)}\mathrm{d}s\right).
\end{aligned}
\end{equation*}
Finally, we use the triangular inequality to estimate the error $\tnorm u-u_h\tnorm_\epsilon \leq \tnorm e_h\tnorm_\epsilon + \tnorm e_I \tnorm_\epsilon$ and the thesis follows.

\end{proof}

\section{Numerical results}
\label{sec:9}
\noindent In this section, we numerically verify the convergence of the monodomain model \eqref{eq:monodomain} with ionic current defined by \eqref{eq::F_analitical}, exploiting the analytical solution described in Section \ref{sec:2}.  We consider a square domain $\Omega = (-3,3)^2$, discretized with a mesh constructed by PolyMesher \cite{talischi2012polymesher} and we numerically approximate the energy norm defined in \eqref{eq::energy_norm}. We analyze the configuration defined by the parameter values reported in Table \ref{table::analitical}, taken from \cite{pezzuto2016space}.

\begin{figure}[h]
    \centering
   \begin{subfigure}{0.48\textwidth}
   \centering
    \resizebox{\textwidth}{!}{
\begin{tikzpicture}{
\begin{axis}[%
width=3.875in,
height=2in,
at={(1.733in,0.687in)},
scale only axis,
xmode=log,
xmin=0.08,
xmax=1.3,
xminorticks=true,
xlabel = { $h$ [-]},
ylabel = { $||V_m(T)-V_m^h(T)||_\varepsilon$/$||V_m(T)||_\varepsilon$},
ymode=log,
ymin=0,
ymax=0.1,
yminorticks=true,
axis background/.style={fill=white},
title={\color{black}  Energy norm convergence in $h$},
xmajorgrids,
xminorgrids,
ymajorgrids,
yminorgrids,
legend style={legend cell align=left, align=left, draw=white!15!black}
]
       			      			
\addplot [color=red, line width=2.0pt]
  table[row sep=crcr]{%
    1.1994  0.0550226964217769 \\
    0.6169  0.0207642270108950 \\
    0.3507  0.00731352579107177\\
    0.1554  0.00183351969564353\\
    0.1003  0.000962615462632588\\
};
\addlegendentry{$p=1$}
			
\addplot [color=orange, line width=2.0pt]
  table[row sep=crcr]{%
    1.1994  0.0208 \\
    0.6169  0.0052 \\
    0.3507  0.0012 \\
    0.1554  1.6695e-04 \\
    0.1003  5.8227e-05 \\
};
\addlegendentry{$p=2$}
	    		
\addplot [color=green, line width=2.0pt]
  table[row sep=crcr]{%
    1.1994  0.0103  \\
    0.6169  0.0017 \\
    0.3507  2.4127e-04  \\
    0.1554  1.5639e-05  \\
    0.1003  3.9343e-06  \\
};
\addlegendentry{$p=3$}      
    
\addplot [color=blue, line width=2.0pt]
  table[row sep=crcr]{%
    1.1994  0.0034      \\
    0.6169  3.9173e-04  \\
    0.3507  2.6561e-05  \\
    0.1554  7.1122e-07  \\
    0.1003  2.1038e-07  \\  
};
\addlegendentry{$p=4$}

\node[right, align=left, text=black, font=\footnotesize]
at (axis cs:0.1805,0.00075) {$1$};

\addplot [color=black, line width=1.5pt]
  table[row sep=crcr]{%
0.180   0.00100\\
0.130   0.00066\\
0.180   0.00066\\
0.180   0.00100\\
};

\node[right, align=left, text=black, font=\footnotesize]
at (axis cs:0.1805,0.00006) {$2$};

\addplot [color=black, line width=1.5pt]
  table[row sep=crcr]{%
0.180   0.0000800\\
0.130   0.0000452\\
0.180   0.0000452\\
0.180   0.0000800\\
};

\node[right, align=left, text=black, font=\footnotesize]
at (axis cs:0.1805,4e-6) {$3$};

\addplot [color=black, line width=1.5pt]
  table[row sep=crcr]{%
0.180   0.0000060\\
0.130   0.0000025\\
0.180   0.0000025\\
0.180   0.0000060\\
};

\node[right, align=left, text=black, font=\footnotesize]
at (axis cs:0.1805,2e-7) {$4$};
 
\addplot [color=black, line width=1.5pt]
  table[row sep=crcr]{%
0.180   4.00e-07\\
0.130   1.27e-07\\
0.180   1.27e-07\\
0.180   4.00e-07\\
};

\end{axis}}
\end{tikzpicture}}
\end{subfigure}
\begin{subfigure}[b]{0.48\textwidth}
    \resizebox{\textwidth}{!}{
\begin{tikzpicture}{
\begin{axis}[%
width=3.875in,
height=2in,
at={(1.733in,0.687in)},
scale only axis,
xmin=0.5,
xmax=8.5,
xminorticks=true,
xlabel = { $p$ [-]},
ylabel = { $||V_m(T)-V_m^h(T)||_\varepsilon/||V_m(T)||_\varepsilon$},
%log x ticks with fixed point,
ymode=log,
ymin=0,
ymax=0.1,
yminorticks=true,
axis background/.style={fill=white},
title={\color{black}  Energy norm convergence in $p$},
xmajorgrids,
xminorgrids,
ymajorgrids,
yminorgrids,
legend style={legend cell align=left, align=left, draw=white!15!black}
]

\addplot [color=green, line width=2.0pt]
  table[row sep=crcr]{%
     1  0.0235\\
     2  0.0064 \\
     3  0.0019 \\
     4  7.3366e-04 \\
     5  2.1887e-04 \\
     6  6.8778e-05 \\
     7  2.0322e-05 \\
     8  6.1048e-06\\
};
\addlegendentry{$\|\cdot\|_\varepsilon$}        
\end{axis}}
\end{tikzpicture}}    
\end{subfigure}%
\caption{Computed errors and convergence rates as a function of $h$ (left) and $p$ (right).}
    \label{fig:convergence_h}
\end{figure}
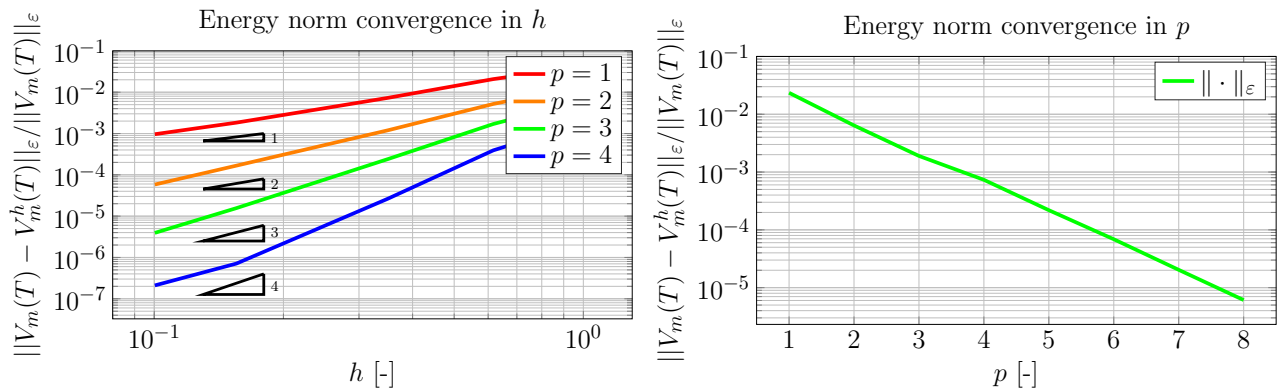

In Figure \ref{fig:convergence_h} we report the computed relative errors in the energy norm at the final time $T=1e-4$ with $\Delta t = 1e-6$. 
The errors was calculated keeping fixed the polynomial order of the space approximation and using different mesh refinements ($N_{el} = 80,300,950,2000,12000$).   

\noindent We observe that the theoretical rates of convergence are achieved for all the polynomial degrees $p$; indeed, the rate of convergence equals the degree of approximation, as proved in Theorem 2. We also perform a convergence analysis concerning the polynomial order $p$ with a mesh of 300 elements. The results are reported in Figure \ref{fig:convergence_h}, where we observe exponential convergence with the energy norm defined in \eqref{eq::energy_norm}. 

\begin{table}[h]
\centering
\caption{Values of the model parameters used in the convergence analysis\label{table::analitical}}%
\begin{tabular}{lll}  
\toprule
 \textbf{Parameters} & \textbf{Values} &  \textbf{Unit}\\
\midrule
    $\sigma_n$ &  $0.17$ & $\mathrm{mS\cdot mm^{-1}}$  \\
    $\sigma_t$ &  $0.62$ & $\mathrm{mS\cdot mm^{-1}}$  \\
    $V_{\text{depol}}$ &  $30$ & $\mathrm{mV}$  \\
    $V_{\text{rest}}$ &  $-85$ & $\mathrm{mV}$  \\
    $V_{\text{thres}}$ &  $-57.6$ & $\mathrm{mV}$  \\
    $a$ &  $1.4e-5$ & $\mathrm{mS \cdot mm^{-2} \cdot mV^{-2}}$  \\
    $\mathrm{\chi_m}$ &  $140$ & $\mathrm{mm^{-1}}$  \\
    $\mathrm{C_m}$ &  $0.01$ & $\mathrm{\mu F \cdot mm^{-2}}$  \\
    $c$ &  $0.5$ &  $\mathrm{mm \cdot ms^{-1}}$  \\
    $\epsilon$ &  $0.2$ & $\mathrm{mm}$  \\    
\bottomrule
\end{tabular}
\label{tab:analytical}
\end{table}

\subsubsection*{Approximation of the conduction speed}
\noindent Next, we investigate how modifying the element size and increasing the polynomial order impact conduction velocity within the context of the PolyDG method impacts the quality of the approximate solutions. 
In Figure \ref{fig:convergence_cv} (left), we observe the influence of varying the element size while maintaining a constant polynomial order on the conduction velocity. A clearer perspective on this relationship can be found in Figure \ref{fig:convergence_cv} (center), where the element size is constant while the polynomial order increases. When working with an equivalent number of degrees of freedom, employing higher-order elements instead of smaller lower-order elements yields a conduction velocity approximation that aligns more closely with the exact solution, as demonstrated in Figure \ref{fig:convergence_cv} (right).  

\begin{figure}[h!]
\centering
\begin{subfigure}[b]{0.33\textwidth}
    \resizebox{\textwidth}{!}{
\begin{tikzpicture}{
\begin{axis}[%
width=3.775in,
height=3in,
at={(1.733in,0.687in)},
scale only axis,
xmin=0.2,
xmax=2.1,
xminorticks=true,
xlabel = {\Large $h$ [-]},
ylabel = {\Large $c_v$ [mm/ms]},
%log x ticks with fixed point,
ymin=0.4,
ymax=1.3,
yminorticks=true,
axis background/.style={fill=white},
title={\color{black}\Large Velocity approximation changing $h$},
xmajorgrids,
xminorgrids,
ymajorgrids,
yminorgrids,
legend style={legend cell align=left, align=left, draw=white!15!black}
]
\addplot [color=green, mark=square,mark size=3pt,line width=2.0pt]
  table[row sep=crcr]{%
    1.9169  1.13591     \\ % 3.05284 - 1.91693
    1.1994  0.81615     \\ %2.6369 - 1.82075
    0.6169  0.61173     \\ % 2.40648 - 1.79511
    0.3507  0.52782     \\ % 2.30762 - 1.7798 
};
\addlegendentry{$p1$}
\addplot [color=blue, mark=triangle,mark size=3pt,line width=2.0pt]
  table[row sep=crcr]{%
    1.9169  0.86086     \\ % 2.72641 - 1.86555
    1.1994  0.64204     \\ % 1.7502 -  2.39224
    0.6169  0.50942     \\ % 1.78808 - 2.2975
    0.3507  0.5082     \\ %2.23894 - 1.74674
};
\addlegendentry{$p2$}
\addplot [color=red,mark=10-pointed star,mark size=3pt, line width=2.0pt]
  table[row sep=crcr]{%
    1.9169  0.72583     \\ % 2.4891 - 1.76327
    1.1994  0.54317     \\ % 2.30731 - 1.76414
    0.6169  0.50319     \\ % 2.23788 - 1.74469
    0.3507  0.49936     \\ % 1.74538 - 2.23674
};
\addlegendentry{$p3$}
\addplot [color=orange, mark=diamond,mark size=3pt,line width=2.0pt]
  table[row sep=crcr]{%
    1.9169  0.54167     \\ % 1.81576 - 2.35743
    1.1994  0.50835    \\ % 1.75695 - 2.2623
    0.6169  0.50147     \\ % 2.24523 - 1.73976
    0.3507  0.4990     \\ % 2.23768 - 1.74372
};
\addlegendentry{$p4$}
\addplot [color=purple, mark=o,mark size=3pt,line width=2.0pt]
  table[row sep=crcr]{%
    1.9169  0.50847     \\ 
    1.1994  0.50273    \\ % 1.75695 - 2.2623
    0.6169  0.50147     \\ % 2.24523 - 1.73976
    0.3507  0.4990     \\ % 2.23768 - 1.74372
};
\addlegendentry{$p5$}
\addplot [color=purple, mark=asterisk,mark size=3pt,line width=2.0pt]
  table[row sep=crcr]{%
    1.9169   0.4953     \\ 
    1.1994   0.4981    \\ 
    0.6169  0.50147     \\ 
    0.3507  0.499379     \\ 
};
\addlegendentry{$p6$}
\draw[dashed,black, line width=2.0pt] (0,0.5) -- (4,0.5);      
\end{axis}}
\end{tikzpicture}}    
\end{subfigure}\hfill
\begin{subfigure}[b]{0.33\textwidth}
    \resizebox{\textwidth}{!}{
\begin{tikzpicture}{
\begin{axis}[%
width=3.775in,
height=3in,
at={(1.733in,0.687in)},
scale only axis,
xmin=0.2,
xmax=6.5,
xminorticks=true,
xlabel = {\Large $h$ [-]},
ylabel = {\Large $c_v$ [mm/ms]},
%log x ticks with fixed point,
ymin=0.4,
ymax=1.3,
yminorticks=true,
axis background/.style={fill=white},
title={\color{black}\Large Velocity approximation changing $p$},
xmajorgrids,
xminorgrids,
ymajorgrids,
yminorgrids,
legend style={legend cell align=left, align=left, draw=white!15!black}
]
\addplot [color=red,mark=square, line width=2.0pt]
  table[row sep=crcr]{%
    1  1.13591     \\ % 3.05284 - 1.91693
    2  0.86086     \\
    3  0.72583     \\ % 2.40648 - 1.79511
    4  0.54167     \\ % 2.30762 - 1.7798
    5  0.50847     \\ % 1.7694 - 2.2779
    6  0.4953     \\ % 1.7694 - 2.2779
};
\addlegendentry{$h=1.91$}
\addplot [color=blue, mark=triangle,mark size=3pt,line width=2.0pt]
  table[row sep=crcr]{%
    1      0.81615 \\
    2      0.64204 \\
    3      0.54317 \\
    4      0.50835 \\
    5      0.50273 \\
    6      0.4981  \\
};
\addlegendentry{$h=1.19$}
\addplot [color=orange,mark=10-pointed star,mark size=3pt, line width=2.0pt]
  table[row sep=crcr]{%
    1      0.61173 \\
    2      0.50942 \\
    3      0.50319 \\
    4      0.50147 \\
    5      0.50147 \\
    6      0.4983  \\
};
\addlegendentry{$h=0.61$}
\addplot [color=green, mark=diamond,mark size=3pt,line width=2.0pt]
  table[row sep=crcr]{%
    1       0.52782 \\
    2      0.5082   \\
    3      0.49936  \\
    4      0.4990   \\
    5      0.499379 \\
};
\addlegendentry{$h=0.35$}
\draw[dashed,black, line width=2.0pt] (0,0.5) -- (6.5,0.5); 
\end{axis}}
\end{tikzpicture}}
\end{subfigure}\hfill
\begin{subfigure}[b]{0.33\textwidth}
    \resizebox{\textwidth}{!}{
\begin{tikzpicture}{
\begin{axis}[%
width=3.775in,
height=3in,
at={(1.733in,0.687in)},
scale only axis,
xmin=-1000,
xmax=21000,
xminorticks=true,
xlabel = {\Large ndof [-]},
ylabel = {\Large $c_v$ [mm/ms]},
%log x ticks with fixed point,
ymin=0.4,
ymax=1.3,
yminorticks=true,
axis background/.style={fill=white},
title={\color{black}\Large Velocity approximation along ndof},
xmajorgrids,
xminorgrids,
ymajorgrids,
yminorgrids,
legend style={legend cell align=left, align=left, draw=white!15!black}
]
\addplot [color=red, mark=square, line width=2.0pt]
  table[row sep=crcr]{%
   90    1.13591 \\
   180   0.86086 \\
   300   0.72583 \\
   450   0.54167 \\
   630   0.50847 \\
   840   0.4953  \\
};
\addlegendentry{$h=1.91$}  
\addplot [color=blue, mark=triangle,mark size=3pt, line width=2.0pt]
  table[row sep=crcr]{%
   240   0.81615 \\
   480   0.64204 \\
   800   0.54317 \\
   1200  0.50535 \\
   1680  0.50273 \\
   2240  0.4981  \\
};
\addlegendentry{$h=1.19$}   
\addplot [color=orange, mark=10-pointed star,mark size=3pt,line width=2.0pt]
  table[row sep=crcr]{%
    900   0.61173\\
   1800   0.50942\\
   3000   0.50319\\
   4500   0.50147\\
   6300   0.49865\\
   8400   0.4983\\
};
\addlegendentry{$h=0.61$} 
\addplot [color=green, mark=diamond,mark size=3pt, line width=2.0pt]
  table[row sep=crcr]{%
   2850     0.52782 \\
   5700     0.5082  \\
   9500     0.49936 \\
   14250    0.4990  \\
   19950    0.499379  \\%26600     0.4990  \\
};
\addlegendentry{$h=0.35$} 
\draw[dashed,black, line width=2.0pt] (-1000,0.5) -- (22000,0.5); 
\end{axis}}
\end{tikzpicture}}    
\end{subfigure}
\caption{Approximation of the velocity of numerical solution as a function of the mesh size $h$ (left), the polynomial degree $p$ (center), and the total number of dofs (right).}
\label{fig:convergence_cv}
\end{figure}
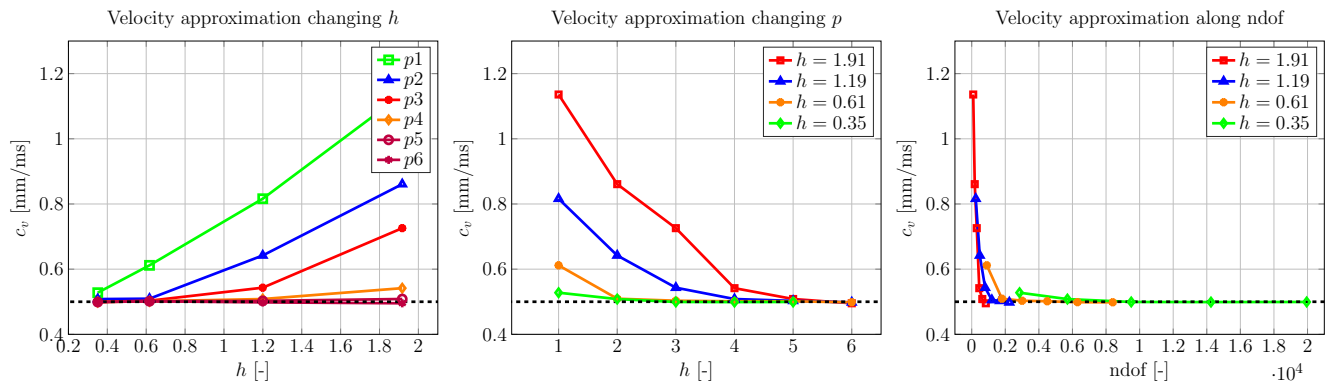

\subsubsection*{Dissipation analysis}
\noindent Using the same parameter setting described in Table \ref{tab:analytical}, we then perform an analysis concerning the undershoots and overshoots of the numerical solution with respect to the expected thresholds of the potential, which are $V_{\text{depol}} = 30 \mathrm{mV}$ and $V_{\text{thres}} = -85\mathrm{mV}$, respectively. Each plot shows the relative value of the undershoots and overshoots of the solution at a specific time in the whole domain for different values of $h$ and $p$. Specifically, in Figure \ref{fig:oscillations}, we have the evolution of the numerical oscillations as the polynomial degree varies. Figure \ref{fig:oscillations_h} highlights that the oscillations in the solution are highly associated with a moderate/high values of $h$. 

These numerical results illustrate how discretizations that increase the polynomial degree are less dissipative than those that refine the discretization.
In particular, we observe that numerical oscillations are significantly reduced by refining in $h$ but at a greater computational cost. On the other hand, increasing in $p$ implies a smaller increase in the degrees of freedom with a similar improvement in accuracy. As an example, if we consider the case $h=1.19$ and $p=2$ ($540$ dofs), we have an overshoot of $7.99\%$ and an undershoot of $4.32\%$, while $h=0.61$ and $p=1$ ($900$ dofs), we have an overshoot of $10.27\%$ and undershoot of $6.74\%$.
Moreover, if we consider the case $h=1.19$ and $p=3$ ($900$ dofs), we have an overshoot of $2.35\%$ and an undershoot of $2\%$, while $h=0.61$ and $p=2$ ($1800$ dofs), we have an overshoot of $4.5\%$ and undershoot of $2.8\%$.

\begin{figure}[h!]
\begin{subfigure}[b]{0.23\textwidth}
\centering
    \resizebox{\textwidth}{!}{
\begin{tikzpicture}{
\begin{axis}[%
width=2.975in,
height=3in,
at={(1.733in,0.687in)},
scale only axis,
xmin=0.2,
xmax=6.5,
xminorticks=true,
xlabel = {\Large $p$ [-]},
ylabel = {\Large $\%$},
%log x ticks with fixed point,
ymin=-25,
ymax=40,
yminorticks=true,
axis background/.style={fill=white},
title={\color{black}\Large $h=1.91$},
xmajorgrids,
xminorgrids,
ymajorgrids,
yminorgrids,
legend style={legend cell align=left, align=left, draw=white!15!black}
]
\addplot [color=red, mark=square,line width=2.0pt]
  table[row sep=crcr]{%
    1  36.4133    \\ %40.924
    2  18.9167    \\ %35.675 
    3  10.1233    \\ %33.037
    4  9.9333     \\ %32.08
    5  5.32       \\ %31.598
    6  4.5        \\ %31.352
};
\addlegendentry{$max$}
\addplot [color=blue, mark=square,line width=2.0pt]
  table[row sep=crcr]{%
    1  -17.5294   \\ %-99.9
    2  -10.7153   \\ %-94.108 
    3  -8.2424    \\ % -92.006
    4  -5.8200    \\ % -89.947
    5  -2.78      \\ %-87.371
    6  -1.65      \\%-86.413
};
\addlegendentry{$min$}    
\end{axis}}
\end{tikzpicture}}    
\end{subfigure}\hfill
\begin{subfigure}[b]{0.23\textwidth}
\centering
    \resizebox{\textwidth}{!}{
\begin{tikzpicture}{
\begin{axis}[%
width=2.975in,
height=3in,
at={(1.733in,0.687in)},
scale only axis,
xmin=0.2,
xmax=6.5,
xminorticks=true,
xlabel = {\Large $p$ [-]},
ylabel = {\Large $\%$},
%log x ticks with fixed point,
ymin=-25,
ymax=40,
yminorticks=true,
axis background/.style={fill=white},
title={\color{black}\Large $h=1.19$},
xmajorgrids,
xminorgrids,
ymajorgrids,
yminorgrids,
legend style={legend cell align=left, align=left, draw=white!15!black}
]
\addplot [color=red, mark=square,line width=2.0pt]
  table[row sep=crcr]{%
    1  30.7     \\ %39.21
    2  7.99     \\ %32.397 
    3  2.3567   \\ %30.707
    4  1.5567   \\ %30.467
    5  1.103    \\ %30.331
    6  0.08     \\ %30.024
};
\addlegendentry{$max$}
\addplot [color=blue, mark=square,line width=2.0pt]
  table[row sep=crcr]{%
    1  -14.5871  \\ %-97.399
    2  -4.3247   \\ %-88.676 
    3  -2.0071   \\ %-86.706
    4  -0.5600   \\ %-85.476
    5  -0.37     \\%-85.32
    6  -0.0658     \\%-85.056
};
\addlegendentry{$min$}    
\end{axis}}
\end{tikzpicture}}    
\end{subfigure}\hfill
\begin{subfigure}[b]{0.23\textwidth}
\centering
    \resizebox{\textwidth}{!}{
\begin{tikzpicture}{
\begin{axis}[%
width=2.975in,
height=3in,
at={(1.733in,0.687in)},
scale only axis,
xmin=0.2,
xmax=6.5,
xminorticks=true,
xlabel = {\Large $p$ [-]},
ylabel = {\Large $\%$},
%log x ticks with fixed point,
ymin=-25,
ymax=40,
yminorticks=true,
axis background/.style={fill=white},
title={\color{black}\Large $h=0.61$},
xmajorgrids,
xminorgrids,
ymajorgrids,
yminorgrids,
legend style={legend cell align=left, align=left, draw=white!15!black}
]
\addplot [color=red, mark=square,line width=2.0pt]
  table[row sep=crcr]{%
    1  10.2767  \\ % 33.083
    2  4.5067   \\ % 31.352 
    3  1.3500   \\ %30.405
    4  0.0400   \\ %30.012
    5  0.003    \\ %30.001
    6  0        \\ 
};
\addlegendentry{$max$}
\addplot [color=blue, mark=square,line width=2.0pt]
  table[row sep=crcr]{%
    1  -6.7424    \\ %-90.731
    2  -2.8447    \\ %-87.418
    3  -0.1271    \\ %-85.108
    4  -0.0659    \\  %-85.056
    5  -0.002     \\ %-85.002
    6  0          \\ 
};
\addlegendentry{$min$}    
\end{axis}}
\end{tikzpicture}}    
\end{subfigure}\hfill
\begin{subfigure}[b]{0.23\textwidth}
\centering
    \resizebox{\textwidth}{!}{
\begin{tikzpicture}{
\begin{axis}[%
width=2.975in,
height=3in,
at={(1.733in,0.687in)},
scale only axis,
xmin=0.2,
xmax=6.5,
xminorticks=true,
xlabel = {\Large $p$ [-]},
ylabel = {\Large $\%$},
%log x ticks with fixed point,
ymin=-25,
ymax=40,
yminorticks=true,
axis background/.style={fill=white},
title={\color{black}\Large $h=0.35$},
xmajorgrids,
xminorgrids,
ymajorgrids,
yminorgrids,
legend style={legend cell align=left, align=left, draw=white!15!black}
]
\addplot [color=red, mark=square,line width=2.0pt]
  table[row sep=crcr]{%
    1  0.4367   \\ % 33.083
    2  0.0433   \\ % 30.013 
    3  0        \\ 
    4  0        \\
    5  0        \\  
    6  0        \\
};
\addlegendentry{$max$}
\addplot [color=blue, mark=square,line width=2.0pt]
  table[row sep=crcr]{%
    1  -0.9918    \\ %-90.731
    2  -0.0141    \\ %-85.012
    3  0          \\ 
    4  0          \\  
    5  0          \\  
    6  0          \\  
};
\addlegendentry{$min$}    
\end{axis}}
\end{tikzpicture}}    
\end{subfigure}%
\caption{Overshoot and undershoot of the numerical solution as a function of the polynomial degree $p$ for different choices of $h = 1.91,1.19,0.61,0.35$.}
\label{fig:oscillations}
\end{figure}
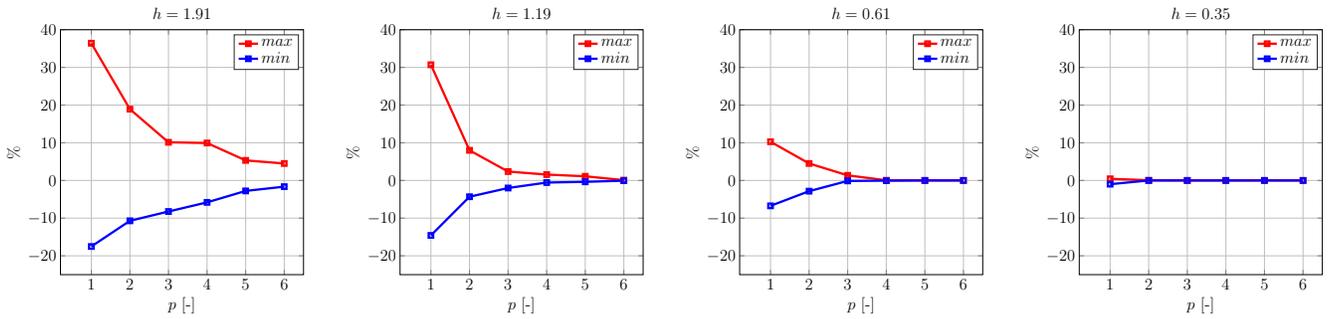
\vspace{-5mm}
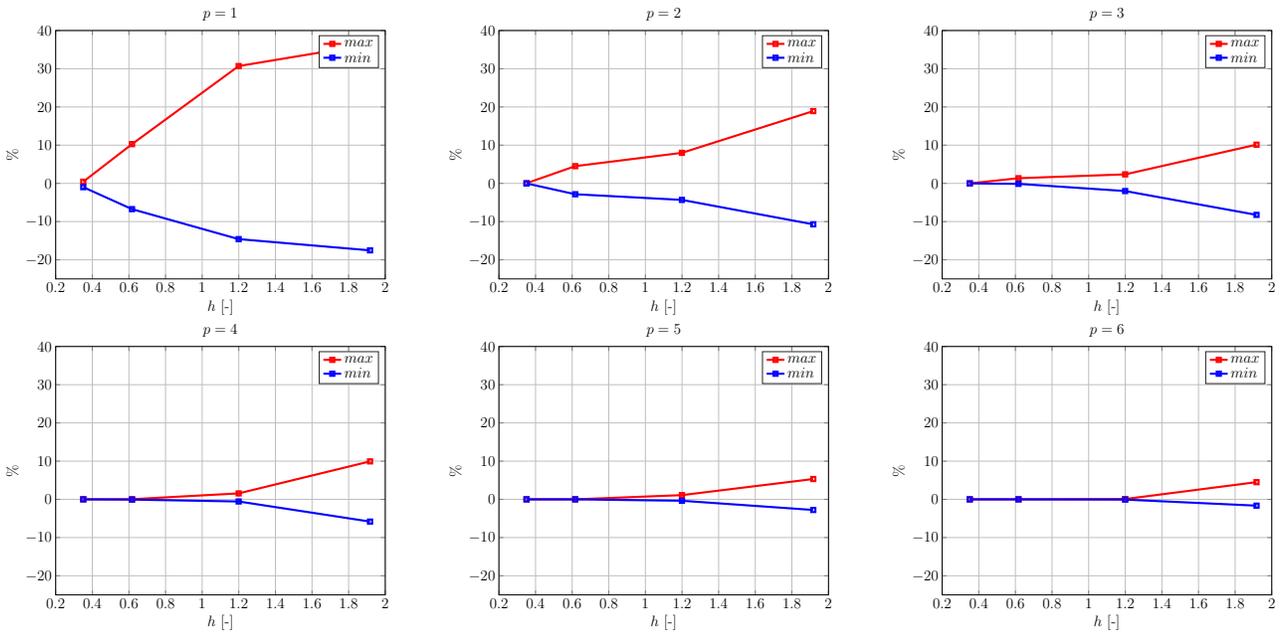
\begin{figure}[h!]
\centering
\begin{subfigure}[b]{0.33\textwidth}
    \resizebox{0.9\textwidth}{!}{
\begin{tikzpicture}{
\begin{axis}[%
width=4.375in,
height=3.3in,
at={(1.733in,0.687in)},
scale only axis,
xmin=0.2,
xmax=2,
xminorticks=true,
xlabel = {$h$ [-]},
ylabel = {$\%$},
%log x ticks with fixed point,
ymin=-25,
ymax=40,
yminorticks=true,
axis background/.style={fill=white},
title={\color{black}\Large $p=1$},
xmajorgrids,
xminorgrids,
ymajorgrids,
yminorgrids,
legend style={legend cell align=left, align=left, draw=white!15!black}
]
\addplot [color=red, mark=square,line width=2.0pt]
  table[row sep=crcr]{%
    1.9169 36.4133 \\
    1.1994 30.7    \\
    0.6169 10.2767 \\
    0.3507 0.4367  \\
};
\addlegendentry{$max$}
\addplot [color=blue, mark=square,line width=2.0pt]
  table[row sep=crcr]{%
    1.9169 -17.5294 \\
    1.1994 -14.5871 \\
    0.6169 -6.7424  \\
    0.3507 -0.9918  \\
};
\addlegendentry{$min$}    
\end{axis}}
\end{tikzpicture}}    
\end{subfigure}\hfill
\begin{subfigure}[b]{0.33\textwidth}
    \resizebox{0.9\textwidth}{!}{
\begin{tikzpicture}{
\begin{axis}[%
width=4.375in,
height=3.3in,
at={(1.733in,0.687in)},
scale only axis,
xmin=0.2,
xmax=2,
xminorticks=true,
xlabel = {$h$ [-]},
ylabel = {$\%$},
%log x ticks with fixed point,
ymin=-25,
ymax=40,
yminorticks=true,
axis background/.style={fill=white},
title={\color{black}\Large $p=2$},
xmajorgrids,
xminorgrids,
ymajorgrids,
yminorgrids,
legend style={legend cell align=left, align=left, draw=white!15!black}
]
\addplot [color=red, mark=square,line width=2.0pt]
  table[row sep=crcr]{%
    1.9169  18.9167  \\
    1.1994  7.99     \\
    0.6169  4.5067   \\
    0.3507  0.0433   \\
};
\addlegendentry{$max$}
\addplot [color=blue, mark=square,line width=2.0pt]
  table[row sep=crcr]{%
    1.9169  -10.7153  \\
    1.1994  -4.3247   \\
    0.6169  -2.8447   \\
    0.3507  -0.0141   \\  
};
\addlegendentry{$min$}    
\end{axis}}
\end{tikzpicture}}    
\end{subfigure}\hfill
\begin{subfigure}[b]{0.33\textwidth}
    \resizebox{0.9\textwidth}{!}{
\begin{tikzpicture}{
\begin{axis}[%
width=4.375in,
height=3.3in,
at={(1.733in,0.687in)},
scale only axis,
xmin=0.2,
xmax=2,
xminorticks=true,
xlabel = {$h$ [-]},
ylabel = {$\%$},
%log x ticks with fixed point,
ymin=-25,
ymax=40,
yminorticks=true,
axis background/.style={fill=white},
title={\color{black}\Large $p=3$},
xmajorgrids,
xminorgrids,
ymajorgrids,
yminorgrids,
legend style={legend cell align=left, align=left, draw=white!15!black}
]
\addplot [color=red, mark=square,line width=2.0pt]
  table[row sep=crcr]{%
    1.9169  10.1233  \\
    1.1994  2.3567   \\
    0.6169  1.3500   \\
    0.3507  0        \\
};
\addlegendentry{$max$}
\addplot [color=blue, mark=square,line width=2.0pt]
  table[row sep=crcr]{%
    1.9169  -8.242    \\
    1.1994  -2.0071   \\
    0.6169  -0.1271   \\
    0.3507  0         \\
};
\addlegendentry{$min$}    
\end{axis}}
\end{tikzpicture}}    
\end{subfigure}\hfill
\begin{subfigure}[b]{0.33\textwidth}
    \resizebox{0.9\textwidth}{!}{
\begin{tikzpicture}{
\begin{axis}[%
width=4.375in,
height=3.3in,
at={(1.733in,0.687in)},
scale only axis,
xmin=0.2,
xmax=2,
xminorticks=true,
xlabel = {$h$ [-]},
ylabel = {$\%$},
%log x ticks with fixed point,
ymin=-25,
ymax=40,
yminorticks=true,
axis background/.style={fill=white},
title={\color{black}\Large $p=4$},
xmajorgrids,
xminorgrids,
ymajorgrids,
yminorgrids,
legend style={legend cell align=left, align=left, draw=white!15!black}
]
\addplot [color=red, mark=square,line width=2.0pt]
  table[row sep=crcr]{%
    1.9169  9.9333  \\
    1.1994  1.5567   \\
    0.6169  0.0400   \\
    0.3507  0        \\
};
\addlegendentry{$max$}
\addplot [color=blue, mark=square,line width=2.0pt]
  table[row sep=crcr]{%
    1.9169  -5.8200   \\
    1.1994  -0.5600   \\
    0.6169  -0.0659   \\
    0.3507  0         \\
};
\addlegendentry{$min$}    
\end{axis}}
\end{tikzpicture}}    
\end{subfigure}\hfill
\begin{subfigure}[b]{0.33\textwidth}
    \resizebox{0.9\textwidth}{!}{
\begin{tikzpicture}{
\begin{axis}[%
width=4.375in,
height=3.3in,
at={(1.733in,0.687in)},
scale only axis,
xmin=0.2,
xmax=2,
xminorticks=true,
xlabel = {$h$ [-]},
ylabel = {$\%$},
%log x ticks with fixed point,
ymin=-25,
ymax=40,
yminorticks=true,
axis background/.style={fill=white},
title={\color{black}\Large $p=5$},
xmajorgrids,
xminorgrids,
ymajorgrids,
yminorgrids,
legend style={legend cell align=left, align=left, draw=white!15!black}
]
\addplot [color=red, mark=square,line width=2.0pt]
  table[row sep=crcr]{%
    1.9169  5.32   \\
    1.1994  1.103   \\
    0.6169  0.003    \\
    0.3507  0        \\
};
\addlegendentry{$max$}
\addplot [color=blue, mark=square,line width=2.0pt]
  table[row sep=crcr]{%
    1.9169  -2.78   \\
    1.1994  -0.37   \\
    0.6169  -0.002   \\
    0.3507  0         \\
};
\addlegendentry{$min$}    
\end{axis}}
\end{tikzpicture}}    
\end{subfigure}\hfill
\begin{subfigure}[b]{0.33\textwidth}
    \resizebox{0.9\textwidth}{!}{
\begin{tikzpicture}{
\begin{axis}[%
width=4.375in,
height=3.3in,
at={(1.733in,0.687in)},
scale only axis,
xmin=0.2,
xmax=2,
xminorticks=true,
xlabel = {$h$ [-]},
ylabel = {$\%$},
%log x ticks with fixed point,
ymin=-25,
ymax=40,
yminorticks=true,
axis background/.style={fill=white},
title={\color{black}\Large $p=6$},
xmajorgrids,
xminorgrids,
ymajorgrids,
yminorgrids,
legend style={legend cell align=left, align=left, draw=white!15!black}
]
\addplot [color=red, mark=square,line width=2.0pt]
  table[row sep=crcr]{%
    1.9169  4.5  \\
    1.1994  0.08   \\
    0.6169  0   \\
    0.3507  0        \\
};
\addlegendentry{$max$}
\addplot [color=blue, mark=square,line width=2.0pt]
  table[row sep=crcr]{%
    1.9169  -1.65   \\
    1.1994  -0.0658   \\
    0.6169  0   \\
    0.3507  0         \\
};
\addlegendentry{$min$}    
\end{axis}}
\end{tikzpicture}}    
\end{subfigure}%
\caption{Relative overshoot and undershoot of the numerical solution as a function of the mesh size $h$ for different polynomial degrees $p=1,...,6$.}
\label{fig:oscillations_h}
\end{figure}

\section{Conclusion}
\label{sec:10}
\noindent In this work, we have presented a theoretical analysis and a numerical investigation of PolyDG discretizations of brain electrophysiology models, namely the monodomain equation coupled with the Barreto-Cressman ionic model. Numerical results demonstrate that the  PolyDG method allows for flexibility in selecting the polynomial degree, enabling high-order accurate approximation of the electric wavefront originating from an unstable region of gray matter. Specifically, increasing the order of the approximation allows for more accurate reconstruction of conduction velocities at a lower cost than spatial refinement.
Furthermore, PolyDG is built upon polygonal agglomerated grids that capture geometric details of complex geometries characterized by highly heterogeneous materials.

We have stability and convergence error estimates for the semi-discrete formulation on a simplified monodomain problem where a non-linear dependence on the transmembrane potential characterizes the reaction term. This numerical verification confirmed the theoretical results of our analysis on polygonal meshes with an explicit treatment of the non-linear term.   
These additional numerical results further highlight how high-order methods enable the approximate solution of electrophysiology problems while maintaining a good balance between computational cost and accuracy.
As future developments of this work, we plan the extension of $p$-PolyDG adaptive schemes to track the wavefront with high accuracy and significantly reduce the computational costs of the simulations, and the construction of space-time DG formulations \cite{antonietti2020space, antonietti2023discontinuous} to achieve higher-order approximations also in time. 

\section*{Declaration of competing interests}
The authors declare that they have no known competing financial interests or personal relationships that could have appeared to influence the work reported in this article.

\section*{Acknowledgments}
The brain MRI images were provided by OASIS-3: Longitudinal Multimodal Neuroimaging: Principal Investigators: T. Benzinger, D. Marcus, J. Morris; NIH P30 AG066444, P50 AG00561, P30 NS09857781, P01 AG026276, P01 AG003991, R01 AG043434, UL1 TR000448, R01 EB009352. AV-45 doses were provided by Avid Radiopharmaceuticals, a wholly-owned subsidiary of Eli Lilly.

\bibliographystyle{hieeetr}
\bibliography{sample.bib}
\end{document}